\documentclass[11pt]{amsart}

\usepackage{amssymb,amsthm,amsfonts,latexsym,mathtools}
\usepackage{amsmath}
\usepackage{mathrsfs}
\usepackage{stmaryrd}
\usepackage{color}
\usepackage{graphicx}
\usepackage{accents}
\usepackage{tikz}
\usepackage{tikz-cd}
\usepackage{pbox}

\usepackage{makecell}

\usepackage{hyperref}

\input{xy}
\xyoption{all}
\xyoption{poly}
\usepackage[all]{xy}
\allowdisplaybreaks

\newcommand*{\doi}[1]{\href{http://dx.doi.org/#1}{doi: #1}}

 \newcommand\FF{{\mathbb{F}}}
 \def\KK{{\mathbb{K}}}
 
 \def\QQ{\mathbb{Q}}

 \newcommand\ZZ{{\mathbb{Z}}}
 
 \def\A{\mathcal{A}}
 \def\B{{\mathcal B}}
 \def\C{{\mathcal C}}
 \def\D{{\mathcal D}}
 
 \def\F{{\mathcal F}}

 \def\L{{\mathcal L}}

 \def\O{{\mathcal O}}
 
 \def\Q{{\mathcal Q}}
 
 \def\S{{\mathcal S}}
 \def\T{{\mathcal T}}
 \def\U{{\mathcal U}}

 \DeclareMathOperator\Lk{Lk}

 \def\op{\mathrm{op}}

 \newcommand{\GF}[1]{\mathbb{F}_{#1}}

 \DeclareMathOperator\GL{GL}

 \DeclareMathOperator\SL{SL}

 \DeclareMathOperator\Rad{Rad}

 \DeclareMathOperator\Aut{Aut}
 \DeclareMathOperator\Hom{Hom}

 \DeclareMathOperator\id{id}
  
 \def\Im{\mathrm{Im}}

 \DeclareMathOperator{\rk}{rk}
 \newcommand{\tq}{\mathrel{{\ensuremath{\: : \: }}}}

 \DeclareMathOperator{\Sub}{{\mathcal{S}}}
 \DeclareMathOperator{\PD}{\mathcal{PD}}
 \DeclareMathOperator{\OD}{\mathcal{OD}}
 
 \DeclareMathOperator{\PF}{\mathcal{PF}}

 \DeclareMathOperator\PB{PB} 
 \DeclareMathOperator\FC{FC} 
 
 \def\groupiso{\cong}

 \newcommand\gen[1]{\left\langle#1\right\rangle}

  \makeatletter
  \newcommand{\numberlike}[2]{%
     \expandafter\def\csname c@#1\endcsname{%
         \expandafter\csname c@#2\endcsname}%
  }
  \makeatother

  \def\DefaultNumberTheoremWithin{section}

  \theoremstyle{plain}
  \newtheorem{lemma}{Lemma}
     \numberwithin{lemma}{\DefaultNumberTheoremWithin}
     \labelformat{lemma}{Lemma~#1}
  
     \numberwithin{claim}{\DefaultNumberTheoremWithin}
     \numberlike{claim}{lemma}
     \labelformat{claim}{Claim~#1}
  \newtheorem{theorem}{Theorem}
     \numberwithin{theorem}{\DefaultNumberTheoremWithin}
     \numberlike{theorem}{lemma}
     \labelformat{theorem}{Theorem~#1}
  \newtheorem{corollary}{Corollary}
     \numberwithin{corollary}{\DefaultNumberTheoremWithin}
     \numberlike{corollary}{lemma}
     \labelformat{corollary}{Corollary~#1}
     
  \newtheorem{proposition}{Proposition}
     \numberwithin{proposition}{\DefaultNumberTheoremWithin}
     \numberlike{proposition}{lemma}
     \labelformat{proposition}{Proposition~#1}
     
  \newtheorem{conjecture}{Conjecture}
     \numberwithin{conjecture}{\DefaultNumberTheoremWithin}
     \numberlike{conjecture}{lemma}
     \labelformat{conjecture}{Conjecture~#1}

  \newtheorem*{warning*}{Warning}

  \theoremstyle{definition}
  \newtheorem{definition}{Definition}
     \numberwithin{definition}{\DefaultNumberTheoremWithin}
     \numberlike{definition}{lemma}
     \labelformat{definition}{Definition~#1}

  \theoremstyle{definition}
  \newtheorem{question}{Question}
     \numberwithin{question}{\DefaultNumberTheoremWithin}
     \numberlike{question}{lemma}
     \labelformat{question}{Question~#1}

  \theoremstyle{definition}
  
     \numberwithin{problem}{\DefaultNumberTheoremWithin}
     \numberlike{problem}{lemma}
     \labelformat{problem}{Problem~#1}

  \theoremstyle{remark}
  \newtheorem{remark}{Remark}
     \numberwithin{remark}{\DefaultNumberTheoremWithin}
     \numberlike{remark}{lemma}
     \labelformat{remark}{Remark~#1}
  \theoremstyle{remark}

  \newtheorem{example}{Example}
     \numberwithin{example}{\DefaultNumberTheoremWithin}
     \numberlike{example}{lemma}
     \labelformat{example}{Example~#1}

\labelformat{equation}{(#1)}
\labelformat{figure}{Figure~#1}
\labelformat{table}{Table~#1}
\labelformat{chapter}{Chapter~#1}
\labelformat{appendix}{Appendix~#1}
\labelformat{section}{Section~#1}
\labelformat{subsection}{Subsection~#1}

\newlength{\dhatheight}

\newcommand{\hookdownarrow}{\mathrel{\rotatebox[origin=c]{-90}{$\hookrightarrow$}}}

\DeclareMathOperator{\Ivec}{IVec}
\DeclareMathOperator{\RIvec}{RIVec}

\DeclareMathOperator{\HT}{HT}
\DeclareMathOperator{\HF}{HF}
\DeclareMathOperator{\CB}{CB} 

\newcommand{\red}[1]{#1^*}
\newcommand{\redm}[1]{#1^\circ}

\newcommand{\Ord}{\mathcal{O}}
\newcommand{\Forget}{F}
\newcommand{\RE}[1]{\tilde{\chi}(#1)}
\newcommand{\Vog}{\mathcal{B}}
\newcommand{\PVog}{\mathcal{PB}}
\DeclareMathOperator{\Freefactors}{FC}

\newcommand{\Charney}{\mathcal{G}}
\newcommand{\sch}{(\sqcup,h)}
\newcommand{\Bergman}{\Delta}

\newcommand{\oPerpSymbol}{\begin{tikzpicture}[scale=0.134]
  \draw (0,-0.5)--(0,1); \draw (-0.866,-0.5)--(0.866,-0.5);
  \draw (0,0) circle [radius=1];
\end{tikzpicture}}
\newcommand{\operp}{\mathbin{\raisebox{-1pt}{\oPerpSymbol}}}

\begin{document}

\title[Posets arising from decompositions of objects in a monoidal category]{Posets arising from decompositions of objects in a monoidal category}

\author{Kevin I. Piterman}
\address{Philipps-Universit{\"a}t Marburg, Fachbereich Mathematik und Informatik, 35032 Marburg, Germany; and Vrije Universiteit Brussel, Department of Mathematics and Data Science, 1050 Brussels, Belgium}

\email{piterman@mathematik.uni-marburg.de; kevin.piterman@vub.be}

\author{Volkmar Welker}
\address{Philipps-Universit\"at Marburg \\
Fachbereich Mathematik und Informatik \\
35032 Marburg, Germany}
\email{welker@mathematik.uni-marburg.de}

\begin{abstract}
Given a symmetric monoidal category $\C$ with product $\sqcup$, where the neutral element for the product is an initial object,
we consider the poset of $\sqcup$-complemented subobjects of a given object $X$. When this poset has finite height, we define decompositions and partial decompositions of $X$ which are coherent with $\sqcup$, and order them by refinement. From these posets, we
define complexes of frames and partial bases, augmented Bergman complexes and 
related ordered versions.
We propose a unified approach to the study of their combinatorics and homotopy type, establishing various properties and relations between them. Via explicit homotopy formulas, we will be able to transfer structural properties, such as
Cohen-Macaulayness.

In well-studied scenarios,
the poset of $\sqcup$-complemented subobjects specializes to the poset of free factors of a free group, the subspace poset of a vector space, the poset of non-degenerate subspaces of a vector space with a non-degenerate form, and the lattice of flats of a matroid. The decomposition and partial decomposition posets, the complex of frames and partial bases together with the ordered versions, either coincide with well-known structures, generalize them, or yield
new interesting objects. In these particular cases, we provide new results along with open questions and conjectures.
\end{abstract}



\maketitle

\tableofcontents

\section{Introduction}

In diverse areas of mathematics, posets and simplicial complexes arising from algebraic, combinatorial or geometric objects are widely studied.
Usually, one tries to answer questions or determine intrinsic properties of a mathematical object by looking into the combinatorics and topology of associated combinatorial structures.
One of the most classical examples is the poset of subspaces of a finite-dimensional vector space, which has been extensively studied in connection with the combinatorics, representation theory, and cohomology of linear groups.
Sometimes one even endeavors to associate several posets or simplicial complexes with a given object, establish relations among them, and then achieve stronger results. In this direction, many works provide resembling constructions and morphisms between the various combinatorial structures, analyzed often via ``Quillen's fiber theorem" type tools. In some cases, these yield wedge/join decompositions, allowing us to translate or interchange combinatorial and homotopy properties from one structure to another.

This paper aims to present a general framework in which many of these constructions are particular cases.
We discuss fundamental combinatorial properties that are intrinsic to our approach and show how they can be used to establish comparison results between the different posets and simplicial complexes arising from the same mathematical object.
These results provide a toolset for verifying homotopical properties such as high connectivity, sphericity or even Cohen-Macaulayness, and also for computing 
Euler characteristics (or, more generally, the module afforded in the homology).
These are goals prevalent in the literature.

The starting point of our constructions is the subobject poset of a given object $X$ in some category $\C$.
However, from a combinatorial point of view, this poset is often not suitable since it may contain infinite chains, as is, for example, the case of the subgroup poset of a free group, or even the subgroup poset of a free Abelian group.
Infinite chains obstruct the construction of ``natural" posets of decompositions and frame complexes associated with $X$. 
For that reason, we will keep only the subobjects that are complemented with respect to a suitable product. Thus, we shift gears and work instead with an initial category $\C$ equipped with a symmetric monoidal product $\sqcup$, which will often be the coproduct of $\C$.
With this additional structure at hand, we can extract more ``interesting" subobject posets and decompositions that encode different ways of resembling the original object from sets 
or tuples of subobjects whose $\sqcup$-product is isomorphic to $X$.

To be more precise, the constructions in this article are built up from a so-called \textit{poset of complemented subobjects} $\Sub(X,\sqcup)$ of a fixed object $X$ in our initial category $\C$ with symmetric monoidal product $\sqcup$.
A prominent example to keep in mind is the above-mentioned poset of subspaces $\S(V)$ of a 
vector space $V$ of dimension $n$ over a field $\KK$.
In this case, we consider the category of vector spaces over $\KK$ with $\sqcup$ being just the usual coproduct.
The order complex of $\S(V)$ is the Tits building of $\SL_n(\KK)$, and it has been widely studied in many areas of mathematics.
Then we define the posets of \textit{$\sqcup$-decompositions} and \textit{partial $\sqcup$-decompositions} (see \ref{def:decompositionsCategory}).
In our example case, the former corresponds to the poset $\D(V)$ of collections of non-zero subspaces $\{S_1,\ldots, S_r\}$ which are in internal direct sum and span the whole space $V$, that is, $S_1\oplus\cdots  \oplus S_r \cong \gen{S_1,\ldots,S_r} = V$.
The latter, denoted by $\PD(V)$, consists of collections of non-zero subspaces in direct sum, the span being just a subspace.
Both posets are ordered by refinement.
Direct sum decompositions of vector spaces over finite fields were shown to be homotopically Cohen-Macaulay in \cite{Welker}. Indeed, the same result for arbitrary fields is
already implicit in \cite{Charney}, as we explain later in the article {(see \ref{prop:charney} and \ref{sub:vectorspaces})}. 
For partial direct sum decompositions of vector spaces over finite fields, Cohen-Macaulayness was established in the unpublished work \cite{HHS}.
For infinite fields, the proper part of $\PD(V)$ is spherical by results of \cite{BPW,MPW}, as we explain in \ref{sub:vectorspaces}.
However, Cohen-Macaulayness has not been established yet in the case of infinite fields.
Then we define a simplicial complex of \textit{frames}. In the vector space case, a frame is a partial decomposition into $1$-dimensional subspaces.
A further construction is obtained as the \textit{inflation} of the frame complex.
That is, we replace vertices of the frame complex with sets (interpreted as sets of generators or bases).
This yields the \textit{complex of partial bases}.
From subobject posets and frame complexes, one can also build the \textit{augmented Bergman complex} -- which in the matroid case recently gained a lot of attention through \cite{BHMPW}.

For (partial) decompositions, frames and partial bases, there also exist ordered versions.
These are posets whose elements are all linear orders on sets and simplices in the unordered
poset or simplicial complex (see \ref{def:orderedVersions}).
Our key result regarding the ordered versions is that they exhibit a nice wedge decomposition in terms of their unordered counterparts. Therefore, properties such as sphericity or Cohen-Macaulayness can be transferred back and forth. 
In our vector space example, the ordered partial decompositions are linearly ordered sets of subspaces that form internal direct sums, and ordered decompositions are in close relationship to a complex studied in \cite{Charney}, which in our setting we call the \textit{Charney complex}.
The ordered versions of the frame complex and the partial basis complex essentially correspond to the associated complex of injective words (see \cite{JW,ReW}).

Our results about properties and connections among these different posets and complexes allow us to specialize them to important cases.
We summarize the main examples to which we apply our theory in \ref{tab:summaryExamples}.
In the row of this table corresponding to vector spaces with forms, we take the category of vector spaces with \textit{suitable} forms, and morphisms are isometries.
Even though this category misses most of the coproducts, the natural candidate for it is the orthogonal sum, and that is the monoidal product we consider.
Also observe that in the matroid case, there is no categorical approach developed so far that suits our purposes.
The problem relies on the fact that the main poset we want to study is the lattice of flats, but the matroid obtained from the join of two flats depends specifically on the matroid we are working with.
Hence, in that case, we regard the lattice of flats as a poset-category, and the monoidal product is just the join.

Finally, we would like to emphasize that our results are based on combinatorial and topological methods that apply mainly to finite-dimensional complexes and posets.

\renewcommand{\arraystretch}{2.0}
\begin{table}[ht]
    \centering
    \begin{tabular}{|c|c|c|c|}
        \hline 
Object & Category & \makecell{Monoidal \\ product} & $\Sub(X,\sqcup)$ \\
        \hline
        \pbox{3.5cm}{\centering Free group of finite rank} & Groups & Free product & \pbox{3cm}{\centering Poset of free factors}\\  [0.5cm]
        \pbox{3.5cm}{\centering Finite-dimensional $\KK$-vector space} & $\KK$-vector spaces & Direct sum & \pbox{3cm}{\centering Poset of subspaces}\\  [0.5cm]
        \pbox{3.5cm}{\centering Finite-dimensional $\KK$-vector space with a non-degenerate form} & \pbox{3.5cm}{\centering $\KK$-vector spaces with forms and isometries} & Orthogonal sum & \pbox{3cm}{\centering Non-degenerate subspace poset}\\  [0.8cm]
        Matroid & \pbox{3.2cm}{\centering Lattice of flats
        (of the matroid)} & Join & Lattice of flats\\  [0.5cm]
        \pbox{3.5cm}{\centering Free module of finite rank over a Dedekind domain $\O$} & (Free) $\O$-modules & Direct sum & \pbox{3cm}{\centering Poset of summands}\\  [0.5cm]
        \hline
    \end{tabular} \medskip
    \caption{Main examples of the paper}
    \label{tab:summaryExamples}
\end{table}

A related categorical approach is developed in \cite{GKRW, RW, RWW}, where certain complexes are constructed from complemented subobjects.
While our focus is the analysis of the combinatorics and the homotopy types of our constructions, these works address \textit{homological stability} problems for quite general families of groups by starting out with high connectivity assumptions. 
Thus, in a certain sense, our work complements \cite{GKRW, RW, RWW} by establishing a set of posets and simplicial complexes together with properties and relations among them that can be used to prove Cohen-Macaulayness or high connectivity.
We expect that our results will yield new classes of complexes, which then can be fed into the sophisticated machinery devised in these works.

A different but related problem where our results may provide new input data is the search for \textit{dualizing modules} (see e.g., \cite{CP,CFP}). If $V$ is an $n$-dimensional vector space over $\KK$, the order complex of the proper part of the poset of subspaces of $V$ is the Tits building of $\SL_n(\KK)$.
The top-dimensional homology group of this complex is the well-known \textit{Steinberg module} of $\SL_n(\KK)$, and by Borel-Serre \cite{BorelSerre} it is a rational dualizing module for $\SL_n(\O)$, when $\KK$ is a number field and $\O$ is its ring of integers.
The explicit determination of a dualizing module may be a difficult and often unsolved problem. 
For example, a question raised by Hatcher and Vogtmann \cite{HV1} asked whether the top dimensional cohomology group of the free factor complex $\Freefactors(F_n)$ of the free group $F_n$ of 
rank $n$ yields a dualizing module for $\Aut(F_n)$.
However, a negative answer to this question was recently given in the case $n = 5$ in \cite{HMNP}.
In general, the homology (or even the ``Lefschetz module") of $\Sub(X,\sqcup)$ can be regarded as a Steinberg module for a suitable group of automorphisms of this poset, prompting consideration for a similar interpretation in other contexts.

From the combinatorial perspective, this paper explores a unified approach to many well-studied poset constructions (see e.g., \cite{Sta1,Sta2, Wachs}). It comprises important complexes and posets associated with matroids including 
the lattice of flats, the independence complex and the
recently studied augmented Bergman complex (see e.g., \cite{Bjo92,BHMPW}).
In both cases, new constructions and results are added. These yield results on homotopy types about ``natural" combinatorially defined posets and complexes. When looking at the 
reduced Euler characteristic or equivalently the M{\"o}bius number, 
new enumerative identities follow, which will be addressed in an upcoming work \cite{PSW}.

\medskip

The paper is organized as follows.
In \ref{sec:definitionsposet} we give precise definitions of the framework we want to work with and provide definitions of complemented elements and decompositions arising from subobjects in a certain monoidal category.
In \ref{sec:properties} we establish basic properties of the posets of decompositions and partial decompositions, and in \ref{sec:derivedposets} we define other posets that can also be constructed from decompositions. In particular, we include here the definition of the frame complex, the partial basis complex, their ordered versions, the augmented Bergman complex, and the Charney complex.
In \ref{sec:homotopyposets}, we prove results on the homotopy types of posets and simplicial complexes defined in \ref{sec:definitionsposet} and \ref{sec:derivedposets}. Typical results from this paragraph express the
homotopy type of one class of posets in terms of another class. Through
these formulas for the homotopy types we can define conditions under which 
properties, such as sphericity or Cohen-Macaulayness are transferred.
In general, results from this section provide innovative tools 
to investigate the topology of combinatorially defined simplicial complexes.

Finally, in \ref{sec:examples}, we prove a variety of new results on the homotopy types and Euler characteristic 
of posets and simplicial complexes built in the preceding 
sections for a wide range of combinatorial,
algebraic and geometric objects (e.g., subspaces of vector spaces, non-degenerate subspaces, matroids, set-partitions, etc.). The proofs use a combination of methods developed in  
\ref{sec:derivedposets} \& \ref{sec:homotopyposets} and
additional enumerative results.
Throughout that section, we pose some questions and present open problems.

\medskip

\textbf{Acknowledgments.}
We are very grateful to Benjamin Br\"uck for the many discussions during his research visit to Marburg and for providing us with several motivational examples and references.
We thank Jesper {M\o ller} for pointing out to us that the Solomon identity can be used to simplify the proof of \ref{thm:eulerCharOPDVectorSpaces}
from an earlier version of this paper.
We also thank the anonymous referee for valuable suggestions, which greatly improved the presentation of this article.

The first author is supported by a Postdoctoral fellowship granted by the Alexander von Humboldt Stiftung and the FWO grant 12K1223N.

\bigskip

\section{Poset of decompositions into subobjects}
\label{sec:definitionsposet}

We start our poset constructions with a purely combinatorial approach. We will see that 
this approach, even though vital for subsequent developments, fails to cover important classes of posets that otherwise perfectly fit our setting.
For that reason, we will use category theory language later in this section.

First, we recall some basic notations and definitions.
A poset $\S$ is a (not necessarily finite) set with a partial order $\leq$.
We call $\S$ \textit{bounded} if it has a unique minimal element $0_{\S}$ and a unique 
maximal element $1_{\S}$. When the reference poset $\S$ is clear from the context, we just write $0$ and $1$ for the minimum and maximum of $\S$. 
For a bounded poset $\S$, we denote by $\red{\S} = \S\setminus \{0_{\S},1_{\S}\}$ its proper part.
Also if $\S$ has just a unique maximum element $1_{\S}$, we set $\redm{\S} = \S\setminus \{1_{\S}\}$.
For $x \in \S$ we write $\S_{\leq x}$ for the subposet $\{ \,y \in \S~|~y \leq x\}$. 
Analogously defined are $\S_{< x}$, $\S_{\geq x}$ and $\S_{> x}$. 

The order complex of a poset $\S$ is the simplicial complex $\Delta(\S)$ whose $i$-simplices 
are the finite chains $x_0 < \cdots < x_i$ of $i+1$ distinct elements from $\S$. When we speak of
topological properties of a poset we mean the corresponding property of the geometric realization of its
order complex.
For an element $x$ of $\S$, we let $h(x)$ be the dimension 
of the order complex of $\S_{\leq x}$
and call $h(x)$ the \textit{height} of $x$ in $\S$. The height of $\S$ is the supremum of
the heights of its elements or, equivalently, the dimension of its order complex.
All poset constructions we want to unify in this paper yield posets of finite height.
Nevertheless, during the construction process, we will sometimes have to go through infinite-height posets. 

In a poset $\S$ we say for a subset $\sigma$ that the join of $\sigma$ exists in $\S$ if
$\{ y \in \S~|~y \geq x $ for all $x \in \sigma\}$ has a unique minimal element.
In this case, we write $\bigvee_{x \in \sigma} x$ for the join of $\sigma$. 
Analogously, we say that  
the meet of $\sigma$ exists in $\S$ if $\{ y \in \S~|~y \leq x $ for all $x \in \sigma\}$ has a unique maximal element, and write $\bigwedge_{x \in \sigma} x$ for the meet of $\sigma$. 
Here, we adopt the usual convention that in a bounded poset 
the join over the empty set is $0_{\S}$ and the meet over the empty set is $1_{\S}$.

We denote by $|\sigma|$ the size of a set $\sigma$.

\begin{definition} 
\label{def:decompositionsAndPD}
Let $\S$ be a bounded poset.
We say that a non-empty subset $\sigma \subseteq \S $ is a \textit{(full) decomposition} of $\S$ if
   \begin{enumerate}
       \item for all $\tau \subseteq \sigma$ the join $\bigvee_{x \in \tau} x$ exists in $\S$ and
       \[ h\big(\bigvee_{x\in \tau} x\big) = \sum_{x\in \tau} h(x),\]
       \item $\big\{\,\bigvee_{x \in \tau} x \,\big|\, \tau \subseteq \sigma \,\big\} \subseteq \S$ with suprema
       and infima taken in $\S$ is a Boolean lattice on $|\sigma|$ elements with maximal element $1_\S$.
   \end{enumerate}
We say $\tau \subseteq \S$ is a \textit{partial decomposition} if it is a subset of a decomposition. We denote by $\PD(\S)$ the set of all partial decompositions and by $\D(\S)$ the set of decompositions of $\S$.
We order $\PD(\S)$ and $\D(\S)$ by refinement. That is,
$\tau\leq \sigma$ if and only if for all $x\in \tau$ there exists $y\in \sigma$ such that $x\leq y$.
\end{definition}

Note that $0_{\S}$ cannot be part of a decomposition.
It also follows from the definition that $\PD(\S)$ is a bounded poset where the empty set $\emptyset$ is the unique minimal element, and $\{1_{\S}\}$ is the unique maximal element. The subposet $\D(\S)$ is an upper-order ideal of $\PD(\S)$ which may not have a unique minimal element.
We study the case where $\D(\S)$ has a unique minimal element in \ref{coro:minElementD}.

The following simple properties of full and partial decompositions are immediate consequences of the definition.

\begin{lemma}
\label{rk:containmentFull}
Let $\S$ be a bounded poset and $\sigma,\tau \in \D(\S)$.
 \begin{enumerate}
    \item if $\tau \subseteq \sigma$, then $\tau = \sigma$.
    \item if $\tau \leq \sigma$ and $|\tau|=|\sigma|$ then $\tau = \sigma$.
 \end{enumerate}
\end{lemma}

\begin{remark}
\label{rk:severalOrders}
There are several natural orders on $\PD(\S)$:

\begin{enumerate}
    \item[$\bullet$] the order $\leq$ we have defined, i.e., order by refinement;
    \item[$\bullet$] the order $\leq'$ where $\tau\leq'\sigma$ if $\tau\leq \sigma$ and also there exist decompositions $\hat{\tau}\leq\hat{\sigma}$ with $\tau\subseteq \hat{\tau}$ and $\sigma\subseteq \hat{\sigma}$;
    \item[$\bullet$] the order $\subseteq$ induced by set-inclusion.
\end{enumerate}

  The ordering $\leq'$ coincides with $\leq$ if $\tau ,\sigma \in \D(\S)$. Although in general $\leq'$ is more restrictive than $\leq$, they will coincide in many important examples discussed later. See \ref{sec:examples}.
 
  On the other hand, we will not address the poset structure obtained from the ordering $\subseteq$.
  See \ref{rk:partitionsByInclusion} for an example comparing the inclusion and the refinement orderings.
\end{remark}

\begin{example}
\label{ex:booleanLattice}
    If $\S$ is the Boolean lattice of all subsets of a set $\tau$ ordered by
    refinement, then $\D(\S)$ is the partition lattice $\Pi(\tau)$. It is easily checked
    that for a decomposition $\sigma \in \D(\S)$ the upper interval
    $\D(\S)_{\geq \sigma}$ is isomorphic to $\Pi(\sigma)$.
    Recall that the M{\"o}bius number of $\Pi(\tau)$ is $(-1)^{|\tau|-1}(|\tau|-1)!$ if $|\tau|\geq 2$.

    For instance, let us consider the case $\tau = \{1,2,3\}$. The Hasse diagram of the poset of full decompositions (partitions in blue) and partial decompositions (partial partitions in blue or black) of the Boolean lattice on three elements is depicted in \ref{fig:part3}.
    Here we use the usual notation for set partition instead of the formal definition
    of (partial) decompositions, e.g., $1|23$ instead of $\{\{1\},\{2,3\}\}$.
\end{example}

\begin{figure}[ht]
\centering
\vskip-1.3cm
   \scalebox{1}{\begin{tikzpicture}[x=+4143sp, y=+4143sp]
\newdimen\XFigu\XFigu4143sp
\newdimen\XFigu\XFigu4143sp
\clip(525,-6913) rectangle (4910,-1331);
\tikzset{inner sep=+0pt, outer sep=+0pt}
\pgfsetlinewidth{+7.5\XFigu}
\pgfsetcolor{black}
\filldraw  (900,-4050) circle [radius=+45];
\filldraw  (3330,-4050) circle [radius=+45];
\filldraw  (4500,-4050) circle [radius=+45];
\filldraw  (1350,-4950) circle [radius=+45];
\filldraw  (2700,-4950) circle [radius=+45];
\filldraw  (4050,-4950) circle [radius=+45];
\filldraw  (1350,-5850) circle [radius=+45];
\filldraw  (2700,-5850) circle [radius=+45];
\filldraw  (4050,-5850) circle [radius=+45];
\filldraw  (2700,-6750) circle [radius=+45];
\pgfsetstrokecolor{black}
\pgfsetfillcolor{blue!50!black}
\filldraw  (2700,-2250) circle [radius=+45];
\filldraw  (4050,-3150) circle [radius=+45];
\filldraw  (1350,-3150) circle [radius=+45];
\filldraw  (2700,-3150) circle [radius=+45];
\filldraw  (2070,-4050) circle [radius=+45];
\pgfsetlinewidth{+15\XFigu}
\draw (2700,-2250)--(1350,-3150);
\draw (2700,-2250)--(2700,-3150);
\draw (2700,-2250)--(4050,-3150);
\draw (2070,-4050)--(4050,-3150);
\draw (2700,-3150)--(2070,-4050);
\draw (1350,-3150)--(2070,-4050);
\pgfsetstrokecolor{black}
\draw (1350,-3150)--(900,-4050);
\draw (900,-4050)--(1350,-4950)--(2070,-4050);
\draw (2070,-4050)--(2700,-4950);
\draw (2070,-4050)--(4050,-4950);
\draw (3330,-4050)--(2700,-4950);
\draw (4050,-4950)--(4500,-4050);
\draw (4050,-5850)--(4050,-4950);
\draw (2700,-4950)--(4050,-5850);
\draw (4050,-4950)--(2700,-5850);
\draw (1350,-5850)--(2700,-4950);
\draw (1350,-4950)--(2700,-5850);
\draw (1350,-4950)--(1350,-5850);
\draw (2700,-6750)--(1350,-5850);
\draw (2700,-5850)--(2700,-6750);
\draw (2700,-6750)--(4050,-5850);
\draw (2700,-3150)--(3330,-4050);
\draw (4050,-3150)--(4500,-4050);
\pgfsetfillcolor{black}

\pgftext[base,left,at=\pgfqpointxy{600}{-4050}] {\fontsize{12}{14.4}$12$}
\pgftext[base,left,at=\pgfqpointxy{3400}{-4050}] {\fontsize{12}{14.4}$13$}
\pgftext[base,left,at=\pgfqpointxy{4600}{-4050}] {\fontsize{12}{14.4}$23$}
\pgftext[base,left,at=\pgfqpointxy{990}{-4950}] {\fontsize{12}{14.4}$1|2$}
\pgftext[base,left,at=\pgfqpointxy{2340}{-4950}] {\fontsize{12}{14.4}$1|3$}
\pgftext[base,left,at=\pgfqpointxy{4150}{-4950}] {\fontsize{12}{14.4}$2|3$}
\pgftext[base,left,at=\pgfqpointxy{2800}{-5940}] {\fontsize{12}{14.4}$2$}
\pgftext[base,left,at=\pgfqpointxy{4150}{-5940}] {\fontsize{12}{14.4}$3$}
\pgftext[base,left,at=\pgfqpointxy{1150}{-5940}] {\fontsize{12}{14.4}$1$}
\pgfsetfillcolor{black}
\pgftext[base,left,at=\pgfqpointxy{1530}{-4050}] {\fontsize{12}{14.4}$1|2|3$}
\pgftext[base,left,at=\pgfqpointxy{2800}{-2250}] {\fontsize{12}{14.4}$123$}
\pgftext[base,left,at=\pgfqpointxy{950}{-3060}] {\fontsize{12}{14.4}$12|3$}
\pgftext[base,left,at=\pgfqpointxy{2800}{-3060}] {\fontsize{12}{14.4}$13|2$}
\pgftext[base,left,at=\pgfqpointxy{4150}{-3060}] {\fontsize{12}{14.4}$1|23$}
\pgfsetfillcolor{black}
\pgftext[base,left,at=\pgfqpointxy{2800}{-6900}] {\fontsize{12}{14.4}$\emptyset$}
\end{tikzpicture}}
\caption{Full and partial set-partitions of $\{1,2,3\}$} \label{fig:part3}
\end{figure}

Indeed, one reason for including the height condition in the definition of a decomposition is that we want to control upper intervals in $\D(\S)$ in the same way we do it in the partition lattice. 

\begin{lemma}
\label{lm:upperIntervalsDecomp}
  For a bounded poset $\S$ and $\tau\in\D(\S)$, we have $\D(\S)_{\geq \tau} \groupiso \Pi(\tau)$.
\end{lemma}

\begin{proof}
By definition, it is clear that for every partition $\pi$ of $\tau$, the set $\{ \bigvee_{y\in p} y \tq p\in \pi\}$ is a decomposition.
Now suppose that $\tau \leq \sigma$.
For $x\in \sigma$, let $x' = \bigvee_{y\in \tau\tq y\leq x} y$.
Then $\sigma' = \{x' \tq x\in \sigma\}$ is a decomposition since it is obtained from a suitable partition of $\tau$, with $|\sigma'| = |\sigma|$ and $\sigma'\leq \sigma$.
By \ref{rk:containmentFull}, $\sigma' = \sigma$.
This means that elements of $\sigma$ are obtained by joining elements of $\tau$.
Thus $\D(\S)_{\geq \tau} \groupiso \Pi(\tau)$.
\end{proof}

The lower intervals in $\D(\S)$ and $\PD(\S)$ are not necessarily direct products of decomposition or partial decomposition posets if we do not require additional ``relative-extension'' properties.
We have the following characterization.

\begin{lemma}
\label{lm:lowerIntervalsPD}
Let $\S$ be a bounded poset such that for every $\sigma \in \PD(\S)$ it holds:

{\noindent
(E1) \quad $\PD(\S_{\leq x}) = \PD(\S)_{\leq \{x\}}$ for all $x\in \sigma$, and

\noindent
(E2) \quad if $\tau_x\in \PD(\S_{\leq x})$ for $x\in \sigma$, then $\cup_{x\in\sigma} \tau_x\in \PD(\S)_{\leq \sigma}$. 
}

\noindent
Then for every $\sigma \in\PD(\S)$ we have $\PD(\S)_{\leq \sigma} = \prod_{x\in\sigma} \PD(\S_{\leq x})$.
If $\sigma\in \D(\S)$ then also $\D(\S)_{\leq \sigma} = \prod_{x\in\sigma} \D(\S_{\leq x})$.

Conversely, if we naturally have $\PD(\S)_{\leq \sigma} = \prod_{x\in\sigma} \PD(\S_{\leq x})$ then (E1) and (E2) hold.
\end{lemma}

\begin{proof}
Let $\tau\leq \sigma$ and define $\tau_x = \sigma\cap \S_{\leq x}$.
Then by (E1), $\tau_x\in \PD(\S_{\leq x})$.
Therefore we have a well-defined order-preserving map $\PD(\S)_{\leq \sigma}\to \prod_{x\in \sigma} \PD(\S_{\leq x})$ given by $\tau\mapsto (\tau_x)_{x\in \sigma}$.
This map is clearly an embedding by definition of the ordering.
Property (E2) also implies that this is surjective: indeed, if $\tau_x\in \PD(\S_{\leq x})$, then  
$\tau = \cup_{x\in \sigma}\tau_x \in \PD(\S)_{\leq \sigma}$ by (E2), and $\tau\cap \S_{\leq x} = \tau_x$.

The structure of the lower interval in the decomposition poset follows immediately from the result in the partial decompositions. 
We leave the proof of the reverse implication for the reader.
\end{proof}

If we drop the requirement on the heights in \ref{def:decompositionsAndPD}, we may encounter situations where we have two distinct comparable decompositions of the same size.
In particular, the conclusions of \ref{lm:upperIntervalsDecomp} may fail.
This leads us to the definition of weak decompositions.

\begin{definition}
\label{def:weakDecompositions}
Let $\S$ be a bounded poset (not necessarily of finite height), and let $\sigma \subseteq  \S$.
We say that $\sigma$ is a \textit{weak decomposition} if for all $\tau\subseteq \sigma$, $\bigvee_{x\in \tau} x$ exists and the subposet $\{ \bigvee_{x\in \tau} x \tq \tau \subseteq \sigma\}$, with suprema and infima taken in $\S$, is a Boolean lattice on $|\sigma|$ elements with maximal element $1_\S$.
We analogously define the posets of weak decompositions $\D_w(\S)$ and weak partial decompositions $\PD_w(\S)$, both ordered by refinement.
\end{definition}

Clearly decompositions are weak decompositions, so $\D(\S)\subseteq \D_w(\S)$ and $\PD(\S)\subseteq \PD_w(\S)$ are subposets.

Note that in \ref{ex:booleanLattice}, weak decompositions are decompositions. So in that case $\PD(\Sub) = \PD_w(\Sub)$.
However, the following example shows that weak decompositions are not always decompositions and that upper-intervals in $\D_w(\S)$ may not be partition lattices.

\begin{example}
\label{ex:weakDefinitionDecomposition}
Let $\S$ be a bounded poset.
Then $d\in \D_w(\S)$ if and only if 
\begin{enumerate}
    \item[$\bullet$] either $d = \{x,y\}$, with $x$ and $y$ belonging to different connected components of $\red{\S}$, or 
    \item[$\bullet$] else $d\in \D_w(C \cup \{0_{\S},1_{\S}\})$ where $C$ is a connected component of $\red{\S}$.
    \end{enumerate}
Indeed, if $d\in\D_w(\S)$ contains elements $x,y$ in different connected components then $x\vee y = 1$. Thus, for $z\in d$ we have $z \wedge (x\vee y) = z$, which implies that $z = x$ or $y$.
Note also that if $C_1,C_2$ are two different connected components of $\red{\S}$ and $x,y,z\in \red{\S}$ are elements such that $x \leq y \in C_1$ and $z\in C_2$, then $\{x,z\}\leq \{y,z\}$ and both are weak decompositions of $\S$.
In particular, $\D_w(\S)_{\geq \{x,z\}} = \S_{\geq x} \times \S_{\geq z}$, so the upper interval is not necessarily a partition lattice.

Indeed, suppose that $\S = \{0,a,b,c,d,1\}$ with minimum $0$, maximum $1$, $a<c$ and $b<d$.
We have $\D_w(\S) = \{ \{a,b\}, \{b,c\}, \{a,d\}, \{c,d\}, \{1\} \}$, and the upper-interval $\D_w(\S)_{\geq \{a,b\}}$ is not a partition lattice.
See \ref{fig:weakDecompositions} for the Hasse diagrams of $\S$, $\PD_w(\S)$ and $\PD(S)$, and the corresponding decomposition posets which are depicted in blue.
Note that $\S$ is a lattice but $\PD(\S)$ is not.

This example also shows that $(\PD(\S),\leq)\neq (\PD(\S),\leq')$, where $\leq'$ is the ordering defined in \ref{rk:severalOrders}.
The latter poset is displayed in \ref{fig:strongerOrdering}.

\begin{figure}[ht]
\centering

\scalebox{0.85}{
\tikzpicture[x=+4143sp, y=+4143sp]
\newdimen\XFigu\XFigu4143sp
\newdimen\XFigu\XFigu4143sp
\definecolor{blue3}{rgb}{0,0,0.82}
\definecolor{brown3}{rgb}{0.75,0.38,0}
\clip(1860,-3620) rectangle (8602,90);
\tikzset{inner sep=+0pt, outer sep=+0pt}
\pgfsetlinewidth{+7.5\XFigu}
\pgfsetcolor{black}
\filldraw  (2505,-2385) circle [radius=+45];
\filldraw  (2955,-1935) circle [radius=+45];
\filldraw  (2055,-1935) circle [radius=+45];
\filldraw  (2055,-1485) circle [radius=+45];
\filldraw  (2955,-1485) circle [radius=+45];
\filldraw  (2505,-1035) circle [radius=+45];
\pgfsetlinewidth{+15\XFigu}
\draw (2505,-2385)--(2055,-1935);
\draw (2505,-2385)--(2955,-1935);
\draw (2055,-1935)--(2055,-1485);
\draw (2955,-1935)--(2955,-1485);
\draw (2055,-1485)--(2505,-1035);
\draw (2505,-1035)--(2955,-1485);
\pgftext[base,left,at=\pgfqpointxy{2415}{-3600}] {\fontsize{12}{14.4}$\S$}
\pgftext[base,left,at=\pgfqpointxy{2460}{-2700}] {\fontsize{12}{14.4}$0$}
\pgftext[base,left,at=\pgfqpointxy{1875}{-2025}] {\fontsize{12}{14.4}$a$}
\pgftext[base,left,at=\pgfqpointxy{3045}{-2070}] {\fontsize{12}{14.4}$b$}
\pgftext[base,left,at=\pgfqpointxy{1875}{-1485}] {\fontsize{12}{14.4}$c$}
\pgftext[base,left,at=\pgfqpointxy{2460}{-900}] {\fontsize{12}{14.4}$1$}
\pgftext[base,left,at=\pgfqpointxy{3045}{-1485}] {\fontsize{12}{14.4}$d$}
\pgfsetlinewidth{+7.5\XFigu}
\filldraw  (4230,-2390) circle [radius=+45];
\filldraw  (4050,-1940) circle [radius=+45];
\filldraw  (5400,-1940) circle [radius=+45];
\pgfsetlinewidth{+15\XFigu}
\pgfsetcolor{blue3}
\filldraw  (4725,-635) circle [radius=+45];
\filldraw  (4725,-230) circle [radius=+45];
\filldraw  (5400,-1085) circle [radius=+45];
\filldraw  (4725,-1715) circle [radius=+45];
\pgfsetlinewidth{+7.5\XFigu}
\pgfsetcolor{black}
\filldraw  (5220,-2390) circle [radius=+45];
\filldraw  (4725,-2840) circle [radius=+45];
\pgfsetlinewidth{+15\XFigu}
\pgfsetcolor{blue3}
\filldraw  (4050,-1080) circle [radius=+45];
\pgfsetstrokecolor{black}
\draw (4725,-1715)--(4230,-2390);
\draw (4725,-1715)--(5220,-2390);
\draw (4725,-2840)--(4230,-2390);
\draw (4725,-2840)--(5220,-2390);
\draw (4230,-2390)--(4050,-1940);
\draw (5220,-2390)--(5400,-1940);
\draw (5400,-1085)--(5400,-1940);
\pgfsetstrokecolor{blue3}
\draw (4725,-230)--(4725,-635);
\draw (4725,-635)--(4060,-1075);
\draw (4725,-635)--(5400,-1085);
\draw (4725,-1715)--(5400,-1085);
\draw (4725,-1715)--(4045,-1085);
\pgfsetstrokecolor{black}
\draw (4050,-1080)--(4050,-1920);
\pgfsetfillcolor{black}
\pgftext[base,left,at=\pgfqpointxy{4410}{-3590}] {\fontsize{12}{14.4}$\PD_w$}
\pgftext[base,left,at=\pgfqpointxy{3735}{-1895}] {\fontsize{12}{14.4}$\{c\}$}
\pgftext[base,left,at=\pgfqpointxy{3915}{-2570}] {\fontsize{12}{14.4}$\{a\}$}
\pgftext[base,left,at=\pgfqpointxy{5445}{-1895}] {\fontsize{12}{14.4}$\{d\}$}
\pgfsetfillcolor{blue3}
\pgftext[base,left,at=\pgfqpointxy{5445}{-1040}] {\fontsize{12}{14.4}$\{a,d\}$}
\pgftext[base,left,at=\pgfqpointxy{4815}{-635}] {\fontsize{12}{14.4}$\{c,d\}$}
\pgftext[base,left,at=\pgfqpointxy{4680}{-95}] {\fontsize{12}{14.4}$1$}
\pgftext[base,left,at=\pgfqpointxy{4185}{-1760}] {\fontsize{12}{14.4}$\{a,b\}$}
\pgfsetfillcolor{black}
\pgftext[base,left,at=\pgfqpointxy{5355}{-2525}] {\fontsize{12}{14.4}$\{b\}$}
\pgftext[base,left,at=\pgfqpointxy{4680}{-3155}] {\fontsize{12}{14.4}$\emptyset$}
\pgfsetlinewidth{+7.5\XFigu}
\filldraw  (6855,-2370) circle [radius=+45];
\filldraw  (6675,-1920) circle [radius=+45];
\filldraw  (8025,-1920) circle [radius=+45];
\pgfsetlinewidth{+15\XFigu}
\pgfsetcolor{blue3}
\filldraw  (7350,-210) circle [radius=+45];
\filldraw  (8025,-1065) circle [radius=+45];
\pgfsetlinewidth{+7.5\XFigu}
\pgfsetcolor{black}
\filldraw  (7340,-2860) circle [radius=+45];
\filldraw  (7840,-2385) circle [radius=+45];
\pgfsetlinewidth{+15\XFigu}
\pgfsetcolor{blue3}
\filldraw  (6670,-1065) circle [radius=+45];
\pgfsetstrokecolor{black}
\draw (8025,-1065)--(6855,-2370);
\draw (6665,-1060)--(7840,-2400);
\draw (7350,-2870)--(6855,-2370);
\draw (7345,-2865)--(7845,-2395);
\draw (6855,-2370)--(6675,-1920);
\draw (7835,-2390)--(8025,-1920);
\draw (8025,-1065)--(8025,-1920);
\draw (6675,-1065)--(6675,-1920);
\pgfsetstrokecolor{blue3}
\draw (7350,-210)--(6670,-1065);
\draw (7350,-210)--(8025,-1065);
\pgfsetfillcolor{black}
\pgftext[base,left,at=\pgfqpointxy{7215}{-3585}] {\fontsize{12}{14.4}$\PD$}
\pgftext[base,left,at=\pgfqpointxy{7260}{-3135}] {\fontsize{12}{14.4}$\emptyset$}
\pgftext[base,left,at=\pgfqpointxy{6360}{-1875}] {\fontsize{12}{14.4}$\{c\}$}
\pgftext[base,left,at=\pgfqpointxy{6540}{-2550}] {\fontsize{12}{14.4}$\{a\}$}
\pgftext[base,left,at=\pgfqpointxy{8070}{-1875}] {\fontsize{12}{14.4}$\{d\}$}
\pgfsetfillcolor{blue3}
\pgftext[base,left,at=\pgfqpointxy{6180}{-1065}] {\fontsize{12}{14.4}$\{b,c\}$}
\pgftext[base,left,at=\pgfqpointxy{7305}{-75}] {\fontsize{12}{14.4}$1$}
\pgfsetfillcolor{black}
\pgftext[base,left,at=\pgfqpointxy{7890}{-2550}] {\fontsize{12}{14.4}$\{b\}$}
\pgfsetfillcolor{blue3}
\pgftext[base,left,at=\pgfqpointxy{8095}{-1035}] {\fontsize{12}{14.4}$\{a,d\}$}
\pgftext[base,left,at=\pgfqpointxy{3590}{-1010}] {\fontsize{12}{14.4}$\{b,c\}$}
\endtikzpicture%
}

    \caption{From left to right, Hasse diagrams of posets $\S$, $\PD_w(\S)$ and $\PD(\S)$ respectively. The corresponding decomposition posets are displayed in blue}
    \label{fig:weakDecompositions}
\end{figure}

\begin{figure}
\centering
\tikzpicture[x=+4143sp, y=+4143sp]
\newdimen\XFigu\XFigu4143sp
\newdimen\XFigu\XFigu4143sp
\definecolor{blue3}{rgb}{0,0,0.82}
\definecolor{brown3}{rgb}{0.75,0.38,0}
\clip(3675,-2130) rectangle (7178,-105);
\tikzset{inner sep=+0pt, outer sep=+0pt}
\pgfsetlinewidth{+7.5\XFigu}
\pgfsetcolor{black}
\filldraw  (5400,-1800) circle [radius=+45];
\filldraw  (4050,-1350) circle [radius=+45];
\filldraw  (5850,-1350) circle [radius=+45];
\filldraw  (6750,-1350) circle [radius=+45];
\filldraw  (4950,-1350) circle [radius=+45];
\pgfsetcolor{blue3}
\filldraw  (4500,-900) circle [radius=+45];
\filldraw  (6300,-900) circle [radius=+45];
\filldraw  (5400,-450) circle [radius=+45];
\pgfsetlinewidth{+15\XFigu}
\draw (4500,-900)--(5400,-450);
\draw (5400,-450)--(6300,-900);
\pgfsetstrokecolor{black}
\draw (4050,-1350)--(4500,-900);
\draw (4500,-900)--(4950,-1350);
\draw (5400,-1800)--(4050,-1350);
\draw (5400,-1800)--(4950,-1350);
\draw (6300,-900)--(5400,-1800);
\draw (5400,-1800)--(6750,-1350);
\draw (6300,-900)--(6750,-1350);
\pgfsetfillcolor{black}
\pgftext[base,left,at=\pgfqpointxy{5355}{-2115}] {\fontsize{12}{14.4}$\emptyset$}
\pgftext[base,left,at=\pgfqpointxy{3690}{-1395}] {\fontsize{12}{14.4}$\{a\}$}
\pgftext[base,left,at=\pgfqpointxy{4590}{-1440}] {\fontsize{12}{14.4}$\{d\}$}
\pgftext[base,left,at=\pgfqpointxy{5490}{-1395}] {\fontsize{12}{14.4}$\{b\}$}
\pgftext[base,left,at=\pgfqpointxy{6840}{-1395}] {\fontsize{12}{14.4}$\{c\}$}
\pgfsetfillcolor{blue3}
\pgftext[base,left,at=\pgfqpointxy{4050}{-810}] {\fontsize{12}{14.4}$\{a,d\}$}
\pgftext[base,left,at=\pgfqpointxy{5355}{-270}] {\fontsize{12}{14.4}$1$}
\pgftext[base,left,at=\pgfqpointxy{6345}{-765}] {\fontsize{12}{14.4}$\{b,c\}$}
\endtikzpicture%

\caption{Poset $(\PD(\S),\leq')$}
\label{fig:strongerOrdering}
\end{figure}

\end{example}

Let us analyze now some examples in subgroup posets.

\begin{example}
\label{ex:symmetric3}
Let $G = S_3$ be the symmetric group on $3$ letters.
Then the poset of subgroups of $G$, denoted by $\S(G)$, consists of a unique minimal element (the trivial subgroup, denoted by $1$ in this example), a unique maximal element (the whole group $G$), a subgroup of order $3$ generated by a $3$-cycle, and three subgroups of order $2$ (one for each transposition).
Then $\red{\S(G)}$ is a discrete poset with $4$ points.
Hence, a decomposition of $\S(G)$ different from $\{G\}$ is a set of two non-trivial proper subgroups of $G$.
\end{example}

\begin{example}
\label{ex:cyclicGroup}
Let $G$ be a cyclic group of order $m$.
If $m = p^n$ is a prime power, then $\S(G)$, the lattice of subgroups of $G$, has height $n$, and $\PD(\S(G)) = \D(\S(G)) = \{ \{G\} \}$.

Suppose $G = C_{p_1^{n_1}}\times\cdots\times C_{p_r^{n_r}}$ for different primes $p_j$.
Let $G_j = C_{p_j^{n_j}}$.
Then the height of a subgroup of $G$ equals the number of primes (counted with multiplicity) of its order.
In particular, $h(\S(G)) = \sum_j n_j$.
Therefore, the only possible decompositions of $G$ are those obtained by coarsening the decomposition $\{G_1,\ldots,G_r\}$.
Hence $\D(\S(G))$ is the partition lattice on the set $\{G_1,\ldots,G_r\}$.
\end{example}

\subsection{Monoidal categories and restricted decompositions}
\label{sub:monoidalCat}

The constructions above will not cover all the structures we would like to
study. 
For example, consider the case of a non-trivial finitely generated free Abelian group $A$.
Its poset of subgroups is a lattice of infinite height.
Therefore, we cannot talk about decompositions given the finite-dimensional requirement in \ref{def:decompositionsAndPD}.
Nevertheless, it is implicit here that decompositions should be sets of subgroups $A_1,\ldots, A_r$ which span a Boolean lattice inside the poset of subgroups with join $A_1\oplus \cdots \oplus A_r = \gen{A_1,\ldots,A_r} = A$.
Thus, if we instead consider the subposet of complemented subgroups of $A$ (or \textit{summands}), this is a bounded poset of finite height.
However, there is still a problem in this subposet: even if $\{A_1,A_2\}$ is a (weak) decomposition in the poset sense, it is not clear that $A_1\vee A_2$ coincides with the direct sum $A_1\oplus A_2$.
For example, if $A = \ZZ^2$, then $A_1=\gen{(1,2)}$ and $A_2=\gen{(1,0)}$ are summands of $A$ such that $A_1\wedge A_2 = 0$ and $A_1\vee A_2 = A$ in the poset of summands of $A$, but $A_1+A_2 = A_1\oplus A_2 < A$.

A similar situation arises if we consider instead a free group of rank $\geq 2$.
By a decomposition here we not only want the Boolean lattice condition, but also ``free" spans.
That is, if $F$ is a free group of rank $\geq 2$ and $H,K$ are two subgroups, then $H,K$ is a weak decomposition if $H\cap K = 1$ and $H\vee K = \gen{H,K}=F$.
But the span $\gen{H,K}$ forgets the relation between the elements of $H$ and $K$, 
and we may actually want $\gen{H,K}$ to be equal to the free product of $H$ and $K$.
An adaptation of the Abelian-case example with $\ZZ^2$ shows that there are free factors $H,K$ of the free group $F_2$ of rank $2$ such that $H\cap K = 1$ but $H\vee K = \gen{H,K} = F$ is not the free product of $H$ and $K$.

To cover these and other cases, we extend our purely combinatorial setting 
by using concepts from category theory. We start with some notation.

For a category $\C$ and an object $X\in \C$, write $\Hom_\C(Y,X)$ for the homomorphisms from an object $Y$ to $X$, and $\Aut_\C(X)$ for the isomorphisms of $X$. Let $\C \downarrow X$ be the
comma category with objects $(Y,f)$, where
$Y \in \C$ and $f \in \Hom_{\C}(Y,X)$, and
morphisms between $(Y,f)$ and $(Z,g)$ being $h \in \Hom_{\C}(Y,Z)$ with 
$f = gh$.
We consider the subcategory 
$\C \hookdownarrow X$ with objects $(Y,i)$ such that $i$ is a monomorphism, and morphisms inherited from $\C \downarrow X$. Note that a morphism $h : (Y,i) \rightarrow
(Z,j)$ in $\C \hookdownarrow X$ comes indeed from a monomorphism $h:Y\to Z$.

A \textit{subobject} of $X$ is an isomorphism class of objects in 
$\C \hookdownarrow X$, where $(Y,i)$ and $(Z,j)$ are isomorphic if and only if there exists an isomorphism $\phi:Y\to Z$ such that $i = j\phi$. We write $[(Y,i)]$
for the subobject represented by $(Y,i)$.
Let $\Sub(X)$ denote the collection of subobjects of $X$, which is a poset-category (that is, there is at most one arrow between any two objects and this defines an anti-symmetric relation).
We gather some properties in the following lemma.

\begin{lemma}
Let $X$ be an object of a category $\C$.
The following hold:
\begin{enumerate}
    \item $\C\hookdownarrow X$ is a pre-order with $x\lesssim y$ if there is an arrow from $x$ to $y$.
    Namely, there is at most one arrow between any two objects.
    \item $\Sub(X)$ is a skeleton of the category $\C\hookdownarrow X$.
    \item If $\C$ has an initial object $0$, then $0\in \Sub(X)$ is the unique minimum of this poset.
    \item The poset $\Sub(X)$ has a unique maximum given by the class of $(X,i)$, where $i$ is any isomorphism of $X$.
    We denote this class just by $X$.
    \item If $z\in \Sub(X)$ then there is a poset isomorphism $\Sub(X)_{\leq z} \groupiso \Sub(Z)$ for any representative $(Z,i)\in z$.
\end{enumerate}
\end{lemma}

\begin{proof}
These are straightforward facts from the definitions.
\end{proof}

Note that item 5 also implies that if $i,i':Z\to X$ are two monomorphisms and $z:=[(Z,i)]$, $z':=[(Z,i')]$ then $\Sub(X)_{\leq z}\groupiso \Sub(Z) \groupiso \Sub(X)_{\leq z'}$.
In particular, if $i:X\to X$ is a monomorphism, then $\Sub(X)_{\leq [(X,i)]} \groupiso \Sub(X)$.

\begin{warning*}
$\Sub(X)$ is always a bounded poset but, in general, of infinite height.
\end{warning*}

We want to restrict the collection of decompositions to those that have a 
particular shape. To specify the ``shape" of the decompositions, we define beforehand a 
``product of objects", and then we select the decompositions whose join equals the 
product. 
This naturally leads us to the definition of a monoidal product in our category.
However, we will require this product to have some particular properties closer to a coproduct.

\begin{definition}
\label{def:ISMcategory}
An \textit{initial symmetric monoidal category}, or \textit{ISM-category} for short,
is an initial monoidal category $(\C,\sqcup)$, where $\sqcup$ is a naturally associative and commutative product, and the initial object of the category is also a neutral element for $\sqcup$. 
\end{definition}

The notation ``$\sqcup$'' suggests a similarity with the coproduct.
Indeed, sometimes $\sqcup$ will be the coproduct of the category.
However, in many cases, the category $\C$ may not have all coproducts, and so $\sqcup$
will serve as a remedy. 

\begin{remark}
\label{rk:naturalMapToProduct}
Note that, for any two objects $X$, $Y$, we have a \textit{canonical map} $\iota_X:X\to X\sqcup Y$ arising from
\[ X \groupiso X \sqcup 0 \overset{\id_X \sqcup i_0}{\longrightarrow}X\sqcup Y,\]
where $i_0:0\to Y$.
\end{remark}

To define decompositions in $\Sub(X)$ we need finite height, and $\Sub(X)$ may have infinite dimension in general (for example, when $X$ is a non-trivial free group in the category of groups, see \ref{exa:freeGroupsStarting}).
For that reason, we will work with the subposet $\Sub(X,\sqcup)$ of $\sqcup$-complemented subobjects defined below, which we expect to have finite height and whose elements are eligible to be part of decompositions.
To that end, we introduce first the notion of $\sqcup$-compatible joins of elements.

\begin{definition}
\label{def:sqcupCompatible}
Let $(\C,\sqcup)$ be an ISM-category, $X\in \C$ and $\sigma = \{x_1,\ldots, x_r\}\subseteq \T\subseteq \Sub(X)$.
We say that $\sigma$ is \textit{$\sqcup$-compatible in $\T$} if there are suitable representatives $x_i = [(X_i, j_i)]$ such that the following conditions hold:
\begin{enumerate}
 \item For all $\emptyset\neq I \subseteq I' \subseteq [r]$,
  we have natural monomorphisms \footnote{In the sense that they are induced by the canonical maps $Y\to Y\sqcup Z \groupiso Z\sqcup Y$.}
  $$j_{I,I'} : \bigsqcup_{i\in I} X_i \rightarrow  \bigsqcup_{i\in I'} X_i \qquad \text{ and } \qquad j_{I}:\bigsqcup_{i\in I} X_i \rightarrow X;$$
 \item for any $\emptyset\neq I\subseteq [r]$, there exists the join $\bigvee_{i\in I} x_i$ in $\T$ and
    $$\Big[\,\big(\, \bigsqcup_{i \in I} X_i,j_I\,\big)\, \Big] = \bigvee_{i\in I} x_i;$$

 \item
 the following diagram commutes for all $\emptyset\neq I \subseteq I' \subseteq [r]$ and $k\in I$:
 \[\xymatrix{
    X_k \ar[r]^{j_{\{k\},I}} \ar[dr]_{j_k} \ar@/^2pc/[rr]^{j_{\{k\},I'}} & \bigsqcup_{i\in I} X_i \ar[r]^{j_{I,I'}} \ar[d]^{j_I} & \bigsqcup_{i\in I'} X_i \ar[dl]^{j_{I'}}\\
    & X &
 }\]
\end{enumerate}
\end{definition}
We now define posets of subobjects compatible with the monoidal product $\sqcup$:
\begin{definition}
\label{def:decompositionsCategory}
Let $(\C,\sqcup)$ be an ISM-category and let $X\in \C$ be an object.
\begin{enumerate}
    \item An element $x\in \Sub(X)$ is $\sqcup$-\textit{complemented} if there exists a poset-complement $y\in \Sub(X)$ of $x$ (that is, $x\wedge y = 0$ and $x\vee y = 1$) such that the set $\{x,y\}\subseteq\Sub(X)$ is $\sqcup$-compatible in $\Sub(X)$.
    \item We define $\Sub(X,\sqcup)$ to be the subposet of $\sqcup$-complemented elements of $\Sub(X)$.
    \item If $\Sub(X,\sqcup)$ has finite height, a $\sqcup$-\textit{decomposition}  of $X$ is a decomposition $\sigma \in \D(\Sub(X,\sqcup))$ which is $\sqcup$-compatible in $\Sub(X,\sqcup)$.
    \item Denote by $\D(X,\sqcup)$ the subposet of $\D(\Sub(X,\sqcup))$ consisting of $\sqcup$-decompositions, and by $\PD(X,\sqcup)$ the subposet of $\PD(\Sub(X,\sqcup))$ consisting of the subsets of $\sqcup$-decompositions.
\end{enumerate}
Similarly, we define $\sqcup$-weak (partial) decompositions but without the need to require finite height for $\Sub(X,\sqcup)$.
\end{definition}

Note that by definition $\Sub(X,\sqcup)$ is bounded with maximum $X$ and minimum $0$ (the initial object), and all its elements are $\sqcup$-complemented, and in particular complemented.
Also $\PD(X,\sqcup)$ is bounded with maximum $\{X\}$ and minimum $\emptyset$.
Moreover, $\D(X,\sqcup)$ is a non-empty order ideal of $\PD(X,\sqcup)$ and, in particular, 
$\{X\}$ is the unique maximal element of $\D(X,\sqcup)$.

\begin{example}
    Let $\C$ be the category of groups, so $\Sub(X)$ is the subgroup lattice for a group $X$.
    We choose ``$\sqcup$'' to be the direct product of groups.
    In order to get a finite-dimensional poset $\Sub(X,\sqcup)$, we work with a finite group $X$.
    Hence $\Sub(X,\sqcup)$ is the subposet of subgroups of $X$ which are direct-product factors of $X$.
    That is, $H\leq X$ is $\sqcup$-complemented if and only if there exists a subgroup $K\leq X$ such that $H\cap K=1$, $[H,K]=1$ and $H\times K \groupiso HK = X$.
    Recall that, by the Krull-Schmidt theorem for finite groups, there exists a unique direct product decomposition of a finite group into indecomposable factors, up to permutation and isomorphism of the factors.
    In particular, the minimal elements of $\Sub(X,\sqcup)$ are the indecomposable factors of $X$ and the height of an element $H\in \Sub(X,\sqcup)$ is the number of (non-trivial) indecomposable direct product factors.
    Therefore, $\D_w(X,\sqcup)$ is the poset of all (internal) direct-product decompositions of $X$ ordered by refinement, and $\D(X,\sqcup)$ equals $\D_w(X,\sqcup)$.
\end{example}

In view of the previous example and \ref{ex:cyclicGroup}, we propose the following question:

\begin{question}
If $G$ is the direct product of groups $G_1,\ldots,G_r$ of relatively prime orders, do we have $\D(\S(G),\times) = \Pi(r)\times \D(G_1,\times)\times\cdots\times \D(G_r,\times)$?    
\end{question}

\begin{example}
\label{exa:freeGroupsStarting}
Recall that the coproduct in the category of groups is the free product. So we choose ``$\sqcup$'' to be the free product of groups and $X$ a free group of finite rank.
Then $\Sub(X)$ is a lattice with infinite height, and we cannot apply our definition of decompositions.
However, $\Sub(X,\sqcup)$ does have finite height.
Indeed, we claim that $\Sub(X,\sqcup)$ is the poset $\FC(X)$ of free factors of $X$.
To see this, let $H\in \Sub(X,\sqcup)$ and let $K\in \Sub(X)$ be a $\sqcup$-complement for $H$.
This means that $H\cap K = 1$ and $H*K$ is naturally isomorphic to $\gen{H,K} = X$.
That is, $X = H*K$, so $H,K$ are free factors of $X$.
Since a free factor of $X$ belongs to $\Sub(X,\sqcup)$, we conclude that $\Sub(X, \sqcup) = \FC(X)$.
In particular, $\FC(X)$ is a graded poset and if $H\in \FC(X)$ then its height equals $\rk(H)$, the rank of $H$ as a (free) group.
Moreover, since $\rk(H*K) = \rk(H)+\rk(K)$, we conclude that $\sqcup$-weak decompositions are $\sqcup$-decompositions. 
Therefore, $\D_w(X,\sqcup) = \D(X,\sqcup)$ is the poset of decompositions $\{H_1,\ldots, H_k\}$ of $X$ into free factors $X = H_1 * \cdots * H_k$.
\end{example}

\begin{example}
\label{ex:latticeAsCategoryApproach}
Let $\S$ be a lattice.
Then we have the induced poset-category $\C$ whose objects are the elements of $\S$ with arrows $x\to y$ if and only if $x\leq y$.
Then $\Sub(1_{\S}) = \S$.
Moreover, $\C$ is an ISM-category with $\sqcup$ given by the join of elements in $\S$, i.e. $x\sqcup y = x\vee y$.
It is straightforward to verify then that $\Sub(1_{\S},\sqcup)$ is the subposet of complemented elements of $\S$.
In particular, if every element of $\S$ is complemented then $\S = \Sub(1_{\S},\sqcup)$, $\D(\S) = \D(\S,\sqcup)$ and $\PD(\S) = \PD(1_{\S},\sqcup)$.
\end{example}

\section{Properties of decomposition posets}
\label{sec:properties}

From now on, we assume that $X$ is an object of an ISM-category $(\C,\sqcup)$ such that $\Sub(X,\sqcup)$ has finite height.
Recall by \ref{ex:latticeAsCategoryApproach} that we can regard a complemented lattice $\S$ as some poset $\S(X,\sqcup)$ where also $\D(\S) = \D(X,\sqcup)$ and $\PD(\S) = \PD(X,\sqcup)$.
Therefore, the results of this section are also valid for complemented lattices of finite height.

The following proposition recovers the structure of upper intervals of decompositions and it is an immediate consequence of the definitions.

\begin{proposition}
  \label{prop:decompintervals}
  For $\sigma \in \D(X,\sqcup)$ we have
  $\D(X,\sqcup)_{\geq \sigma} \groupiso \Pi(\sigma)$.
\end{proposition}

\begin{proof}
  This follows from \ref{lm:upperIntervalsDecomp} and the fact that $\D(X,\sqcup)$ is an upper order ideal in $\D(\Sub(X,\sqcup))$.
\end{proof}

In some important cases, we also have control over the lower intervals and the following property holds:

\medskip

\begin{definition}
\label{def:propLI}
We say that $X$ satisfies property (LI) if the following condition holds:

\begin{center}
\begin{tabular}{cc}
    \textbf{(LI)} & \pbox{12.5cm}{For all $\sigma \in \PD(X,\sqcup)$ we have $\PD(X,\sqcup)_{\leq \sigma} = \prod_{y\in \sigma} \PD(Y,\sqcup)$, for suitable representatives $(Y,i_y)\in y$.}
\end{tabular}
\end{center}

\medskip

\noindent
Here (LI) stands for Lower Interval.
\end{definition}

In particular, property (LI) implies that the lower interval $\D(X,\sqcup)_{\leq \sigma}$ is also a product $\prod_{y\in \sigma} \D(Y,\sqcup)$.
Conditions analogous to those from \ref{lm:lowerIntervalsPD} imply property (LI).

The following proposition provides some information on the heights of the posets of decompositions and partial decompositions.

\begin{proposition}
  \label{prop:heightDandPD}
  Let $\Sub(X,\sqcup)$ be of height $n$.
  Then every decomposition in $\D(X,\sqcup)$ has size at most $n$.
  In particular, $h(\PD(X,\sqcup)) \leq 2n-1$ and $h(\D(X,\sqcup))\leq n-1$.
  Moreover, the following are equivalent
  \begin{enumerate}
    \item there exists a decomposition of size $n$.
    \item $h(\PD(X,\sqcup)) = 2n-1$. 
    \item $h(\D(X,\sqcup)) = n-1$.
  \end{enumerate}
  If one of the preceding conditions holds, then the following are equivalent:
  \begin{enumerate}
    \item[1'.] for every decomposition $\sigma$ there exists a decomposition 
    $\tau$ of size $n$ with $\tau \leq \sigma$.
    \item[3'.] $\D(X,\sqcup)$ is graded. 
  \end{enumerate}
\end{proposition}

\begin{proof}
  The bound on the size of elements of $\D(X,\sqcup)$ is an immediate consequence of
  the definitions. The bound on the height of $\D(X,\sqcup)$ then follows from \ref{prop:decompintervals}.
  For $\PD(X,\sqcup)$ we argue as follows.
  Suppose that $\sigma \in \PD(X,\sqcup)$, $\sigma\neq\emptyset$, and consider a saturated chain $\gamma$ in $\PD(X,\sqcup)$ with maximal element $\sigma$.
  In particular, $\gamma$ contains the empty partial decomposition.
  Then, a covering relation $\tau \prec \tau'$ in $\gamma$ means that either we ``split" some elements of $\tau'$ (meaning that for some $x\in\tau'$ there are $y,z \in \tau$ with 
  $x \geq y \vee z$ and $x \not\in \tau$) or we eliminated an ``unsplittable" element (i.e. there is $x\in \tau'$ which cannot be split and $\tau= \tau'\setminus \{x\}$).
  Therefore, for $x \in \tau$ we can perform at most $h(x)-1$ splittings and eliminate at most $h(x)$ elements.
  Thus, the maximal chain $\gamma$ has size at most
  \[ 1+\sum_{x \in \sigma} \big(\,h(x)-1 + h(x)\,\big) = 1-|\sigma|+2\sum_{x \in \sigma} h(x) \leq 1-|\sigma| + 2n \leq 2n.\]
  This implies that $h(\PD(X,\sqcup))\leq 2n-1$.

  The equivalence of items 1 and 3 immediately follows from \ref{prop:decompintervals}.
  The equivalence of items 1 and 2 is a consequence of the proof of the bound 
  $h(\PD(X,\sqcup)) \leq 2n-1$. 

  Given the equivalent conditions of items 1, 2 and 3, again \ref{prop:decompintervals}
  implies the equivalence of 1' and 3'. 
\end{proof}

In general, $\D(X,\sqcup)$ has a unique maximal but not necessarily 
a unique minimal element. As defined at the beginning of \ref{sec:definitionsposet}, for a poset $\S$ with a unique
maximal element, we write $\redm{\S}$ for the subposet obtained by
removing the unique maximal element, and in case $\S$ has
also a unique minimal element we denote by $\red{\S}$ the
subposet obtained by removing both the unique minimal and the unique maximal element.
When $\Sub(X,\sqcup)$ is the Boolean lattice with $\sqcup$ the disjoint union, $\D(X,\sqcup)$ is the partition lattice which has a unique
minimal element and in the literature $\red{\D(X,\sqcup)}$ is
studied extensively. However, given that $\D(X,\sqcup)$ usually has
more than one minimal element, in our approach it is more natural to
consider the topology of the partition lattice after removing the maximal
element only. As a consequence, we will usually study $\redm{\D(X,\sqcup)}$.

In the following proposition, we see that $\D(X,\sqcup)$ is a lattice if we add an extra minimum.

\begin{proposition}
  \label{prop:latticeDecompositions}
  The poset $\D(X,\sqcup)$ is a lattice after adding an extra minimum element.
\end{proposition}

\begin{proof}
  We show that two $\sqcup$-decompositions $\sigma_1,\sigma_2$, bounded below by some element $\tau\in \D(X,\sqcup)$, have a meet in $\D(X,\sqcup)$.
  Indeed, since $\sigma_1,\sigma_2\in \D(X,\sqcup)_{\geq \tau} \cong \Pi(\tau)$ is a partition lattice, they have a meet in this poset, which we denote by $\sigma$.
  Let $e:2^{\tau} \to \S(X,\sqcup)$ be the lattice embedding of the Boolean lattice on $\tau$.
  If $x\in \sigma_1$ and $y\in \sigma_2$, then there exist unique sets $\tau_x,\tau_y\subseteq \tau$ such that $x = \bigvee_{a\in \tau_x} a$ and $y = \bigvee_{b\in \tau_y} b$, where the join is taken in the poset $\Sub(X,\sqcup)$.
  Hence $e(\tau_x) = x$, $e(\tau_y) = y$ and $e(\tau_x\cap \tau_y) = x\wedge y$.
  The decomposition $\sigma$ can be described as follows:
  \[ \sigma = \{e(\tau_x\cap \tau_y) \tq x\in \sigma_1,y\in \sigma_2\} \setminus \{0\}.\]
  But then, by definition of $e$,
  \[ \sigma = \{ x\wedge y \tq x\in \sigma_1,y\in \sigma_2\} \setminus \{0\},\]
  where the meet is taken in $\Sub(X,\sqcup)$.
  This description of $\sigma$ does not depend on $\tau$, and clearly $\sigma\geq \tau$.
  Therefore, $\sigma_1\wedge \sigma_2$ exists in $\D(X,\sqcup)$ and it is equal to the element $\sigma$ described above.
\end{proof}

The poset $\PD(X,\sqcup)$ is not a lattice in general, as it is seen in the poset of \ref{fig:weakDecompositions} in \ref{ex:weakDefinitionDecomposition}.
The next example also shows that weak decomposition and weak partial decomposition posets might not be lattices.
Here we invoke \ref{ex:latticeAsCategoryApproach}.

\begin{example}
\label{ex:nonLatticePD}
Let $G = S_4$ be the symmetric group on $4$ letters.
Then for $\S(G)$, the lattice of subgroups of $G$, we have that $\D_w(\S(G))$ is not a lattice after adding a minimal element, and hence nor is $\PD_w(\S(G))$.
This was shown by using GAP.
To be more concrete, we have weak decompositions
\[  \{ \gen{(3\,4)}, \gen{(2\,3)}, \gen{(1\,2)} \}, \quad \text{and} \quad \{ \gen{(3\,4)}, \gen{(2\,3)}, \gen{(1\,4)(2\,3)} \}, \]
for which their join in $\D_w(\S(G))$ does not exist since the elements above both of them in $\D_w(\S(G))$ are
\[ \{ G\},\,\, \{ \gen{ (3\,4), (1\,2)}, \gen{(2\,3), (1\,4)}\} ,\,\, \{ \gen{(1\,4)(2\,3), (1\,3)(2\,4), (3\,4)}, \gen{(2\,3)}\}. \]
Moreover, $\PD(\S(G))$ is not a lattice either since the join fails for the partial decompositions
\[ \{ \gen{(3\,4)} \}, \quad \text{and} \quad \{ \gen{(2\,3)}, \gen{(1\,2)} \}.\]
\end{example}

In view of our definition of decompositions in terms of heights, it will be useful to work with complements of elements which also have the correct height.

\begin{definition}
  \label{def:hcomplementedForPosets}
  Let $\S$ be a poset.
  \begin{enumerate}
    \item[$\bullet$] Let $y\in \S$ and suppose that $\S_{\leq y}$ is bounded of finite height.
      We say that $x\in \S_{\leq y}$ is $h$-complemented in $y$ if there exists $x'\leq y$ such that $x'\wedge x = 0$ and $x\vee x' = y$ (meet and join taken in $\S$) and  $h(x) + h(x') = h(y)$.
  \end{enumerate}
  \noindent
  Suppose now that $\S$ is bounded of finite height.
  \begin{enumerate}
    \item[$\bullet$] An element $x\in \S$ is called $h$-complemented if $x$ is $h$-complemented in $1_{\S}$.
      We denote by $\S_h$ the subposet of $h$-complemented elements.
    \item[$\bullet$] We say that $\S$ is $h$-complemented if $\S = \S_h$.

    \item[$\bullet$] If for every $x\leq y$ in $\S$ there exists a (unique) $h$-complement of $x$ in $y$,   
      then we say that $\S$ is (uniquely) downward $h$-complemented.
  \end{enumerate}
\end{definition}

Next, we define the analogous concepts in posets $\Sub(X,\sqcup)$ by taking into account the monoidal product $\sqcup$.
These concepts will play an important role later in \ref{sec:homotopyposets}.
For instance, the unique downward $(\sqcup,h)$-complementation property, as defined below, will arise in examples like posets of non-degenerate subspaces with the orthogonal complementation, or the Boolean lattice with the usual set complementation.
Recall that we are assuming that $\Sub(X,\sqcup)$ has finite height.

\begin{definition}
  \label{def:hcomplementedForObjects}
  Let $X$ and $\sqcup$ be as above.
  \begin{enumerate}
    \item[$\bullet$] If $z\leq y\in \Sub(X,\sqcup)$, a $\sch$-complement of $z$ in $y$ is an $h$-complement    $w$ of $z$ in $y$ such that $\{z,w\}$ is $\sqcup$-compatible in $\Sub(X,\sqcup)$.
      We denote by $\Sub_h(X,\sqcup)$ the subposet of $\Sub(X,\sqcup)$ whose elements have a $\sch$-complement in $\Sub(X,\sqcup)$.

    \item[$\bullet$] We say that $\Sub(X,\sqcup)$ is $\sch$-complemented if $\Sub_h(X,\sqcup) = \Sub(X,  
      \sqcup)$.

    \item[$\bullet$] We say that $\Sub(X,\sqcup)$ is (uniquely) downward $\sch$-complemented if for all   
      $z\leq y \in \Sub(X,\sqcup)$ there exists a (unique) $\sch$-complement of $z$ in $y$.
      In the case of uniqueness, we denote by $z^{\perp_y}$ the unique $\sch$-complement of $z$ in $y$, and by $z^\perp$ the unique $\sch$-complement of $z$ in $\Sub(X,\sqcup)$.
  \end{enumerate}
\end{definition}

Note that $\S_h$ (resp. $\Sub_h(X,\sqcup)$) contains the unique minimum element and those elements which belong to some $\sigma \in \D(\S)$ (reps. $\sigma\in \D(X,\sqcup)$).

\bigskip
\noindent
Even if $\Sub(X,\sqcup)$ is uniquely $\sch$-complemented, the complements in the purely poset-theoretic sense need not be unique.
For example, for the poset of non-degenerate subspaces of a finite-dimensional vector space equipped with a non-degenerate sesquilinear form and $\sqcup$ the orthogonal sum, complements and even $h$-complements of non-degenerate subspaces are not unique.
However, the poset is uniquely downward $\sch$-complemented, i.e., $\sch$-complements are unique.
See \ref{sub:formedSpaces} for more details.
\bigskip

The following lemma is an immediate consequence of the definitions.

\begin{lemma}
\label{lm:PhiSpanMap}
The map $\Phi : \PD(X,\sqcup) \to \Sub(X,\sqcup)$ sending $\sigma$
to $\Phi(\sigma) = \bigvee_{x\in\sigma}x$ is an order-preserving map with $\Phi^{-1}(0) = \{ \emptyset \}$ and  $\Phi^{-1}(1) = \D(X,\sqcup)$.
Its image is the poset $\Sub_h(X,\sqcup)$.
Thus $\Phi$ is surjective if and only if $\Sub(X,\sqcup)$ is $\sch$-complemented.
\end{lemma}

Similarly, for a bounded poset $\S$ of finite height and $x\in \S$, we have that $\{x\}$ is a partial decomposition of $\S$ if and only if $x\neq 0$ and it is $h$-complemented in $\S$.

We have the following necessary condition for the lattice property on $\PD(X,\sqcup)$.

\begin{lemma}
\label{lm:latticePDimpliesLatticeSh}
If $\PD(X,\sqcup)$ is a lattice, then $\Sub_h(X,\sqcup)$ is also a lattice.
\end{lemma}

\begin{proof}
Suppose that $x,y\in \Sub_h(X,\sqcup)$ have a lower bound $0\neq z\in \Sub_h(X,\sqcup)$.
Then $\{z\}\leq \{x\},\{y\}$ in $\PD(X,\sqcup)$.
Since $\PD(X,\sqcup)$ is a lattice, there exists $\sigma \in \PD(X,\sqcup)$ such that $\sigma = \{x\} \wedge \{y\}$, and therefore $\{z\}\leq \sigma$.
But this implies that $\Phi(\sigma)\leq x,y$, so we must have $\sigma = \{ \Phi(\sigma)\}$.
Hence $z \leq \Phi(\sigma)\in \Sub_h(X,\sqcup)$ and $x\wedge y = \Phi(\sigma)$.
\end{proof}

The converse of this lemma does not hold, as the lattice of \ref{ex:nonLatticePD} shows.

\bigskip

We will work with the following properties.
These are motivated by the question of whether a decomposition of a subobject of $X$ can be extended to a decomposition of $X$.
Essentially, they will give us natural ways of constructing new partial decompositions and will help us to establish stronger relations among the different posets and complexes, in particular when combined with downward $\sch$-complementation.
In well-known scenarios (cf. \ref{sec:examples}), these properties are usually easy to verify, or else they are consequences of known theorems.

\begin{definition}
\label{def:propsEXCM}
Let $X$ and $\sqcup$ be as above.
We define the following properties:

\smallskip

\begin{center}

\begin{tabular}{cc}
    \textbf{(EX)} & \pbox{12cm}{If $y \in \sigma \in \PD(X,\sqcup)$ and $\tau\in \D( \Sub(X,\sqcup)_{\leq y})$ is $\sqcup$-compatible in $\Sub(X,\sqcup)$, then $\left( \sigma \setminus \{y\} \right)\cup \tau \in \PD(X,\sqcup)$.}
\end{tabular}

\bigskip

\begin{tabular}{cc}
    \textbf{(CM)} & \pbox{12cm}{If $\{x,y\}, \{x',y'\}\in \D(X,\sqcup)$ are such that $x\leq x'$ and $y'\leq y$, then $x'\wedge y$ exists and $\{x'\wedge y,x,y'\}\setminus\{0\}\in \D(X,\sqcup)$.}
\end{tabular}

\end{center}

\smallskip

Here (EX) stands for the Extension property, and (CM) for the Complementation property.
For instance, we will see later in \ref{sub:vectorspaces} that in the case of vector spaces with the direct sum, properties (EX) and (CM) hold.
\end{definition}

\begin{remark}
\label{rk:CPprop}
Note that in the definition of property (EX), if $\sigma\in \D(X,\sqcup)$ is a decomposition, then $\left( \sigma \setminus \{y\} \right)\cup \tau \in \D(X,\sqcup)$ by height. 
Basically, property (EX) means that any time we take a partial decomposition, we can replace any of its subobjects with a full decomposition of this subobject.
In the vector space example, where partial decompositions of a given vector space $V$ are sets of subspaces in internal direct sum, this means that we can replace any subspace $S$ by an internal direct sum of $S$ and still obtain a partial decomposition of $V$.

On the other hand, in property (CM), $\{x'\wedge y, y'\}\in \D( \Sub(X,\sqcup)_{\leq y})$ since $(x'\wedge y) \vee y'$ exists, it is below $y$, and the height condition $h(x'\wedge y) + h(y') = h(1) - h(x) = h(y)$ guarantees $(x'\wedge y) \vee y' = y$.
Roughly, property (CM) says that given two decompositions of size two as above, we can easily construct their meet in the decomposition poset by just taking the meet (in the original poset) of the ``largest" elements.
For instance, this property clearly holds in the vector space example by the dimension theorem. See \ref{sub:vectorspaces} for more details.
\end{remark}

In general, property (EX) might not hold (not even for lattices).

\begin{example}
\begin{figure}
\centering
\scalebox{0.9}{\tikzpicture[x=+4143sp, y=+4143sp]
\newdimen\XFigu\XFigu4143sp
\newdimen\XFigu\XFigu4143sp
\clip(615,-5955) rectangle (4837,-2010);
\tikzset{inner sep=+0pt, outer sep=+0pt}
\pgfsetlinewidth{+7.5\XFigu}
\pgfsetcolor{black}
\filldraw  (900,-4050) circle [radius=+50];
\filldraw  (2700,-4005) circle [radius=+50];
\filldraw  (4050,-3150) circle [radius=+50];
\filldraw  (2700,-2250) circle [radius=+50];
\filldraw  (4545,-4005) circle [radius=+50];
\pgfsetlinewidth{+15\XFigu}
\filldraw  (1350,-3150) circle [radius=+50];
\pgfsetlinewidth{+7.5\XFigu}
\filldraw  (2700,-4950) circle [radius=+50];
\pgfsetlinewidth{+15\XFigu}
\draw (1350,-3150)--(2700,-4005);
\draw (2700,-2250)--(1350,-3150);
\draw (2700,-2250)--(4050,-3150);
\draw (2700,-4005)--(4050,-3150);
\draw (4050,-3150)--(4545,-4005);
\draw (1350,-3150)--(900,-4050);
\draw (2700,-4005)--(2700,-4950);
\draw (4545,-4005)--(2700,-4950);
\draw (900,-4050)--(2700,-4950);
\pgftext[base,left,at=\pgfqpointxy{4650}{-4050}] {\fontsize{12}{14.4}$c$}
\pgftext[base,left,at=\pgfqpointxy{4150}{-3050}] {\fontsize{12}{14.4}$e$}
\pgftext[base,left,at=\pgfqpointxy{1150}{-3050}] {\fontsize{12}{14.4}$d$}
\pgftext[base,left,at=\pgfqpointxy{2800}{-4150}] {\fontsize{12}{14.4}$b$}
\pgftext[base,left,at=\pgfqpointxy{2800}{-2160}] {\fontsize{12}{14.4}$1$}
\pgftext[base,left,at=\pgfqpointxy{2800}{-5080}] {\fontsize{12}{14.4}$0$}
\pgftext[base,left,at=\pgfqpointxy{680}{-4050}] {\fontsize{12}{14.4}$a$}
\endtikzpicture}
\vskip-1cm
\caption{Lattice failing properties (EX) and (CM)}
\label{fig:decomp}
\end{figure}
The poset $\S$ with Hasse diagram as in \ref{fig:decomp}
is easily seen to be a lattice. Let $X = \S$ and $\sqcup = \vee$ be the join operation.
Then $y = e \in \sigma = \{a,e\} \in \D(X,\sqcup) \subseteq \PD(X,\sqcup)$. We also have that
$\tau = \{b,c\} \in \D(X,\sqcup)_{\leq e})$, but 
$(\sigma \setminus \{y\}) \cup \tau = \{a,b,c\} \not\in \PD(X,\sqcup)$.
Thus $(X,\sqcup)$ does not satisfy (EX).

The property (CM) fails in this example, too. We have $\{x=a,y=e\},\{x'=d,y'=c\}\in \D(X,\sqcup)$ with $x=a\leq d=x'$ and $y'=c\leq e=y$, but 
$\{x'\wedge y,x,y'\} = \{d\wedge e, a, c\} = \{ b,a,c\}\notin \D(X,\sqcup)$.
\end{example}

We will see later that property (CM) fails in the case when $X$ is a finitely generated free group in the category of groups with $\sqcup$ being the free product.
See \ref{exa:freeGroupsFailCP}.

\begin{corollary}
\label{coro:uniquenessEPandCP}
If $\Sub(X,\sqcup)$ is uniquely downward $\sch$-complemented and property (EX) holds, then the following hold:
\begin{enumerate}
    \item the map $x\mapsto x^\perp$ is an anti-isomorphism of $\Sub(X,\sqcup)$;
    \item if $x\leq y\leq z$, then $\{x,x^{\perp_y},y^{\perp_z},z^\perp\}\setminus\{0\} \in \D(X,\sqcup)$ and $x^{\perp_y}\vee y^{\perp_z} = x^{\perp_z}$;
    \item property (CM) is satisfied.
\end{enumerate}
\end{corollary}

\begin{proof}
We first prove item 2.
The assertion is trivial for $z = 0$, so we assume that $z\neq 0$ and let $x\leq y\leq z$ and $\sigma = \{z,z^\perp\}\setminus\{0\} $.
Clearly $\sigma \in \D(X,\sqcup)$ and $\{y,y^{\perp_z}\}\setminus\{0\} \in \D(\Sub(X,\sqcup)_{\leq z})$ is $\sqcup$-compatible in $\Sub(X,\sqcup)$ by definition.
Hence, by property (EX), $\{z^\perp,y,y^{\perp_z}\}\setminus\{0\} \in \D(X,\sqcup)$.
Similarly, $\{x, x^{\perp_y}\}\setminus\{0\}  \in \D(\Sub(X,\sqcup)_{\leq y})$ and it is $\sqcup$-compatible in $\Sub(X,\sqcup)$, so $\{z^\perp,x, x^{\perp_y},y^{\perp_z}\}\setminus\{0\} \in \D(X,\sqcup)$.
Then $z = x\vee x^{\perp_y} \vee y^{\perp_z}$, and by uniqueness we conclude that $x^{\perp_z} =  x^{\perp_y}\vee y^{\perp_z}$.
This establishes item 2.
Moreover, by putting $z=1$ the maximal element, we see that $x^\perp = x^{\perp_y}\vee y^\perp$, i.e., $x^\perp \geq y^\perp$.
Thus, the map $x\mapsto x^\perp$ is an anti-isomorphism, which proves item 1.

For item 3, we can take decompositions $\{x,x^\perp\}$ and $\{x',x'^\perp\}$ where $x\leq x'$ (by uniqueness and item 1).
We show that $x'\wedge x^\perp = x^{\perp_{x'}}$.
Let $z\leq x^\perp,x'$.
By item 1, $z^\perp \geq x,x'^\perp$.
Since $\{x,x^{\perp_{x'}},x'^\perp\}\setminus\{0\}\in\D(X,\sqcup)$ by item 2, $x\vee x'^\perp$ exists and so $z^\perp \geq x\vee x'^\perp$.
Taking complements again by item 1 we see that $z\leq (x\vee x'^\perp)^\perp$, and the latter is equal to $x^{\perp_{x'}}$ by uniqueness, that is $z\leq x^{\perp_{x'}}$.
Therefore $x^\perp \wedge x'$ exists and it is $x^{\perp_{x'}}$.
Thus property (CM) holds.
\end{proof}

\section{Simplicial complexes and posets derived from decompositions}
\label{sec:derivedposets}

In this section, we define additional posets and complexes related to decompositions, most of them inspired by analogous constructions in specific cases.
We continue with our usual assumptions, so $X$ is an object of an ISM-category $\C$ with product $\sqcup$ and such that $\Sub(X,\sqcup)$ has finite height.
At the end of this section, the reader can find a table with the names of the combinatorial structures defined here, except for the ordered versions. See \ref{tab:objects}.

\subsection{Frames and partial bases}

\begin{definition}
[Frame]
  \label{def:frames} 
A \textit{partial frame} of $X$ is a $\sqcup$-partial decomposition $\sigma \in \PD(X,\sqcup)$ where
    all elements of $\sigma$ are atoms of $\Sub(X,\sqcup)$. We write $\PF(X,\sqcup)$ for the subposet of 
    $\PD(X,\sqcup)$ consisting of all partial frames of $X$. 

A partial frame $\sigma$ is called a \textit{full frame} of $X$ if $\sigma \in \D(X,\sqcup)$. 
    We write $\F(X,\sqcup)$ for the subposet of $\PF(X,\sqcup)$ consisting of partial frames
    contained in full frames.
\end{definition}

The following property follows immediately from the definition.

\begin{lemma} 
The posets $\F(X,\sqcup)\subseteq \PF(X,\sqcup)$ are lower intervals of $\PD(X,\sqcup)$.
Moreover, the refinement ordering $\leq$ in $\F(X,\sqcup)$ and $\PF(X,\sqcup)$ coincides with the inclusion ordering $\subseteq$, and hence these posets are simplicial complexes.  
\end{lemma}

Now we use frames to recover the idea of \textit{bases}.
We propose the following definition for the partial basis complexes, based on inflation complexes.

\begin{definition}
\label{def:inflationComplex}
Let $K$ be a simplicial complex with vertex set $\A$.
For every vertex $x$ of $K$, take a non-empty set $P_x$, and set $P = (P_x)_{x\in \A}$.
The inflation complex $(K,P)$ is the simplicial complex whose vertices are the pairs $(x,a)$, with $a\in P_x$, and simplices are sets of vertices $\{ (x_1,a_1), \ldots, (x_r,a_r)\}$ such that $x_i\neq x_j$ if $i\neq j$ and $\{x_1,\ldots, x_r\}$ is a simplex of $K$.

Write $p_{(K,P)}$ (or simply $p$) for the simplicial map $(K,P)\to K$ obtained by projecting the vertices of $(K,P)$ onto their first variable (the \textit{deflation map}).
\end{definition}

If each $P_{x}$ is finite then $(K,P)$ is an inflation of $K$ in the sense
of \cite[Section 6]{BWW}.
We will see that results analogous to those from \cite{BWW} hold in the general situation, where the sets $P_x$ are not necessarily finite or have the same size.

\begin{definition}
[Partial bases]
    \label{def:vogtmann}
    Let $\A$ be the set of vertices of $\PF(X,\sqcup)$ and
    $(P_x)_{x \in \A}$ be a collection of non-empty sets. 
    A simplex $B$ in the inflation complex $(\PF(X,\sqcup), P)$, is termed a \textit{partial basis} of $X$.
    We say that $B$ extends to a full basis if $p(B)\in \F(X,\sqcup)$.
    Denote by $\PVog(X,\sqcup,P) = (\PF(X,\sqcup), P)$ the \textit{complex of partial bases} with respect to the collection $P$, and by $\Vog(X,\sqcup,P)$ the complex of partial bases that extend to a full basis.
\end{definition}

Intuitively, we think of $P_x$ as a set of suitable bases of the atom $x$.
For example, if $x$ represents a $1$-dimensional vector space in the subspace poset, then $P_x$ can represent the set of non-zero vectors in $x$.
If $x$ represents a $1$-dimensional non-degenerate subspace in the non-degenerate subspace poset of a unitary space, we can take $P_x$ to be the set of normal vectors $v\in x$.

To simplify the notation, we will usually drop the $P$ from the notation and leave the collection implicit.
When the stated results strongly depend on the collection $P$, we will use the notation again.

\subsection{Ordered versions}
We define the ordered versions of the posets attached to $\Sub(X,\sqcup)$.

\begin{definition}
\label{def:orderedVersions}
Let $\T$ be one of the symbols $\D, \PD, \F, \PF, \Vog,\PVog$.
The \textit{ordered poset} $\Ord\T(X,\sqcup)$ consists of distinct-element tuples $(x_1,\ldots,x_r)$, $r\geq 0$, such that $\{x_1,\ldots,x_r\}\in \T(X,\sqcup)$.
The order in $\Ord\T(X,\sqcup)$ is given by order-preserving refinement, that is:
\[ (x_1,\ldots,x_r)\leq (y_1,\ldots, y_s) \quad \Leftrightarrow  \quad \forall \ 1\leq i\leq j\leq r, \ \exists \ k\leq l \tq x_i\leq y_k \text{ and } x_j\leq y_l. \]
We denote by $\Forget_\T:\Ord\T(X,\sqcup)\to \T(X,\sqcup)$ the forgetful poset map.
\end{definition}

\begin{remark} \label{rem:orderedcomplex}
For a simplicial complex $K$, let $\Ord K$ be the complex of injective words on $K$.
That is, $\Ord K$ is a set whose elements are tuples $z = (v_1,\ldots,v_r)$ consisting of distinct vertices of $K$ and such that $F_K(z) = \{v_1,\ldots,v_r\}$ is a simplex of $K$.
Similar to the previous definition, $\Ord K$ is a poset ordered by order-preserving inclusion, and we have a forgetful poset map $F_K:\Ord K\to K$.
However, $\Ord K$ is not a simplicial complex in general.

We see then that the poset $\Ord \T(X,\sqcup)$, for $\T\in \{\F,\PF, \Vog,\PVog\}$, is exactly the complex of injective words over $\T(X,\sqcup)$.
As a consequence, we will be able to invoke results from \cite{JW}. 
\end{remark}

The following example gives us some idea of the structure of the ordered versions:

\begin{example}
Let $X = \GF{2}^2$ be the $2$-dimensional vector space over the field with two elements.
For the monoidal product we take $\sqcup=\oplus$, the direct sum in the 
category of vector spaces over $\GF{2}$.
There are exactly three non-trivial proper subspaces, which we may denote by $U_1, U_2, U_{3}$.
Then $\Sub(X,\sqcup)=\Sub(V)$, the poset of subspaces of $V$, has the following structure:

\begin{figure}
\centering
\begin{tabular}{cc}
\tikzpicture[x=+4143sp, y=+4143sp]
\newdimen\XFigu\XFigu4143sp
\newdimen\XFigu\XFigu4143sp
\filldraw  (5400,0) circle [radius=+45];
\filldraw  (4500,-900) circle [radius=+45];
\filldraw  (5400,-900) circle [radius=+45];
\filldraw  (6300,-900) circle [radius=+45];
\filldraw  (5400,-1800) circle [radius=+45];
\draw (5400,0)--(4500,-900);
\draw (5400,0)--(5400,-900);
\draw (5400,-900)--(5400,-1800);
\draw (5400,0)--(6300,-900);
\draw (6300,-900)--(5400,-1800);
\draw (5400,-1800)--(4500,-900);
\pgftext[base,left,at=\pgfqpointxy{5355}{-2115}] {\fontsize{12}{14.4}$0$}
\pgftext[base,left,at=\pgfqpointxy{4185}{-780}] {\fontsize{12}{14.4}$U_1$}
\pgftext[base,left,at=\pgfqpointxy{5100}{-780}] {\fontsize{12}{14.4}$U_2$}
\pgftext[base,left,at=\pgfqpointxy{6300}{-780}] {\fontsize{12}{14.4}$U_3$}
\pgftext[base,left,at=\pgfqpointxy{5355}{180}] {\fontsize{12}{14.4}$V$}
\endtikzpicture

&

\scalebox{0.85}{
\tikzpicture[x=+4143sp, y=+4143sp]
\newdimen\XFigu\XFigu4143sp
\newdimen\XFigu\XFigu4143sp
\pgfdeclarearrow{
  name = xfiga2,
  parameters = {
    \the\pgfarrowlinewidth \the\pgfarrowlength \the\pgfarrowwidth\ifpgfarrowopen o\fi},
  defaults = {
	  line width=+7.5\XFigu, length=+120\XFigu, width=+60\XFigu},
  setup code = {
    \dimen7 2.6\pgfarrowlength\pgfmathveclen{\the\dimen7}{\the\pgfarrowwidth}
    \dimen7 2\pgfarrowwidth\pgfmathdivide{\pgfmathresult}{\the\dimen7}
    \dimen7 \pgfmathresult\pgfarrowlinewidth
    \pgfarrowssettipend{+\dimen7}
    \pgfarrowssetbackend{+-1.25\pgfarrowlength}
    \dimen9 -\pgfarrowlength\advance\dimen9 by-0.5\pgfarrowlinewidth
    \pgfarrowssetlineend{+\dimen9}
    \dimen9 -\pgfarrowlength\advance\dimen9 by-0.5\pgfarrowlinewidth
    \pgfarrowssetvisualbackend{+\dimen9}
    \pgfarrowshullpoint{+\dimen7}{+0pt}
    \pgfarrowsupperhullpoint{+-1.25\pgfarrowlength}{+0.5\pgfarrowwidth}
    \pgfarrowssavethe\pgfarrowlinewidth
    \pgfarrowssavethe\pgfarrowlength
    \pgfarrowssavethe\pgfarrowwidth
  },
  drawing code = {\pgfsetdash{}{+0pt}
    \ifdim\pgfarrowlinewidth=\pgflinewidth\else\pgfsetlinewidth{+\pgfarrowlinewidth}\fi
    \pgfpathmoveto{\pgfqpoint{-1.25\pgfarrowlength}{-0.5\pgfarrowwidth}}
    \pgfpathlineto{\pgfqpoint{0pt}{0pt}}
    \pgfpathlineto{\pgfqpoint{-1.25\pgfarrowlength}{0.5\pgfarrowwidth}}
    \pgfpathlineto{\pgfqpoint{-\pgfarrowlength}{0pt}}
    \pgfpathclose
    \ifpgfarrowopen\pgfusepathqstroke\else\pgfsetfillcolor{pgfstrokecolor}
	\ifdim\pgfarrowlinewidth>0pt\pgfusepathqfillstroke\else\pgfusepathqfill\fi\fi
  }
}
\definecolor{green1}{rgb}{0,0.56,0}
\definecolor{brown3}{rgb}{0.75,0.38,0}
\tikzset{inner sep=+0pt, outer sep=+0pt}
\pgfsetlinewidth{+7.5\XFigu}
\pgfsetcolor{black}
\filldraw  (1485,-90) circle [radius=+45];
\filldraw  (585,-990) circle [radius=+45];
\filldraw  (1485,-990) circle [radius=+45];
\filldraw  (2385,-990) circle [radius=+45];
\filldraw  (1485,-1890) circle [radius=+45];
\filldraw  (1485,855) circle [radius=+45];
\filldraw  (585,-135) circle [radius=+45];
\filldraw  (2385,-90) circle [radius=+45];
\draw (1485,-90)--(585,-990);
\draw (1485,-990)--(1485,-1890);
\draw (1485,-90)--(2385,-990);
\draw (2385,-990)--(1485,-1890);
\draw (1485,-1890)--(585,-990);
\draw (1485,810)--(1485,-90);
\draw (585,-135)--(1485,-990);
\draw (1485,-990)--(2385,-90);
\draw (2385,-90)--(2385,-990);
\draw (585,-135)--(585,-990);
\draw (585,-135)--(1485,855);
\draw (1485,855)--(2385,-90);
\pgftext[base,left,at=\pgfqpointxy{1440}{-2205}] {\fontsize{12}{14.4}$\emptyset$}
\pgftext[base,left,at=\pgfqpointxy{150}{-1170}] {\fontsize{12}{14.4}$\{U_1\}$}
\pgftext[base,left,at=\pgfqpointxy{1050}{-1170}] {\fontsize{12}{14.4}$\{U_2\}$}
\pgftext[base,left,at=\pgfqpointxy{2430}{-1215}] {\fontsize{12}{14.4}$\{U_3\}$}
\pgftext[base,left,at=\pgfqpointxy{1550}{-90}] 
{\fontsize{12}{14.4}$\{U_1,U_3\}$}
\pgftext[base,left,at=\pgfqpointxy{2500}{-90}] 
{\fontsize{12}{14.4}$\{U_2,U_3\}$}
\pgftext[base,left,at=\pgfqpointxy{-135}{-45}] 
{\fontsize{12}{14.4}$\{U_1,U_2\}$}
\pgftext[base,left,at=\pgfqpointxy{1300}{990}] 
{\fontsize{12}{14.4}$\{V\}$}
\endtikzpicture%
}
\end{tabular}

\bigskip

\scalebox{0.85}{
\tikzpicture[x=+4143sp, y=+4143sp]
\newdimen\XFigu\XFigu4143sp
\newdimen\XFigu\XFigu4143sp
\pgfdeclarearrow{
  name = xfiga2,
  parameters = {
    \the\pgfarrowlinewidth \the\pgfarrowlength \the\pgfarrowwidth\ifpgfarrowopen o\fi},
  defaults = {
	  line width=+7.5\XFigu, length=+120\XFigu, width=+60\XFigu},
  setup code = {
    \dimen7 2.6\pgfarrowlength\pgfmathveclen{\the\dimen7}{\the\pgfarrowwidth}
    \dimen7 2\pgfarrowwidth\pgfmathdivide{\pgfmathresult}{\the\dimen7}
    \dimen7 \pgfmathresult\pgfarrowlinewidth
    \pgfarrowssettipend{+\dimen7}
    \pgfarrowssetbackend{+-1.25\pgfarrowlength}
    \dimen9 -\pgfarrowlength\advance\dimen9 by-0.5\pgfarrowlinewidth
    \pgfarrowssetlineend{+\dimen9}
    \dimen9 -\pgfarrowlength\advance\dimen9 by-0.5\pgfarrowlinewidth
    \pgfarrowssetvisualbackend{+\dimen9}
    \pgfarrowshullpoint{+\dimen7}{+0pt}
    \pgfarrowsupperhullpoint{+-1.25\pgfarrowlength}{+0.5\pgfarrowwidth}
    \pgfarrowssavethe\pgfarrowlinewidth
    \pgfarrowssavethe\pgfarrowlength
    \pgfarrowssavethe\pgfarrowwidth
  },
  drawing code = {\pgfsetdash{}{+0pt}
    \ifdim\pgfarrowlinewidth=\pgflinewidth\else\pgfsetlinewidth{+\pgfarrowlinewidth}\fi
    \pgfpathmoveto{\pgfqpoint{-1.25\pgfarrowlength}{-0.5\pgfarrowwidth}}
    \pgfpathlineto{\pgfqpoint{0pt}{0pt}}
    \pgfpathlineto{\pgfqpoint{-1.25\pgfarrowlength}{0.5\pgfarrowwidth}}
    \pgfpathlineto{\pgfqpoint{-\pgfarrowlength}{0pt}}
    \pgfpathclose
    \ifpgfarrowopen\pgfusepathqstroke\else\pgfsetfillcolor{pgfstrokecolor}
	\ifdim\pgfarrowlinewidth>0pt\pgfusepathqfillstroke\else\pgfusepathqfill\fi\fi
  }
}
\definecolor{green1}{rgb}{0,0.56,0}
\definecolor{brown3}{rgb}{0.75,0.38,0}
\tikzset{inner sep=+0pt, outer sep=+0pt}
\pgfsetlinewidth{+7.5\XFigu}
\pgfsetcolor{black}
\filldraw  (4770,-990) circle [radius=+45];
\filldraw  (5670,-990) circle [radius=+45];
\filldraw  (6570,-990) circle [radius=+45];
\filldraw  (5670,-1890) circle [radius=+45];
\filldraw  (5670,855) circle [radius=+45];
\filldraw  (4275,-135) circle [radius=+45];
\filldraw  (7110,-90) circle [radius=+45];
\filldraw  (4770,-135) circle [radius=+45];
\filldraw  (6570,-90) circle [radius=+45];
\filldraw  (6030,-135) circle [radius=+45];
\filldraw  (5355,-135) circle [radius=+45];
\draw (5355,-135)--(4770,-990);
\draw (5670,-990)--(5670,-1890);
\draw (5355,-135)--(6570,-990);
\draw (6570,-990)--(5670,-1890);
\draw (5670,-1890)--(4770,-990);
\draw (5670,-990)--(4275,-135);%
\draw (5670,-990)--(7110,-90);
\draw (5670,810)--(5355,-135);
\draw (4770,-135)--(5670,-990);
\draw (5670,-990)--(6570,-90);
\draw (6570,-90)--(6570,-990);
\draw (4770,-135)--(4770,-990);
\draw (4770,-135)--(5670,855);
\draw (5670,855)--(6570,-90);
\draw (4275,-135)--(4770,-990);
\draw (4275,-135)--(5670,855);
\draw (7110,-90)--(6570,-990);
\draw (5670,855)--(7110,-90);
\draw (5670,855)--(6030,-135);
\draw (6030,-135)--(4770,-990);
\draw (6030,-135)--(6570,-990);
\pgfsetbeveljoin
\pgfsetarrows{[line width=7.5\XFigu]}
\pgfsetarrowsend{xfiga2}
\pgfsetstrokecolor{green1}
\pgfsetfillcolor{white}
\filldraw (4365,270)--(4364,266)--(4363,259)--(4362,246)--(4359,228)--(4356,206)--(4354,181)
  --(4351,155)--(4350,131)--(4350,108)--(4351,88)--(4353,71)--(4358,57)--(4365,45)
  --(4374,34)--(4386,25)--(4400,17)--(4416,10)--(4433,3)--(4452,-3)--(4471,-8)
  --(4491,-14)--(4510,-19)--(4529,-24)--(4546,-29)--(4562,-35)--(4577,-40)--(4590,-45)
  --(4607,-51)--(4624,-57)--(4642,-63)--(4661,-70)--(4680,-76)--(4698,-81)--(4725,-90);
\filldraw (6525,585)--(6524,583)--(6523,578)--(6521,568)--(6517,555)--(6513,537)--(6507,516)
  --(6501,491)--(6494,465)--(6487,438)--(6480,411)--(6473,386)--(6466,362)--(6460,340)
  --(6453,320)--(6447,302)--(6441,285)--(6435,270)--(6428,253)--(6420,237)--(6413,222)
  --(6405,207)--(6397,192)--(6388,178)--(6380,163)--(6372,149)--(6363,135)--(6355,121)
  --(6346,108)--(6337,94)--(6328,81)--(6319,69)--(6310,57)--(6300,45)--(6288,33)
  --(6276,20)--(6261,8)--(6245,-5)--(6227,-18)--(6208,-31)--(6188,-45)--(6169,-58)
  --(6152,-69)--(6120,-90);
\filldraw (4860,540)--(4860,539)--(4861,536)--(4862,529)--(4865,516)--(4868,499)--(4872,478)
  --(4877,457)--(4881,436)--(4887,416)--(4892,397)--(4898,378)--(4905,360)--(4911,346)
  --(4917,330)--(4925,314)--(4932,298)--(4940,280)--(4948,262)--(4956,244)--(4964,225)
  --(4973,206)--(4981,188)--(4990,170)--(4999,152)--(5009,136)--(5019,120)--(5029,104)
  --(5040,90)--(5052,76)--(5065,63)--(5079,49)--(5096,34)--(5114,19)--(5135,4)
  --(5156,-12)--(5178,-29)--(5199,-44)--(5219,-58)--(5235,-69)--(5265,-90);
\filldraw (7065,270)--(7064,267)--(7062,261)--(7058,252)--(7052,238)--(7045,221)--(7036,202)
  --(7028,182)--(7019,163)--(7010,145)--(7002,129)--(6993,114)--(6984,102)--(6975,90)
  --(6965,79)--(6954,68)--(6942,58)--(6929,47)--(6915,37)--(6901,26)--(6887,16)
  --(6873,6)--(6859,-4)--(6845,-13)--(6831,-22)--(6819,-30)--(6807,-38)--(6795,-45)
  --(6779,-53)--(6763,-60)--(6746,-67)--(6727,-72)--(6707,-78)--(6689,-83)--(6660,-90);
\pgfsetfillcolor{black}
\pgftext[base,left,at=\pgfqpointxy{5610}{-2185}] {\fontsize{12}{14.4}$()$}
\pgftext[base,left,at=\pgfqpointxy{4365}{-1170}] {\fontsize{12}{14.4}$(U_1)$}
\pgftext[base,left,at=\pgfqpointxy{5265}{-1170}] {\fontsize{12}{14.4}$(U_2)$}
\pgftext[base,left,at=\pgfqpointxy{6615}{-1215}] {\fontsize{12}{14.4}$(U_3)$}
\pgftext[base,left,at=\pgfqpointxy{3550}{-90}] 
{\fontsize{12}{14.4}$(U_1,U_2)$}
\pgftext[base,left,at=\pgfqpointxy{7155}{-45}] 
{\fontsize{12}{14.4}$(U_2,U_3)$}
\pgftext[base,left,at=\pgfqpointxy{3825}{330}] 
{\fontsize{12}{14.4}$(U_2,U_1)$}
\pgftext[base,left,at=\pgfqpointxy{6165}{655}] 
{\fontsize{12}{14.4}$(U_1,U_3)$}
\pgftext[base,left,at=\pgfqpointxy{4400}{655}] 
{\fontsize{12}{14.4}$(U_3,U_1)$}
\pgftext[base,left,at=\pgfqpointxy{5580}{990}] 
{\fontsize{12}{14.4}$V$}
\pgftext[base,left,at=\pgfqpointxy{6795}{330}] 
{\fontsize{12}{14.4}$(U_3,U_2)$}
\endtikzpicture%
}

\medskip

 \caption{Poset of subspaces of $V = \GF{2}^2$ (top left), poset of 
 partial decompositions of $V$ (top right) and poset of ordered partial decompositions of $V$ (bottom). Here $U_1 = \gen{(1,0)}$, $U_2 = \gen{(1,1)}$, and $U_3 = \gen{(0,1)}$}
    \label{fig:GF(2)dim2}
\end{figure}

From the Hasse diagram of $\Sub(V)$ we see that there are three full decompositions
\[ \{U_1,U_2\}, \ \{U_1, U_{3}\}, \ \{U_2, U_{3}\}.\]
From this, one easily determines the partial decomposition poset $\PD(X,\sqcup)=\PD(V)$ and its ordered version.
See \ref{fig:GF(2)dim2}.
\end{example}

\begin{example}
\label{ex:orderedPartitions}
Let $X = \{1,\ldots, n\}$ be in the category of sets with $\sqcup = \coprod$ the disjoint union.
Then $\Sub(X,\sqcup)$ is the Boolean lattice, $\D(X,\sqcup) = \Pi(n)$ is the partition lattice, and $\PD(X,\sqcup)$ is the partial partition lattice.
The ordered partial partition poset $\O\PD(X,\sqcup)$ for the case $n=3$ is given in \ref{fig:ordpart3}. 
The blue subposet shows the ordered partition poset
$\O\D(X,\sqcup)$.

Proposition 1.4 of \cite{BS} states that the ordered partition lattice is the face lattice of the permutohedron $P_n$ of order $n$. 
\end{example}

\begin{figure}
\centering
\tikzpicture[x=+4143sp, y=+4143sp]
\newdimen\XFigu\XFigu4143sp
\newdimen\XFigu\XFigu4143sp
\definecolor{blue3}{rgb}{0,0,0.82}
\definecolor{brown3}{rgb}{0.75,0.38,0}
\clip(-1095,-2220) rectangle (4255,2595);
\tikzset{inner sep=+0pt, outer sep=+0pt}
\pgfsetlinewidth{+7.5\XFigu}

\pgfsetstrokecolor{black}
\draw (1125,-135)--(585,-990);
\draw (1485,-990)--(1485,-1890);
\draw (1125,-135)--(2385,-990);
\draw (2385,-990)--(1485,-1890);
\draw (1485,-1890)--(585,-990);
\draw (0,-135)--(1485,-990);
\draw (1485,-990)--(2430,-135);
\draw (2430,-135)--(2385,-990);
\draw (0,-135)--(585,-990);
\draw (585,-135)--(585,-990);
\draw (585,-135)--(1485,-990);
\draw (1845,-135)--(585,-990);
\draw (1845,-135)--(2385,-990);
\draw (1485,-990)--(3105,-135);
\draw (3105,-135)--(2385,-990);
\draw (-630,675)--(0,-135);
\draw (-630,675)--(1125,-135);
\draw (-630,675)--(2430,-135);
\draw (-90,675)--(0,-135);
\draw (-90,675)--(585,-135);
\draw (495,675)--(585,-135);
\draw (495,675)--(1125,-135);
\draw (495,675)--(2430,-135);
\draw (1080,675)--(1125,-135);
\draw (0,-135)--(1080,675);
\draw (1080,675)--(3105,-135);
\draw (1485,675)--(1125,-135);
\draw (1485,675)--(1845,-135);
\draw (1935,675)--(1845,-135);
\draw (0,-135)--(1935,675);
\draw (1485,675)--(3105,-135);
\draw (1845,-135)--(2610,675);
\draw (2610,675)--(2430,-135);
\draw (585,-135)--(2610,675);
\draw (2430,-135)--(3240,675);
\draw (3240,675)--(3105,-135);
\draw (3735,675)--(3105,-135);
\draw (-315,1530)--(-90,675);
\draw (495,1530)--(-90,675);
\draw (1170,1530)--(1485,675);
\draw (1485,675)--(1890,1530);
\draw (3240,675)--(2565,1530);
\draw (3240,1530)--(3240,675);
\pgfsetstrokecolor{blue3}
\draw (-630,675)--(2565,1530);
\draw (495,675)--(1890,1530);
\draw (1080,675)--(2565,1530);
\draw (-315,1530)--(1575,2385);
\draw (-315,1530)--(-630,675);
\draw (-315,1530)--(495,675);
\draw (1170,1530)--(1575,2385);
\draw (495,1530)--(1575,2385);
\draw (495,1530)--(1935,675);
\draw (495,1530)--(3735,675);
\draw (1170,1530)--(1080,675);
\draw (1170,1530)--(1935,675);
\draw (1575,2385)--(1890,1530);
\draw (1575,2385)--(2565,1530);
\draw (1575,2385)--(3240,1530);
\draw (3240,1530)--(3735,675);
\draw (2610,675)--(3240,1530);
\draw (2610,675)--(1890,1530);

\pgfsetcolor{black}
\filldraw  (585,-990) circle [radius=+45];
\filldraw  (1485,-990) circle [radius=+45];
\filldraw  (2385,-990) circle [radius=+45];
\filldraw  (1485,-1890) circle [radius=+45];
\filldraw  (0,-135) circle [radius=+45];
\filldraw  (585,-135) circle [radius=+45];
\filldraw  (1845,-135) circle [radius=+45];
\filldraw  (2430,-135) circle [radius=+45];
\filldraw  (3105,-135) circle [radius=+45];
\filldraw  (1125,-135) circle [radius=+45];
\filldraw  (-90,675) circle [radius=+45];
\filldraw  (1485,675) circle [radius=+45];
\filldraw  (3240,675) circle [radius=+45];
\pgfsetstrokecolor{blue3}
\pgfsetfillcolor{blue3!85!black}
\filldraw  (1575,2385) circle [radius=+45];
\filldraw  (-315,1530) circle [radius=+45];
\filldraw  (495,1530) circle [radius=+45];
\filldraw  (1170,1530) circle [radius=+45];
\filldraw  (1890,1530) circle [radius=+45];
\filldraw  (2565,1530) circle [radius=+45];
\filldraw  (3240,1530) circle [radius=+45];
\filldraw  (3735,675) circle [radius=+45];
\filldraw  (2610,675) circle [radius=+45];
\filldraw  (1935,675) circle [radius=+45];
\filldraw  (1080,675) circle [radius=+45];
\filldraw  (495,675) circle [radius=+45];
\filldraw  (-630,675) circle [radius=+45];

\pgfsetfillcolor{black}
\pgftext[base,left,at=\pgfqpointxy{1440}{-2205}] {\fontsize{12}{14.4}$\emptyset$}
\pgftext[base,left,at=\pgfqpointxy{405}{-1125}] {\fontsize{12}{14.4}1}
\pgftext[base,left,at=\pgfqpointxy{1305}{-1125}] {\fontsize{12}{14.4}2}
\pgftext[base,left,at=\pgfqpointxy{2475}{-1125}] {\fontsize{12}{14.4}3}
\pgftext[base,left,at=\pgfqpointxy{-360}{-180}] {\fontsize{12}{14.4}1\textbar 2}
\pgftext[base,left,at=\pgfqpointxy{270}{-180}] {\fontsize{12}{14.4}2\textbar 1}
\pgftext[base,left,at=\pgfqpointxy{810}{-180}] {\fontsize{12}{14.4}1\textbar 3}
\pgftext[base,left,at=\pgfqpointxy{1530}{-180}] {\fontsize{12}{14.4}3\textbar 1}
\pgftext[base,left,at=\pgfqpointxy{2520}{-180}] {\fontsize{12}{14.4}2\textbar 3}
\pgftext[base,left,at=\pgfqpointxy{3195}{-180}] {\fontsize{12}{14.4}3\textbar 2}
\pgftext[base,left,at=\pgfqpointxy{-360}{630}] {\fontsize{12}{14.4}12}
\pgftext[base,left,at=\pgfqpointxy{1215}{630}] {\fontsize{12}{14.4}13}
\pgftext[base,left,at=\pgfqpointxy{3330}{630}] {\fontsize{12}{14.4}23}
\pgfsetfillcolor{blue3}
\pgftext[base,left,at=\pgfqpointxy{1215}{2430}] {\fontsize{12}{14.4}123}
\pgftext[base,left,at=\pgfqpointxy{90}{1485}] 
{\fontsize{12}{14.4}3\textbar 12}
\pgftext[base,left,at=\pgfqpointxy{-720}{1485}] {\fontsize{12}{14.4}12\textbar 3}
\pgftext[base,left,at=\pgfqpointxy{-1080}{630}] {\fontsize{12}{14.4}1\textbar 2\textbar 3}
\pgftext[base,left,at=\pgfqpointxy{45}{630}] 
{\fontsize{12}{14.4}2\textbar 1\textbar 3}
\pgftext[base,left,at=\pgfqpointxy{630}{630}] 
{\fontsize{12}{14.4}1\textbar 3\textbar 2}
\pgftext[base,left,at=\pgfqpointxy{1980}{1485}] {\fontsize{12}{14.4}2\textbar 13}
\pgftext[base,left,at=\pgfqpointxy{2655}{1485}] {\fontsize{12}{14.4}1\textbar 23}
\pgftext[base,left,at=\pgfqpointxy{3330}{1485}] {\fontsize{12}{14.4}23\textbar 1}
\pgftext[base,left,at=\pgfqpointxy{3825}{630}] {\fontsize{12}{14.4}3\textbar 2\textbar 1}
\pgftext[base,left,at=\pgfqpointxy{2745}{630}] {\fontsize{12}{14.4}2\textbar 3\textbar 1}
\pgftext[base,left,at=\pgfqpointxy{2025}{630}] {\fontsize{12}{14.4}3\textbar 1\textbar 2}
\pgftext[base,left,at=\pgfqpointxy{765}{1485}] {\fontsize{12}{14.4}13\textbar 2}
\endtikzpicture%
    \caption{Ordered partial and full partition poset}
\label{fig:ordpart3}
\end{figure}

Since we aim to relate later the homotopy type of $\Ord \T(X,\sqcup)$ with that of $\T(X,\sqcup)$, in the following lemmas we take a closer look at the fibers by the canonical forgetful maps and the intervals in the ordered versions. 

\begin{lemma}
\label{lm:fibersForgetfulMap}
The following hold:
\begin{enumerate}
    \item $\Forget_\D^{-1}( \D(X,\sqcup)_{\geq \sigma} )$ is the ordered partition lattice on the set $\sigma$, which is the face poset of a permutohedron of order $|\sigma|$.
    Thus $\Forget_\D^{-1}( \redm{\D(X,\sqcup)}_{\geq \sigma} )$ is a sphere of dimension $|\sigma|-2$ (and also Cohen-Macaulay).
    
    \item If $\sigma \in \red{\PD(X,\sqcup)}\setminus \redm{\D(X,\sqcup)}$ then $\Forget_{\PD}^{-1}\big( \ (\red{\PD(X,\sqcup)})_{\geq \sigma} \ \big)$ is contractible.

    \item For a simplicial complex $K$ and $\sigma\in K$,  $\Forget_{K}^{-1}( K_{\leq \sigma} )$ is the ordered version of the Boolean lattice, namely the poset of injective words.
    In particular, $\Forget_{K}^{-1}( K_{\leq \sigma} )$ is shellable of dimension $|\sigma|-1$ and the number of spheres is the number of derangements on $|\sigma|$ points:
    \[ \sum_{i=0}^{|\sigma|} (-1)^{i} \frac{|\sigma|!}{i!}.\]
\end{enumerate}
\end{lemma}

\begin{proof}
Since $\D(X,\sqcup)_{\geq \sigma}$ is the partition lattice on the set $\sigma$ by \ref{lm:upperIntervalsDecomp}, it is clear that 
$\Forget_\D^{-1}( \D(X,\sqcup)_{\geq \sigma} )$ is the ordered partition lattice on the set $\sigma$.
Now, by Proposition 1.4 of \cite{BS}, this is the face lattice of the permutahedron of order $|\sigma|$.
The second part of item 1 now follows since the order complex of the face lattice of a polytope is a sphere.

Item 3 follows similarly, and the shellability is implied by Theorem 6.1 of \cite{BW}.

For item 2, we construct a homotopy between the identity and a constant map.
Since $\sigma$ is a non-empty $\sqcup$-partial decomposition (not full), there exists $x = \bigvee_{y\in \sigma} y$ and this is different from $1$ and $0$.
Let $\tau\in \Forget^{-1}_{\PD}( (\red{\PD(X,\sqcup)})_{\geq \sigma})$, and write $\tau = (z_1,\ldots,z_s)$.
Define
\[ z_i' = \bigvee_{y\in \sigma \tq y\leq z_i} y.\]
and let $r(\tau)$ be the tuple $(z_1',\ldots,z_s')$ after removing the $0$ elements.
We see that $\Forget_{\PD}(r(\tau))$ is a $\sqcup$-partial decomposition obtained by joining elements of $\sigma$, so $\Forget_{\PD}(r(\tau))\geq \sigma$.
Thus we have well-defined order-preserving maps
\[ \tau \geq r(\tau) \leq (x).\]
This implies that $\Forget^{-1}_{\PD}( (\red{\PD(X,\sqcup)})_{\geq\sigma})$ is contractible.
\end{proof}

We describe now some intervals in the ordered versions.
Recall property (LI) from \ref{def:propLI}.
Here $|z|$ denotes the size of an ordered tuple $z$, that is, the number of (distinct) elements of the tuple.

\begin{lemma}
\label{lm:intervalsOrdered}
    The following hold:
    \begin{enumerate}
        \item For $z\in \Ord\D(X,\sqcup)$, $\Ord\D(X,\sqcup)_{\geq z}$ is the Boolean lattice on a set of size $|z|-1$.
        In particular, if $\D(X,\sqcup)$ is finite then
        \[\tilde{\chi}\left(\redm{\Ord\D(X,\sqcup)}\right) = \sum_{k\geq 1} (-1)^{k} \cdot k! \cdot \# \{\sigma\in \D(X,\sqcup) \tq |\sigma| = k\}.\]
        \item Suppose that property (LI) holds.
        Then for $z\in \Ord\PD(X,\sqcup)$ we have
        \[ 
        \Ord\PD(X,\sqcup)_{\leq z} \cong \prod_{y\in z} \Ord \PD(Y,\sqcup),
        \]
        and if $z\in \Ord\D(X,\sqcup)$ then also
        \[ 
        \Ord\D(X,\sqcup)_{\leq z} \cong \prod_{y\in z} \Ord \D(Y,\sqcup).
        \]
        Here $(Y,i_y)\in y$ denote suitable representatives for $y\in z$.
        \item Let $K$ be a simplicial complex and $z\in \Ord K$. Then $\Ord K_{>z} \cong \Ord\Lk_{K}(z)$ and $\Ord K_{\leq z}$ is the Boolean lattice on $F_{K}(z)$.
    \end{enumerate}
\end{lemma}

\begin{proof}
These are straightforward.
\end{proof}

Finally, we establish a connection between the ordered version of the poset of decompositions and the subposet of $\sch$-complemented subobjects.

\begin{definition}
\label{def:GMapOrdDtoS}
Let $G:\red{\Ord\PD(X,\sqcup)} \to {(\Delta \red{\Sub_h(X,\sqcup)})}^{\op} \cup \{\emptyset\}$ be the function
\[G(x_1,\ldots,x_r) = \{x_1 < x_1\vee x_2 < \cdots < x_1\vee\cdots \vee x_{r-1} \}.\]
\end{definition}

It is clear that if $\{x_1,\ldots,x_r\}$ is a $\sqcup$-partial decomposition, then any join of its elements is $\sch$-complemented, and hence belongs to $\Sub_h(X,\sqcup)$.
Thus $G$ is a well-defined function.
However, $G$ is not order-preserving for partial decompositions: for example, if $x\leq y$ and $\{x,z\}, \{y,z\}$ are $\sqcup$-partial decompositions, then $(x,z)\leq (y,z)$ but $G(x,z) = \{x\}\not\subseteq \{y\} = G(y,z) $.
Nevertheless, this is a poset map when restricted to decompositions:

\begin{lemma}
\label{lm:GMapOrdDtoS}
The restriction $G: \redm{\Ord\D(X,\sqcup)}\to {\left(\Delta \big( \red{\Sub_h(X,\sqcup)}\big) \right)}^{\op}$ is an order-preserving map.
\end{lemma}

\begin{proof}
This follows since elements of $\Ord\D(X,\sqcup)_{>\tau}$ are obtained by joining adjacent blocks of an ordered decomposition $\tau$.
That is, if $\sigma>\tau$ and $\tau = (x_1,\ldots,x_r)$, we can write
\[ \sigma = \bigg( {\bigvee_{i=1}^{t_1}}x_i, \bigvee_{i=t_1+1}^{t_2} x_i,\ldots, \bigvee_{i=t_{s-1}+1}^{r}x_i\bigg)\]
for suitable $1 \leq t_1 < t_2 < \cdots < t_s = r$.
Thus $G(\tau)$ contains $G(\sigma)$.
\end{proof}

We will see later that this map turns out to be an isomorphism under unique complementation and suitable extension properties.
See \ref{lm:isoGMap} and \ref{lm:GMapSurjective}.

\subsection{The augmented Bergman complex}

We introduce now a (simplified) version of the augmented Bergman complex in our context.
This complex was first defined for matroids and it 
gained substantial attention due to the work \cite{BHMPW} where its Stanley-Reisner ring is a key object in the proofs of the Top-Heavy conjecture (raised by Dowling and Wilson), and the non-negativity of the coefficients of Kazhdan-Lusztig polynomials for all matroids.
Its topology in the matroid case was carefully analyzed in \cite{BKR}.


\begin{definition}
[Bergman]
  \label{def:bergman}
  Take disjoint copies  
   of the vertices of $\F(X,\sqcup)$ and the atoms of $\Sub(X,\sqcup)$. 
   
  In $\Bergman(X,\sqcup)$ we collect all sets
  $\sigma \cup \{ x_1 \lneq x_2 \lneq \cdots \lneq x_r\}$ where 
  $\sigma \in \F(X,\sqcup)$ and $\Phi(\sigma) \leq x_1,\ldots , x_r \in \redm{\Sub(X,\sqcup)}$. 
  We call $\Bergman(X,\sqcup)$ the \textit{augmented Bergman complex} of $X$. 
\end{definition}

It is a direct consequence of the definition that the augmented Bergman complex $\Bergman(X,\sqcup)$ is indeed a simplicial complex.
Note that if $\sigma\in\F(X,\sqcup)$ is a full frame, then $\sigma\in \Bergman(X,\sqcup)$ is a maximal simplex.
Also note that $0 \in \Sub(X,\sqcup)$ is a vertex of $\Bergman(X,\sqcup)$, and a simplex containing $0$ must be a flag of $\redm{\Sub(X,\sqcup)}$.

\subsection{The Charney poset}

A different version of the ordered decompositions was introduced by R. Charney in \cite{Charney}.
There she works with the poset of summands of a finitely generated free module over a Dedekind domain, and defines a poset whose elements are ordered direct sum decompositions of size two with a zigzag ordering.
Another version of Charney's poset appears in the context of split buildings for spaces with a certain form, and we interpret decompositions of size two there as direct sum decomposition pairs $(V_1,V_2)$ such that $V_1$ and the orthogonal complement of $V_2$ are totally isotropic subspaces (see \cite{LR}).

We adapt the definition of this poset to our context.

\begin{definition}
\label{def:charney}
The Charney poset associated with $X$ is the poset $\Charney(X,\sqcup)$ whose underlying set consists of ordered decompositions of size two:
\[ \Charney(X,\sqcup) =  \{ z \in \OD(X,\sqcup) \tq |z| = 2\},\]
and the ordering is given by
\[ (x,y)\leq (x',y') \quad \text{ if } \quad  x\leq x' \text{ and } y\geq y'.\]
\end{definition}

The following proposition relates the ordered decompositions to the Charney complex.

\begin{proposition}
\label{prop:charney}
Let $\beta$ be the function
\[ \beta :\left\{ \begin{array}{ccc} {\redm{\OD(X,\sqcup)}}^\op & \to & \Delta \Charney(X,\sqcup)\\
  \tau=(z_0,z_1,\ldots,z_r) & \mapsto & \{ \big(\bigvee_{i=0}^j z_i, \bigvee_{k=j+1}^r z_k\big) \tq 0\leq j <r \} \end{array} \right.\]
Then $\beta$ is a rank-preserving poset embedding whose image is downward closed.

Moreover, if properties (EX) and (CM) from \ref{def:propsEXCM} hold, then $\beta$ is surjective and therefore an isomorphism.
\end{proposition}

\begin{proof}
It is clear that $\beta$ is a well-defined, order-preserving and rank-preserving map, and also injective: if $\beta(z_0,\ldots,z_r) = \{ (x_0,y_0) < \cdots < (x_{r-1}, y_{r-1}) \}$ then
\begin{equation}
\label{eq:charneyInverse}
    z_j = (\bigvee_{i=0}^j z_i\big) \wedge \big( \bigvee_{k=j}^r z_k\big) = x_j \wedge y_{j-1}.
\end{equation}
Note that this meet always exists by definition of a decomposition.
By convention, we set $x_{-1}=0, y_{-1}=1$ and $x_r=1,y_r=0$.

We show now that $\beta$ is an embedding.
Suppose that $\beta(\tau)\supseteq \beta(\sigma)$ with
\[ \beta(\tau) = \beta(z_0,\ldots,z_r) = \{ (x_0,y_0) < \cdots < (x_{r-1}, y_{r-1}) \}, \]
\[ \beta(\sigma) = \beta(w_0,\ldots,w_s) = \{ (u_0,v_0) < \cdots < (u_{s-1}, v_{s-1}) \}.\]
For $0\leq i\leq s-1$ we take the unique index $0\leq j_i\leq r-1$ such that $(u_i,v_i) = (x_{j_i},y_{j_i})$.
Since $(u_{-1},v_{-1}) = (x_{-1},y_{-1})$ and $(u_s,v_s) = (x_r,y_r)$, we set $j_{-1} = -1$ and $j_s = r$.
It is clear that $(j_i)_{i=-1}^{s}$ is a strictly monotone increasing sequence.
Then, for all $0\leq i\leq s$ we have
\[ w_i = u_i\wedge v_{i-1} = x_{j_i}\wedge y_{j_{i-1}} = (\bigvee_{k=0}^{j_i} z_k\big) \wedge \big( \bigvee_{k=j_{i-1}+1}^r z_k\big) = \bigvee_{k=j_{i-1}+1}^{j_i}z_k.\]
This implies that $\tau\leq \sigma$, so $\beta$ is a poset embedding.
It is also clear that if $\omega \subseteq \beta(z)$ then $\omega \in \Im(\beta)$.

Finally, we show that $\beta$ is surjective under properties (EX) and (CM).
We argue by induction on the size $r$ of a chain $s \in \Delta \Charney(X,\sqcup)$ and in view of \ref{eq:charneyInverse}, we show that if
\[ s = \{ (x_0,y_0) < \cdots < (x_r,y_r) \} \]
then $y_{i-1}\wedge x_i$ exists for all $1\leq i\leq r$ and, if we set $y_{-1} = 1$ and $x_{r+1} = 1$, then $\{x_i \wedge y_{i-1} \tq 0\leq i\leq r+1\}\in \D(X,\sqcup)$.

The case $r = 0$ is trivial.
Suppose then that $r > 0$ and that $s = \{ (x_0,y_0) < \cdots < (x_r,y_r)\}$.
By induction, we have a $\sqcup$-decomposition 
\[ \sigma = \{ x_0, x_0\wedge y_1,\ldots, x_{r-1}\wedge y_{r-2}, y_{r-1}\}\in \D(X,\sqcup).\]
By property (CM) applied to $(x_{r-1},y_{r-1}) < (x_r,y_r)$, we have  $\tau := \{x_r\wedge y_{r-1}, x_{r-1}, y_{r}\}\in \D(X,\sqcup)$.
Hence $\{x_r\wedge y_{r-1},y_r\}\in\D( \Sub(X,\sqcup)_{\leq y_{r-1}})$ by \ref{rk:CPprop} and it is $\sqcup$-compatible in $\Sub(X,\sqcup)$, so by property (EX) we have
\[ \sigma' := \left( \sigma \setminus \{y_{r-1}\} \right) \cup \tau \in \D(X,\sqcup).\]
But note that
\[ \sigma' = \{x_0, x_1\wedge y_0,\ldots, x_{r-1}\wedge y_{r-2}, x_r\wedge y_{r-1}, y_r\} = \{ x_i\wedge y_{i-1} \tq 0\leq i\leq r+1\}.\]
This concludes the induction, and thus $\beta$ is surjective since
\[ \beta(x_0, x_1\wedge y_0,\ldots, x_r\wedge y_{r-1}, y_r) = \{ (x_0,y_0) < \cdots < (x_r,y_r) \}. \]
\end{proof}

We summarize the non-ordered constructions in \ref{tab:objects}.

\renewcommand{\arraystretch}{1.5}
\begin{table}[ht]
\centering
    \centering
    \begin{tabular}{|c|c|c|}
    \hline
    Structure & Name & Reference\\
    \hline
    $\Sub(X)$ & \parbox{8cm}{\centering poset of subobjects} & \ref{sub:monoidalCat} \\
    $\Sub(X,\sqcup)$ & \parbox{8cm}{\centering poset of $\sqcup$-complemented subobjects} & \ref{def:decompositionsCategory}\\
    $\D(X,\sqcup)$ & \parbox{8cm}{\centering poset of $\sqcup$-decompositions} & \ref{def:decompositionsCategory}\\
    $\PD(X,\sqcup)$ & \parbox{8cm}{\centering poset of $\sqcup$-partial  decomposition} & \ref{def:decompositionsCategory}\\
    $\PF(X,\sqcup)$ & \parbox{8cm}{\centering complex of partial frames} & \ref{def:frames}\\
    $\F(X,\sqcup)$ & \parbox{8cm}{\centering frame complex} & \ref{def:frames}\\
    $\PVog(X,\sqcup,P)$ & \parbox{8cm}{\centering complex of partial bases} & \ref{def:vogtmann}\\
    $\Vog(X,\sqcup,P)$ & \parbox{8cm}{\centering (pure) complex of partial bases} & \ref{def:vogtmann}\\
    $\Bergman(X,\sqcup)$ & \parbox{8cm}{\centering augmented Bergman complex} & \ref{def:bergman}\\
    $\Charney(X,\sqcup)$ & \parbox{8cm}{\centering Charney poset} & \ref{def:charney}\\
    \hline
\end{tabular}
\smallskip
\caption{Structures associated to an object $X$ of an ISM-category $(\C,\sqcup)$ for which $\Sub(X,\sqcup)$ has finite height}
\label{tab:objects}
\end{table}

\begin{remark}
\label{rk:posetApproach}
Most of the structures we defined for an object of an ISM-category, along with their corresponding properties, can be translated to the context of posets according to \ref{def:decompositionsAndPD}.
In the case of lattices, they coincide with the definitions in the categorical context as discussed in \ref{ex:latticeAsCategoryApproach}.
However, we will not present those analogous results here; the interested reader can verify the validity of the results presented in this article within the context of posets.
\end{remark}

\section{Homotopy equivalences and isomorphisms}
\label{sec:homotopyposets}

\subsection{General homotopy formulas}

In this section, we first recall some known facts about homotopy types of posets and then apply them to
the posets we defined in the previous sections. When writing down explicit formulas for the homotopy
type, we identify the poset with its order complex. For example $\T \simeq \S \vee \S$ means that the
geometric realization of the order complex of $\T$ is homotopy equivalent to a wedge sum of two
copies of the geometric realization of the order complex of $\S$. If we need to emphasize the topological 
spaces behind our results and constructions, we write $|\S|$ for the geometric realization of the order
complex of a poset $\S$.

Assume that $\S$ is a poset of finite height $n$.
We say that $\S$ is spherical if it is $(n-1)$-connected or, equivalently, homotopy equivalent to a wedge of $n$-spheres.
The poset $\S$ is Cohen-Macaulay if $\S$ is spherical and for all $x,y\in \S$, the intervals $\S_{>x}$, $\S_{<y}$ and $\S_{> x} \cap \S_{< y}$ (if $x<y$) are spherical of height $n-h(x)-1$, $h(y)-1$ and $h(y)-h(x)-2$ respectively.
This is the classical notion of homotopically Cohen-Macaulayness of a poset or simplicial complex.
This implies Cohen-Macaulayness over any field, which is defined by the corresponding high homological connectedness.

Recall that if $f,g:\T\to\S$ are order-preserving maps such that $f(x)\leq g(x)$ for all $x\in \T$, then $f$ and $g$ are homotopy equivalent.
We will repeatedly use this fact throughout this paper to obtain homotopy equivalences and establish the contractibility of some posets.

The first tool we will need is a consequence of Propositions 3.1, 3.5, and 3.7 from \cite{WZZ}.
Note that in \cite{WZZ}, Proposition 3.1 was formulated only for diagrams over finite posets, but the proof 
only relies on paracompactness of the spaces involved (see e.g., \cite{S}). As a consequence, the
results from \cite{WZZ} can be applied to finite height posets.

\begin{proposition}
\label{prop:wedgeLemma}
  Let $\S,\T$ be posets of finite height and $f:\T\to \S$ a poset map such that for all $x\in \S$, 
  the inclusion $f^{-1}\big(\,\S_{<x}\,\big) \hookrightarrow f^{-1}\big(\,\S_{\leq x}\,\big)$ is homotopy 
  equivalent to the constant map with image $c_x\in f^{-1}(\S_{\leq x})$.
  Then there is a homotopy equivalence
  \[ \T \simeq \S \vee \bigvee_{x\in \S} f^{-1}\big(\,\S_{\leq x}\,\big) * \S_{>x},\]
  where the wedge identifies $x\in \S$ with $c_x\in f^{-1}(\S_{\leq x})$.
\end{proposition}

The following well-known fact will be useful for verifying the requirements of 
\ref{prop:wedgeLemma}.

\begin{lemma}
\label{lm:nullHomotopicInclusion}
  Let $\Q\subseteq \S$ be posets such that $\S$ is $k$-connected and $\Q$ has dimension at most 
  $k$.
  Then the inclusion $\Q\hookrightarrow \S$ is null-homotopic.
\end{lemma}

Then from \ref{prop:wedgeLemma} and \ref{lm:nullHomotopicInclusion} we can derive the following version of Quillen's fiber theorem, which we will use repeatedly.

\begin{theorem}[Quillen's fiber theorem]
\label{thm:quillenFiber}
  Let $f:\T\to \S$ be a map between posets of finite height.
  Let $m\geq -1$, and suppose that for all $x\in \S$, $f^{-1}(\S_{\leq x}) * \S_{>x}$ is $m$-connected.
  Then $f$ is an $(m+1)$-equivalence. 
\end{theorem}

We will also use the non-Hausdorff mapping cylinder (see \cite{BM}) of a map $f:\T\to \S$  between posets. This is 
the poset $M_f = \T\coprod \S$ such that its ordering $\leq$ restricts to the given ordering in $\T$ and $\S$, and 
for $x\in \T$ and $y\in \S$ we have $x < y$ if $f(x)\leq y$.
Note that if $\S$ has a unique maximal element $1_{\S}$ then this is also the unique maximal element of $M_f$, and we can write $\redm{M_f} = M_f \setminus \{1_{\S}\}$.

\begin{example}
\label{ex:nonHausdorff}
Let $\T$ be the poset with points $(0,0),(0,1),(1,0),(1,1)$, with $(a,b)\leq (c,d)$ if and only if $a\leq c$.
Let $\S$ be the poset $\{0 < 1\}$, and consider the projection $f:\T\to \S$ defined by $f(a,b)=a$.
Then the non-Hausdorff mapping cylinder of $f$ is the poset displayed in \ref{fig:nonHausdorff1}.

\begin{figure}[ht]
    \centering
    \begin{tikzpicture}
	\draw[draw=black, fill=black, thin, solid] (-3.00,-1.00) circle (0.1);
	\draw[draw=black, fill=black, thin, solid] (-3.00,1.00) circle (0.1);
	\draw[draw=black, fill=black, thin, solid] (-1.00,1.00) circle (0.1);
	\draw[draw=black, fill=black, thin, solid] (-1.00,-1.00) circle (0.1);
	\draw[draw=black, thin, solid] (-3.00,-1.00) -- (-1.00,1.00);
	\draw[draw=black, thin, solid] (-3.00,1.00) -- (-3.00,-1.00);
	\draw[draw=black, thin, solid] (-3.00,1.00) -- (-1.00,-1.00);
	\draw[draw=black, thin, solid] (-1.00,-1.00) -- (-1.00,1.00);
	\draw[draw=black, fill=black, thin, solid] (1.00,2.50) circle (0.1);
	\draw[draw=black, fill=black, thin, solid] (1.00,4.50) circle (0.1);
	\draw[draw=black, thin, solid] (-3.00,-1.00) -- (1.00,2.50);
	\draw[draw=black, thin, solid] (-1.00,-1.00) -- (1.00,2.50);
	\draw[draw=black, thin, solid] (-1.00,1.00) -- (1.00,4.50);
	\draw[draw=black, thin, solid] (-3.00,1.00) -- (1.00,4.50);
	\draw[draw=black, thin, solid] (1.00,2.50) -- (1.00,4.50);
	\node[black, anchor=south west] at (-4.08,-1.60) {(0,0)};
	\node[black, anchor=south west] at (-2.08,-1.60) {(0,1)};
	\node[black, anchor=south west] at (-2.08,0.78) {(1,1)};
	\node[black, anchor=south west] at (-4.08,0.78) {(1,0)};
	\node[black, anchor=south west] at (0.76,4.61) {1};
	\node[black, anchor=south west] at (0.76,1.71) {0};
\end{tikzpicture}
    \caption{Hasse diagram of the non-Hausdorff mapping cylinder of the poset map $f:\T\to \S$}
    \label{fig:nonHausdorff1}
\end{figure}
\end{example}

The adjective non-Hausdorff is motivated by the
non-Hausdorffness of the construction when considering finite posets as finite topological spaces.
We refer the reader to \cite{BM} for a more detailed discussion.

\begin{lemma}
\label{lm:mappingCylinder}
    Let $f:\T\to \S$ be a map of posets of finite height, where $\S$ is bounded. If $M := \{ x \in \T~|~f(x) = 1_{\S}\}$ is an antichain then 
    $\redm{M_f} = M_f \setminus \{1_{\S}\}$ is homotopy equivalent to
    $$\bigvee_{x \in M} \Sigma \T_{< x}.$$
\end{lemma}

\begin{proof}
  Note that $\redm{M_f}\setminus M$ is contractible since $\redm{\S}$ is contractible and the map $r:\redm{M_f}\setminus M \to \redm{\S}$, defined by 
  $r(x) = f(x)$ if $x\in \T$ or $r(x)=x$ if $x\in \S$, is a homotopy equivalence with inverse given by the natural inclusion of $\redm{\S}$ in $\redm{M_f}$.
  Therefore, $|\redm{M_f}| \simeq |\redm{M_f}|/|\redm{M_f}\setminus M|$, and since $M$ is an antichain, we conclude 
  that $|\redm{M_f}|/|\redm{M_f}\setminus M|$ is homotopy equivalent to the wedge of the suspension of the 
  links in $|\redm{M_f}|$ of $x\in M$. The fact that $\Lk_{|\redm{M_f}|}(x) = |\T_{<x}|$ then concludes the proof.
\end{proof}

Let us apply \ref{lm:mappingCylinder} to the map $f$ from \ref{ex:nonHausdorff}.

\begin{example}
Let $\T$, $\S$ and $f$ be as in \ref{ex:nonHausdorff}.
Note that $\S$ is bounded.
Following the notation of \ref{lm:mappingCylinder}, we have that $M = \{ (1,0), (1,1)\}$ is an antichain.
Thus $\redm{M_f}$ is the poset obtained from $\T$ after adding an extra element $0$ above $(0,0)$ and $(0,1)$.
The order complex of this poset is depicted in \ref{fig:nonHausdorff2}, which clearly has the homotopy type of $S^1\vee S^1$.
This coincides with the description given in \ref{lm:mappingCylinder}.

\begin{figure}[ht]
    \centering
    \begin{tikzpicture}
	\draw[draw=black, fill=black, thin, solid] (-4.00,0.00) circle (0.1);
	\draw[draw=black, fill=black, thin, solid] (0.00,0.00) circle (0.1);
	\draw[draw=black, fill=black, thin, solid] (-2.00,0.00) circle (0.1);
	\draw[draw=black, fill=black, thin, solid] (-2.00,2.00) circle (0.1);
	\draw[draw=black, fill=black, thin, solid] (-2.00,-2.00) circle (0.1);
	\draw[draw=black, thin, solid] (-4.00,0.00) -- (0.00,0.00);
	\draw[draw=black, thin, solid] (-4.00,0.00) -- (-2.00,-2.00);
	\draw[draw=black, thin, solid] (-2.00,-2.00) -- (0.00,0.00);
	\draw[draw=black, thin, solid] (0.00,0.00) -- (-2.00,2.00);
	\draw[draw=black, thin, solid] (-2.00,2.00) -- (-4.00,0.00);
	\node[black, anchor=south west] at (-5.19,-0.35) {(0,0)};
	\node[black, anchor=south west] at (0.1,-0.37) {(0,1)};
	\node[black, anchor=south west] at (-2.56,-2.75) {(1,1)};
	\node[black, anchor=south west] at (-2.52,2.05) {(1,0)};
	\node[black, anchor=south west] at (-2.22,-0.73) {0};
\end{tikzpicture}
    \caption{Geometric realization of the poset $\redm{M_f}$ from \ref{ex:nonHausdorff}}
    \label{fig:nonHausdorff2}
\end{figure}
\end{example}

In general, we also have the following formula relating the Euler characteristic of posets related by an 
order-preserving map (see e.g. Corollary 3.2 of \cite{Walker1}).

\begin{lemma}
Let $f:\T\to \S$ be a map of finite posets.
Then 
\[  \RE{\T} = \RE{\S} - \sum_{x\in \S} \RE{f^{-1}(\S_{\leq x})}\RE{\S_{>x}}.\]
\end{lemma}

\subsection{Homotopy formulas and results for posets from \texorpdfstring{\ref{sec:definitionsposet}}{Section 2} - \texorpdfstring{\ref{sec:derivedposets}}{Section 4}}

Our first homotopy equivalence relates the strictly partial decomposition with the subposet of $\sch$-complemented subobjects.
We refer the reader to \ref{def:hcomplementedForObjects} for the different notions of complementation employed here, as well as for the definition of the subposet $\Sub_h(X,\sqcup)$.
Recall that $\Phi(\sigma) = \bigvee_{x\in \sigma} x$ for a partial decomposition $\sigma \in \PD(X,\sqcup)$.

The main applications of the results in this subsection can be found in \ref{sec:examples}.
Here we will give just some small examples. 

\begin{proposition}
\label{prop:PhiEquivalence}
The map $\Phi:\red{\PD(X,\sqcup)} \setminus \D(X,\sqcup)\to \red{\Sub_h(X,\sqcup)}$ is a homotopy equivalence with inverse $i:\red{\Sub_h(X,\sqcup)}\to \red{\PD(X,\sqcup)} \setminus \D(X,\sqcup)$ given by $i(x) = \{x\}$.
In particular, $\red{\Sub_h(X,\sqcup)}$ is a deformation retract of $\red{\PD(X,\sqcup)} \setminus \D(X,\sqcup)$.

Moreover, $\red{\Ord\PD(X,\sqcup)}\setminus \Ord\D(X,\sqcup)$ is also homotopy equivalent to $\red{\Sub_h(X,\sqcup)}$.
\end{proposition}

\begin{proof}
This follows from \ref{lm:PhiSpanMap}, the definition of $\Sub_h(X,\sqcup)$ and the relations $\Phi (i(x)) = x$ and $i(\Phi(\sigma)) \geq \sigma$.
\end{proof}

Now we apply the wedge-decomposition result to relate the ordered and unordered versions of the posets $\D,\PD,\F,\PF, \Vog, \PVog$. Recall the definition of $\Ord K$ for a simplicial complex $K$ from \ref{rem:orderedcomplex}.
We will sometimes write CM for Cohen-Macaulay (do not confuse with property (CM) from \ref{def:propsEXCM}).

\begin{theorem}
\label{thm:orderedVersions}
If $\Sub(X,\sqcup)$ has finite height, the following hold:
\begin{enumerate}
    \item $$\Ord\redm{\D(X,\sqcup)} \simeq \redm{\D(X,\sqcup)} \vee \bigvee_{\sigma \in \redm{\D(X,\sqcup)}} S^{|\sigma|-2} * \redm{\D(X,\sqcup)}_{< \sigma}.$$
    Hence $\Ord\D(X,\sqcup)$ is Cohen-Macaulay if $\D(X,\sqcup)$ is.
    \item $$\Ord\red{\PD(X,\sqcup)} \simeq \red{\PD(X,\sqcup)} \vee \bigvee_{\sigma \in \redm{\D(X,\sqcup)}} S^{|\sigma|-2} * (\red{\PD(X,\sqcup)})_{< \sigma}.$$
    Hence $\Ord\PD(X,\sqcup)$ is Cohen-Macaulay if $\PD(X,\sqcup)$ is.
    \item If $K$ is a finite-dimensional simplicial complex, then
    $$\Ord K \simeq K \vee \bigvee_{\sigma \in K} \bigg( \bigvee_{D(|\sigma|)} S^{|\sigma|-1} \bigg) * \Lk_{K}(\sigma),$$
    where $D(n) = \sum_{i=0}^n(-1)^i \frac{n!}{i!}$.
    Moreover, $K$ is homotopically CM (resp. CM over a ring $R$) if and only if $\Ord K$ is homotopically CM (resp. CM over $R$).
    In particular, this applies to the complexes $\F, \PF, \Vog$ or $\PVog$.
\end{enumerate}
\end{theorem}

\begin{proof}
These conclusions follow from the description of the fibers in \ref{lm:fibersForgetfulMap} in conjunction with \ref{lm:nullHomotopicInclusion} and \ref{prop:wedgeLemma}.
Note that, when $K$ is a simplicial complex regarded as a poset, then $K_{>\sigma}$ is isomorphic to $\Lk_K(\sigma)$ for any simplex $\sigma\in K$.
The Cohen-Macaulay implication in items 1 and 2 follows from \ref{lm:fibersForgetfulMap} items 1 and 2, respectively, jointly with Corollary 9.7 of \cite{Qui78}.

The ``only if" implication of the Cohen-Macaulay part of item 3 follows by an application of Theorem 5.1 of \cite{BWW} by noting that if $\sigma \in K$, then $F_{K}^{-1}(K_{\leq \sigma})$ is Cohen-Macaulay.
That is, $F_K^{-1}(K_{\leq \sigma})$ is the poset of injective words on the alphabet $\sigma$, and this was proved to be shellable (and hence Cohen-Macaulay) by \cite{BW}.

For the ``if" implication, it remains to prove that if $\Ord K$ is Cohen-Macaulay, then the links $\Lk_K(\sigma)$ are spherical.
It is not hard to show that they have the correct dimension.
Now fix $\sigma\in K$, and let $z\in F_{K}^{-1}(\sigma)$.
Let $(v_i)_{i\in I}$ be a total ordering of the vertices of $\Lk_K(\sigma)$.
Consider the map $H:\Lk_K(\sigma)\to \Ord K_{>z}$ given by $H(\tau) = (z,v_{i_1},\ldots,v_{i_k})$ if $\tau = \{v_{i_1}<\cdots<v_{i_k}\}$.
Clearly, $H$ is a well-defined order-preserving poset map.
Let $F':\Ord K_{>z}\to \Lk_K(\sigma)$ be the map that forgets the ordering and deletes the vertices in $\sigma$, that is, $F'(w) = F_{K}(w) \setminus \sigma$.
Thus $F'H(\tau) = \tau$ is the identity map, so $F'$ induces a surjective map on the homotopy (and homology) groups.
Since $\Ord K_{>z}$ is spherical of the same dimension than $\Lk_K(\sigma)$, we conclude that $\Lk_K(\sigma)$ is spherical.
This concludes the proof.
\end{proof}


We study now natural maps between our posets that might be useful to compute the homotopy type relations between them.
We begin with the relation between the complex of partial bases $\Vog(X,\sqcup,P)$ and the frame complex.
Since these are inflation complexes in the sense of \cite{BWW} (see \ref{def:inflationComplex}), by their Theorem 6.2 adapted to the infinite situation, we can prove:

\begin{proposition}
\label{prop:inflationDecompositionCM}
Let $K$ be a finite-dimensional simplicial complex and $P = (P_x)_{x\in \A}$ a collection of non-empty sets $P_x$ indexed by the set of vertices $\A$ of $K$.
Let $p:(K,P)\to K$ be the deflation map.
Then we have the following homotopy equivalence:
\[ (K,P) \simeq K \vee \bigvee_{\sigma\in K} \bigg( \bigvee_{\prod_{x\in \sigma} |P_x|-1} S^{|\sigma|-1} \bigg) * \Lk_K(\sigma),\]
Furthermore, $p^{-1}(K_{\leq\sigma}) = P_{x_1} * \cdots * P_{x_r}$ for $\sigma = \{x_1,\ldots, x_r\}\in K$ (where the $P_x$ are regarded as discrete complexes), and $K$ is homotopically CM if and only if $(K,P)$ is homotopically CM.
The same conclusions hold for Cohen-Macaulayness over a ring.

In particular, this applies to $K = \PF(X,\sqcup)$ or $\F(X,\sqcup)$ with $(K,P) = \PVog(X,\sqcup,P)$ or $\Vog(X,\sqcup,P)$ respectively.
\end{proposition}

\begin{proof}
The homotopy equivalence is a consequence of \cite[Theorem 6.2]{BWW} adapted to the infinite case, regarding a simplicial complex as a poset via its face poset.
Alternatively, this follows from \ref{prop:wedgeLemma} since $p^{-1}(K_{\leq \sigma}) = P_{x_1} * \cdots * P_{x_r}$ for $\sigma = \{x_1,\ldots,x_r\} \in K$.
Also recall that $K_{>\sigma} \cong \Lk_K(\sigma)$.

Then by \cite[Corollary 9.7]{Qui78}, since the lower fibers $p^{-1}(K_{\leq \sigma})$ are homotopically CM of the correct dimension, if $K$ is homotopically CM then $(K,P)$ is homotopically CM as well.

Conversely, suppose that $(K,P)$ is homotopically CM.
By the wedge decomposition, $K$ is spherical.
It remains to prove that the links of its simplices are spherical of the correct dimension.
Let $\sigma \in K$ and consider a simplex $B\in (K,P)$ with $p(B) = \sigma$.
Then $\Lk_{(K,P)}(B)$ is isomorphic to the inflation complex $(\Lk_{K}(\sigma), P^B)$, where $P^B:=(P_x)_{x \in \A_{\sigma}}$ and $\A_{\sigma}$ is the set of vertices of $\Lk_{K}(\sigma)$.
Since $\Lk_{(K,P)}(B)$ is spherical by Cohen-Macaulayness, by the wedge decomposition again we conclude that $\Lk_{K}(\sigma)$ is spherical.

The homological version of CM over a given ring is obtained by a similar argument.
\end{proof}

We refer to \ref{sec:examples} for concrete applications of these results.
For instance, in the case of (finite) matroids (see \ref{sub:matroids}), it is well-known that the independence complex is shellable and hence homotopically Cohen-Macaulay.
In our language, this is the partial basis complex associated with the lattice of flats.
Therefore, as a consequence of the previous result, the corresponding frame complex is homotopically Cohen-Macaulay.


We derive some results on the augmented Bergman complex using the non-Hausdorff mapping cylinder.

\begin{lemma}
$\Bergman(X,\sqcup)$ is homotopy equivalent to $\redm{M_\Phi}$, where $M_{\Phi}$ is the non-Hausdorff mapping cylinder of the map $\Phi:\F(X,\sqcup)\to \Sub(X,\sqcup)$.
\end{lemma}

\begin{proof}
Consider the map $\varphi: \Bergman(X,\sqcup) \to \redm{M_\Phi}$ given by
\[ \varphi(\sigma \cup \{x_1,\ldots,x_r\}) = \begin{cases}
    \sigma & r = 0,\\
    x_r & r > 0.
\end{cases}\]
Clearly, $\varphi$ is a well-defined order-preserving map.
To prove it is a homotopy equivalence, we show that the lower homotopy fibers are contractible.
Indeed, for $\sigma \in \F(X,\sqcup)$ we have
\[ \varphi^{-1} \big( {\redm{M_\Phi}}_{\leq \sigma} \big) = \Bergman(X,\sqcup)_{\leq \sigma}\simeq *.\]
On the other hand, for $x\in \redm{\Sub(X,\sqcup)}$ we get
\[ \sigma \cup \{x_1,\ldots,x_r\} \in \varphi^{-1}\big( {\redm{M_\Phi}}_{\leq x} \big) \, \Longleftrightarrow \, \Phi(\sigma) \leq x \text{ and } x_r\leq x.\]
When $r>0$, the condition $\Phi(\sigma)\leq x$ follows immediately from $x_r\leq x$.
In this way, we can construct a contraction homotopy inside this fiber:
\[ \sigma \cup \{x_1,\ldots,x_r\}\leq \sigma\cup \{x_1,\ldots,x_r,x\} \geq \{x\}.\]
Note that all these terms lie in $\varphi^{-1} \big( {\redm{M_\Phi}}_{\leq \sigma} \big)$.

Since the lower homotopy fibers are contractible, by Quillen's fiber theorem (see \ref{thm:quillenFiber}) we conclude that $\varphi$ is a homotopy equivalence.
\end{proof}

Since the elements of the poset $\F(X,\sqcup)$ whose join is the unique maximal element of $\Sub(X,\sqcup)$ consist of exactly the set of full frames, they form an antichain.
In conjunction with \ref{lm:mappingCylinder} this immediately implies the following result, which is well-known in the matroid case (see \cite{BKR}).

\begin{proposition} \label{prop:bergmantop}
Let $M$ be the set of full frames of $\Sub(X,\sqcup)$.
Then $\Bergman(X,\sqcup)$ is homotopy equivalent to 
$$\bigvee_{\sigma \in M} S^{|\sigma|-1}.$$
In particular, $\Bergman(X,\sqcup)$ is spherical, and if $M$ is empty (i.e., there are no full frames) then $\Bergman(X,\sqcup)$ is contractible.
\end{proposition}

We have the following partial result on the connectivity of $\red{\PD(X,\sqcup)}$ if we have information on the lower intervals of both $\Sub_h(X,\sqcup)$ and the partial decompositions.
We will use property (LI) from \ref{def:propLI} to prove the following theorem.

\begin{theorem}
\label{thm:highConnectivity}
Let $n$ denote the height of $\Sub(X,\sqcup)$.
Assume that property (LI) holds and that for all $y\in \Sub_h(X,\sqcup)$, the poset $\red{\Sub_h(Y,\sqcup)}$ is spherical of dimension $h(y)-2$.
Then $\red{\PD(X,\sqcup)}$ is $(n-3)$-connected.
Moreover, the inclusion $\red{\Sub_h(X,\sqcup)}\hookrightarrow \red{\PD(X,\sqcup)}$ is an $(n-2)$-equivalence.
\end{theorem}

\begin{proof}
Note that $\red{\Sub_h(X,\sqcup)}$ is spherical of dimension $n-2$ by hypothesis, and that the inclusion $\red{\Sub_h(X,\sqcup)}\hookrightarrow \red{\PD(X,\sqcup)}\setminus \D(X,\sqcup)$ is a homotopy equivalence by \ref{prop:PhiEquivalence}.
Therefore, it is enough to show that the inclusion $i:\red{\PD(X,\sqcup)}\setminus \D(X,\sqcup) \hookrightarrow \red{\PD(X,\sqcup)}$ is an $(n-2)$-equivalence.

We analyze the lower fibers of full decompositions for the map $i$ and apply Quillen's fiber theorem later.
Let $\sigma_0\in \redm{\D(X,\sqcup)}$.
By property (LI), the lower fiber $i^{-1}(\redm{\D(X,\sqcup)})$ can be identified with the poset
\[ F_{\sigma_0} := \red{\big(\prod_{y\in \sigma} \PD(Y,\sqcup) \big)} \setminus \prod_{y\in \sigma} \D(Y,\sqcup).\]
Write $\sigma_0 = \{y_1,\ldots,y_r\}$, and $\PD(y_i) := \PD(Y_i,\sqcup)$, $\D(y_i) := \D(Y_i,\sqcup)$ and $\Sub_h(y_i) := \Sub_h(Y_i,\sqcup)$.
We claim that $F_{\sigma_0}$ is homotopy equivalent to $\T := \red{ \big( \prod_{i=1}^r \Sub_h(y_i) \big)}$.

Let $\Psi:F_{\sigma_0}\to \T$ be the map such that 
\[
\Psi(\tau_1,\ldots,\tau_r) = (\Phi(\tau_1),\ldots,\Phi(\tau_r)).
\]
Clearly, $\Psi$ is well-defined and order-preserving.

On the other hand, we have a map $\Gamma : \T \to F_{\sigma_0}$ which is just the canonical inclusion $\Gamma(x_1,\ldots,x_r) = ( \{x_1\} \setminus\{0\}, \ldots, \{x_r\}\setminus\{0\})$.
Then $\Psi\circ\Gamma$ is the identity map on $\T$, and
\[ \Gamma( \Psi(\tau_1,\ldots,\tau_r) )= (\{\Phi(\tau_1)\}\setminus\{0\},\ldots,\{\Phi(\tau_r)\}\setminus\{0\}) \geq (\tau_1,\ldots,\tau_r).\]
Now, $\T$ is the proper part of the direct product of bounded posets, and hence it is a standard fact that the geometric realization of $\T$ is isomorphic to an iterated  suspension of a join of the spaces $|\red{\Sub_h(y_i)}|$:
\[ |\T| \cong \Sigma( \cdots \Sigma\big(\Sigma\big(|\red{\Sub_h(y_1)}|*|\red{\Sub_h(y_2)}|\big)*|\red{\Sub_h(y_3)}|\big) \cdots * |\red{\Sub_h(y_r)}|\big).\]
See \cite{Walker2}.
Since $\red{\Sub_h(y_i)}$ is $(h(y_i)-3)$-connected by hypothesis, we conclude that $\T$ is
\[ 3(r-1) + \sum_i (h(y_i) -3) = 3(r - 1) + n - 3r = n - 3\]
connected.
On the other hand, we know that $\redm{\D(X,\sqcup)}_{>\sigma_0}$ is the proper part of the partition lattice on $\sigma_0$, so it is spherical of dimension $r-3$, and hence $(r-4)$-connected (the empty space is $(-2)$-connected by convention).
Thus $F_{\sigma_0} * \redm{\D(X,\sqcup)}_{>\sigma_0}$ is
\[ (n-3) + (r-4) + 2 = n+r-5\]
connected.
In particular, since $r\geq 2$, these are $(n-3)$-connected.
By \ref{thm:quillenFiber}, we conclude that $i$ is an $(n-2)$-equivalence, and therefore $\red{\PD(X,\sqcup)}$ is $(n-3)$-connected.
\end{proof}

\begin{example}
Let $L$ be the lattice of subsets of $\{1,2,3,4\}$ of size different from three.
In \ref{sub:matroids} we will see that this is the lattice of flats of the uniform matroid $U_{4,3}$.
Then, by \ref{ex:latticeAsCategoryApproach}, $L = \Sub(X,\sqcup)$ if $L$ is regarded as a category with monoidal product $\sqcup=\vee$, and $X=\{1,2,3,4\}$ the maximal element.
Note $L$ has height three, and $\PD(L) = \PD(X,\sqcup)$.
Also, $L = \Sub_h(X,\sqcup)$, and by the previous theorem we conclude that the inclusion $i:\red{L} \hookrightarrow \red{\PD(L)}$ is a $1$-equivalence.
Since $\red{L}$ is connected, we see that $\red{\PD(L)}$ is connected.
Now we show that $\red{\PD(L)}$ is simply connected, so the inclusion $i$ induces the trivial map in the fundamental groups.

If we regard $\red{L}$ as a $1$-dimensional simplicial complex, it is not hard to see that $\red{L}$ becomes simply connected after attaching $2$-cells to the cycles of the form
\begin{equation}
\label{eq:apartmentDim2}
\{1\}<\{1,2\}>\{2\}<\{2,3\}>\{3\}<\{1,3\}>\{1\}. 
\end{equation}
Thus, $\red{\PD(L)}$ is simply connected if each such cycle is null-homotopic when included in $\red{\PD(L)}$.
Let $C$ denote the cycle from \ref{eq:apartmentDim2}, seen as a subposet of $\red{L}$.
Then $i:C\to \red{\PD(L)}$ is null-homotopic via the following equivalences of maps:
\[ i(x)  = \{x\} \geq \{\{k\} \tq k\in \{1,2,3\}\cap x\} \leq \{\{1\},\{2\},\{3\}\}.\]

From this we can conclude that $\red{\PD(L)}$ is at least $1$-connected.
Moreover, it is not hard to show that $\red{\PD(L)}$ collapses to the subposet whose elements are the non-empty subsets of size at most three of $\{1,2,3,4\}$ (which is the proper part of the Boolean lattice of rank four).
Since this poset has dimension two and it is $1$-connected, we conclude that $\red{\PD(L)}$ has the homotopy type of a wedge of $2$-spheres.
\end{example}

In the following two results, we will see that the situation in the previous example actually generalizes to a more general setting.

\begin{proposition}
\label{prop:topHomology}
Let $n$ be the height of $\Sub(X,\sqcup)$ and suppose $\sigma\in \F(X,\sqcup)$ is a full frame (i.e., a decomposition of size $n$).
Then the \textit{Coxeter complex} $S_{\sigma}$ associated with $\sigma$ gives a non-zero class $[S_\sigma]\in \widetilde{H}_{n-2}( \red{\Sub_h(X,\sqcup)},\ZZ) \subseteq \widetilde{H}_{n-2}( \red{\Sub(X,\sqcup)},\ZZ)$.

Moreover, $S_{\sigma}$ is null-homotopic when embedded in $\red{\PD(X,\sqcup)}$ via the canonical inclusion $\red{\Sub_h(X,\sqcup)}\hookrightarrow \red{\PD(X,\sqcup)}$.
\end{proposition}

\begin{proof}
There is an $(n-2)$-dimensional sphere embedded in the order complex of $\Sub_h(X,\sqcup)$ which is obtained by taking the usual apartment of a basis, here regarded as a frame.
Concretely, this sphere is a Coxeter complex obtained as the order complex of the subposet
\[ S_{\sigma} := \big\{\bigvee_{x\in \tau} x \tq \emptyset\neq \tau\subsetneq \sigma \big\}.\]
Since $S_{\sigma}$ is an $(n-2)$-dimensional sphere, its class gives a non-zero cycle for the $(n-2)$-homology group of $\red{\Sub_h(X,\sqcup)}$, which is naturally included in $\widetilde{H}_{n-2}( \red{\Sub(X,\sqcup)},\ZZ)$.

On the other hand, note that the image of $S_{\sigma}$ under the canonical map $\red{\Sub_h(X,\sqcup)}\to \red{\PD(X,\sqcup)}$
is contained in the subposet
\[T_{\sigma} = \{ \tau \tq \emptyset\neq \tau\subseteq \tau'\in\redm{\D(X,\sqcup)}_{\geq \sigma} \}.\]
Then $T_{\sigma}$ is the proper part of the partial decomposition poset of the Boolean lattice on $\sigma$.
In particular, by \ref{coro:minElementD}, the poset $T_{\sigma}$ is contractible.
Hence the inclusion of the sphere $S_{\sigma} \hookrightarrow \red{\PD(X,\sqcup)}$ is null-homotopic (and thus trivial in homology).
\end{proof}

\begin{corollary}
\label{coro:highConnectivity}
Let $n$ be the height of $\Sub(X,\sqcup)$, and assume the following conditions hold:
\begin{enumerate}
    \item[$\bullet$] property (LI);
    \item[$\bullet$] for all $y\in \Sub_h(X,\sqcup)$, the poset $\red{\Sub_h(Y,\sqcup)}$ is spherical of dimension $h(y)-2$;
    \item[$\bullet$] the spheres $S_{\sigma}$, for $\sigma$ a full frame, generate the $(n-2)$th homotopy group of $\red{\Sub_h(X,\sqcup)}$.
\end{enumerate}
Then $\red{\Sub_h(X,\sqcup)}\hookrightarrow \red{\PD(X,\sqcup)}$ is null-homotopic.
In particular, $\red{\PD(X,\sqcup)}$ is $(n-2)$-connected.
\end{corollary}

\begin{proof}
By \ref{thm:highConnectivity}, it is enough to show that $\red{\Sub_h(X,\sqcup)}\hookrightarrow \red{\PD(X,\sqcup)}$ maps a set of generators of the $(n-2)$th homotopy group to the trivial subgroup.
By hypothesis, we can take as a generating set the classes of the spheres obtained as the order complexes of $S_{\sigma}$ for the full frames $\sigma$.
By \ref{prop:topHomology}, these spheres are null-homotopic in $\red{\PD(X,\sqcup)}$.
The result follows.
\end{proof}

\subsection{Unique complementation}

Now we explore some consequences of the relations between the different posets we defined when we additionally have unique downward $\sch$-complementation and property (EX) holds.
Recall that property (EX) states that if we have a partial decomposition $\sigma$ of $\Sub(X,\sqcup)$ and we pick a full decomposition $\tau \in \D(\Sub(X,\sqcup)_{\leq y})$ of the lower interval of one of its elements $y\in \sigma$ which is $\sqcup$-compatible in $\Sub(X,\sqcup)$, then $( \sigma \setminus \{y\})\cup \tau$ is a partial decomposition of $X$.
See \ref{def:propsEXCM}.

\begin{lemma}
\label{lm:isoGMap}
If $\Sub(X,\sqcup)$ is uniquely downward $\sch$-complemented and the map $G$ defined in \ref{lm:GMapOrdDtoS} is surjective, then $G$ is an isomorphism between $\redm{\Ord\D(X,\sqcup)}$ and the opposite poset of $\Delta\red{\Sub(X,\sqcup)}$.
\end{lemma}

\begin{proof}
Since $G$ is a surjective poset map by hypothesis, it is enough to prove that $G(\tau) \supseteq G(\sigma)$ implies that $\tau \leq \sigma$.

Assume then that $G(\tau)\supseteq G(\sigma)$.
Write $\tau=(z_1,\ldots,z_r)$ and $\sigma=(w_1,\ldots,w_s)$.
For $1\leq j\leq s-1$, let $1\leq \beta(j) \leq r$ be such that
\[ y_j:=w_1\vee\cdots\vee w_j = z_1\vee\cdots\vee z_{\beta(j)}.\]
We claim that
\[ w_{j+1} = z_{\beta(j)+1}\vee \cdots\vee z_{\beta(j+1)}.\]
Indeed, $w_{j+1}$ is the unique $\sch$-complement of
\[ y_j = z_1\vee\cdots\vee z_{\beta(j)}\]
in
\[ y_j \vee w_{j+1} = y_j \vee z_{\beta(j)+1} \vee \cdots\vee z_{\beta(j)}.\]
But also $z_{\beta(j)+1} \vee \cdots\vee z_{\beta(j)}$ is the unique $\sch$-complement of $y_j$ in $y_j \vee w_{j+1}$.
Thus they must coincide.
This shows that $\tau\leq \sigma$.
\end{proof}

\begin{lemma}
\label{lm:GMapSurjective}
If $\Sub(X,\sqcup)$ is uniquely downward $\sch$-complemented and satisfies property (EX), then the map $G$ is surjective.
In particular, $\redm{\Ord\D(X,\sqcup)}$ is isomorphic to $(\Delta\red{\Sub(X,\sqcup)})^{\op}$.
\end{lemma}

\begin{proof}
We prove this by induction on the size of a chain.
Let $c = \{ z_0 < z_1 \cdots < z_r \} \in \Delta \red{\Sub(X,\sqcup)}$ be a chain.
If $r = 0$, let $w_0$ be the unique $\sch$-complement of $z_0$.
Thus $\{z_0,w_0\}$ is a $\sqcup$-decomposition and $G(z_0,w_0) = (z_0)$.
Assume then that $r > 0$, and let $c' = c\setminus \{z_r\}$.
Hence there exists an ordered $\sqcup$-decomposition $(x_0,\ldots,x_r)$ with $z_i = x_0 \vee \cdots \vee x_i$, for $0\leq i \leq r-1$.
By uniqueness, we have $x_r = z_{r-1}^\perp > z_r^{\perp}$.
Let $u$ be the unique $\sch$-complement of $z_r^\perp$ in $x_r$.
Then $\{u,z_r^{\perp}\}\in \D( \Sub(X,\sqcup)_{\leq x_r})$ and it is $\sqcup$-compatible in $\Sub(X,\sqcup)$.
By property (EX), we see that $\{x_0,\ldots,x_{r-1},u,z_r^{\perp}\}\in \D(X,\sqcup)$ and also $z_r = z_{r-1}\vee u$.
Hence $G(x_0,\ldots,x_{r-1},u,z_r^{\perp}) = (z_0,\ldots,z_{r-1},z_r)$.
\end{proof}

\begin{corollary}
\label{coro:minElementD}
Suppose that $\D(X,\sqcup)$ has a unique minimum element $\sigma_0$.
Then $\D(X,\sqcup)$ is the partition lattice on $\sigma_0$, $\Sub_h(X,\sqcup)$ is a Boolean lattice on the set $\sigma_0$ and $\red{\PD(X,\sqcup)}$ is contractible.
Moreover, $\Sub_h(X,\sqcup)$, $\D(X,\sqcup)$, $\PD(X,\sqcup)$ are Cohen-Macaulay.
\end{corollary}

\begin{proof}
By \ref{prop:decompintervals}, $\D(X,\sqcup)$ is the partition lattice on $\sigma_0$.
On the other hand, any element of $x\in \Sub_h(X,\sqcup)$ is part of a $\sqcup$-decomposition, so there exists a unique set $\tau\subseteq\sigma_0$ such that $x$ is the join of the elements of $\tau$.
Therefore $\Sub_h(X,\sqcup)$ is the Boolean lattice on the set $\sigma_0$.
In particular $\Sub_h(X,\sqcup)$ and $\D(X,\sqcup)$ are Cohen-Macaulay.

Now we prove that $\red{\PD(X,\sqcup)}$ is contractible.
In fact, we show that if $\sigma \in \PD(X,\sqcup)\setminus \D(X,\sqcup)$, then $\PD(X,\sqcup)_{>\sigma} \setminus\{ \{1\}\}$ is contractible (so $\red{\PD(X,\sqcup)}$ is contractible by taking $\sigma = \{\}$).

Let $\sigma \in \PD(X,\sqcup)\setminus \D(X,\sqcup)$.
Then $\sigma$ is contained in some $\sqcup$-decomposition (i.e. partition) $\tau$ and every element $y\in \sigma$ can be obtained by joining elements of $\sigma_0$.
Thus, the set
\[ f(\sigma) := \sigma \cup \{x \in \sigma_0 \tq x \nleq y \text{ for all } y\in\sigma\}\]
is a partition of $\sigma_0$.
Now we fix $\sigma$, and let $\sigma'>\sigma$.
Then $f(\sigma') \geq \sigma' > \sigma$.
Since $\sigma \notin \D(X,\sqcup)$ and $f(\sigma)\in\D(X,\sqcup)$, we see that $f(\sigma)>\sigma$.
Thus we have the following inequalities in $\PD(X,\sqcup)_{>\sigma}\setminus \{ \{1\} \}$:
\[ \sigma' \leq f(\sigma') \geq f(\sigma).\]
This proves that $\PD(X,\sqcup)_{>\sigma}\setminus \{ \{1\} \}$ is contractible.
It is also not hard to show that this interval has the correct dimension.

Since $f(\emptyset) = \sigma_0$, the homotopy above is
\[ \sigma'\leq f(\sigma')\geq \sigma_0.\]
Hence $\red{\PD(X,\sqcup)}$ is contractible.

On the other hand, if $\sigma\in \D(X,\sqcup)$ then
\[ \PD(X,\sqcup)_{>\sigma} = \D(X,\sqcup)_{>\sigma} \groupiso \Pi(\sigma)_{>\sigma},\]
which is a Cohen-Macaulay poset.
Therefore, upper intervals in the poset $\red{\PD(X,\sqcup)}$ are spherical of the correct dimension.

Finally, since lower intervals in $\PD(X,\sqcup)$ are products of partial decomposition posets, which are partial decompositions of Boolean lattices, we conclude that $\PD(X,\sqcup)$ is a Cohen-Macaulay poset.
\end{proof}

\begin{proposition}
\label{prop:uniqueDownwardhComplementationAndEP}
If $\Sub(X,\sqcup)$ is uniquely downward $\sch$-complemented and satisfies property (EX) then there is a deformation retract of $\red{\PD(X,\sqcup)}$ onto $\redm{\D(X,\sqcup)}$.
\end{proposition}

\begin{proof}
We prove that for any $\tau \in \red{\PD(X,\sqcup)} \setminus \D(X,\sqcup)$, the interval $\D_{\tau}:=\redm{\D(X,\sqcup)}_{\geq \tau}$ is contractible.
Let $\sigma \in \D_{\tau}$.
For $y\in \sigma$, let $\tau_y := \{x\in\tau\tq x\leq y\}$ and $\alpha(y) := \bigvee_{x\in \tau_y}x$.
Then $\alpha(y)\leq y$.
Let $\beta(y)\leq y$ be the unique $\sch$-complement for $\alpha(y)$ in $y$.
Hence we get a decomposition $\{\alpha(y),\beta(y)\}\setminus \{0\} \in \D(\Sub(X,\sqcup)_{\leq y})$ which is $\sqcup$-compatible in $\Sub(X,\sqcup)$.
By property (EX),
\[ \sigma'_y:=(\sigma\setminus\{y\}) \cup \{\alpha(y),\beta(y)\} \setminus \{0\}\]
is a $\sqcup$-decomposition of $\Sub(X,\sqcup)$.
We can continue in this way until producing a $\sqcup$-decomposition
\[ \Lambda(\sigma) := \{ \alpha(y),\beta(y) : y\in \sigma\} \setminus \{0\} \in \D(X,\sqcup).\]
Note that $\Lambda(\sigma)\leq \sigma$ and $\tau \leq \Lambda(\sigma)$.
Now, let $x$ be the join of the elements of $\tau$ and $x'$ be the $\sch$-complement of $x$ in $\Sub(X,\sqcup)$.
We claim that $\beta(y)\leq x'$.
Since $x = \bigvee_{y \in \sigma} \alpha(y)$, if $x'' = \bigvee_{y\in \sigma} \beta(y)$ then $\{x,x''\}\in \D(X,\sqcup)$, so $x,x''$ are $\sch$-complements of each other since $\Lambda(\sigma)$ is a $\sqcup$-decomposition.
By uniqueness, $x'' = x'$, and this implies that $\beta(y)\leq x'$.
Hence we have the following zigzag whose elements lie in $\D_{\tau}$:
\[\sigma \geq \Lambda(\sigma) \leq \{x,x'\}.\]
It remains to prove that the map $\sigma\mapsto \Lambda(\sigma)$ is order-preserving.
For simplicity, we prove it for covering relations, so assume that $\sigma\leq \rho$, where $\rho$ is obtained by joining two elements $y_1,y_2$ of $\sigma$.
Then we have the $\sqcup$-decompositions by $\sch$-complements $y_1 = \alpha(y_1) \vee \beta(y_1)$ and $y_2 = \alpha(y_2)\vee\beta(y_2)$.
Clearly $\alpha(y_1\vee y_2) = \alpha(y_1)\vee\alpha(y_2)$.
Let $z = \beta(y_1)\vee\beta(y_2)$, which exists since $\Lambda(\sigma)$ is a $\sqcup$-decomposition.
This also implies that $h(z) = h(\beta(y_1)) + h(\beta(y_2))$, and that $z \wedge (\alpha(y_1)\vee\alpha(y_2)) = 0$.
We also have that $z\vee \alpha(y_1)\vee\alpha(y_2) = y_1\vee y_2$.
Therefore $z$ is an $\sch$-complement for $\alpha(y_1)\vee\alpha(y_2)$ in $y_1\vee y_2$.
By uniqueness, $z = \beta(y_1\vee y_2)$.
Hence $\Lambda(\rho)$ is obtained from $\Lambda(\sigma)$ after taking the joins $\alpha(y_1)\vee \alpha(y_2)$ and $\beta(y_1)\vee \beta(y_2)$.
Thus $\Lambda(\sigma)\leq \Lambda(\rho)$.
\end{proof}

\begin{corollary}
\label{coro:collapseFrames}
Suppose that $\Sub(X,\sqcup)$ is uniquely downward $\sch$-complemented, and let $\sigma \in \PF(X,\sqcup)$ be a simplex of size $n-1$.
Then the following hold:
\begin{enumerate}
    \item $\sigma$ is contained in a unique maximal simplex.
    \item There is a deformation retract of $\F(X,\sqcup)$, viewed as a poset, onto the subposet $\widehat{\F}(X,\sqcup)$ of frames of size $\neq n-1$.
    \item If property (EX) holds then we have a poset inclusion $\widehat{\F}(X,\sqcup)^{\op} \hookrightarrow \redm{\D(X,\sqcup)}$ given by $\tau \mapsto \tau \cup \{ \big(\bigvee_{x\in \tau}x \big)^\perp\}$.
\end{enumerate}
\end{corollary}

\begin{proof}
Let $\sigma \in \PF(X,\sqcup)$ be a frame of size $n-1$, and recall that $\Phi(\sigma)$ is the join of the elements of $\sigma$.
Since $\sigma$ is a partial decomposition, $\sigma \cup \{\Phi(\sigma)^\perp\}\in\D(X,\sqcup)$.
Now, $|\sigma|=n-1$ implies that $h(\Phi(\sigma)) = n-1$ and $h(\Phi(\sigma)^\perp) = 1$.
Therefore $\sigma \cup \{\Phi(\sigma)^\perp\}$ is a frame of size $n$, and clearly $\PF(X,\sqcup)_{>\sigma} = \{ \sigma \cup \{\Phi(\sigma)^\perp\}\} $ by uniqueness of the complement.
This proves item 1.

Item 2 is an easy consequence of item 1. We write down the deformation retract.
Let $\widehat{\F}(X,\sqcup)$ be the subposet of $\F(X,\sqcup)$ consisting of frames of size $\neq n-1$.
By the previous paragraph, $\F(X,\sqcup)_{>\sigma} = \{ \sigma \cup \{\Phi(\sigma)^\perp\}\}$ for any frame $\sigma$ of size $n-1$.
Therefore, the retraction $r:\F(X,\sqcup)\to \widehat{\F}(X,\sqcup)$ that takes a frame of size $n-1$ to the unique maximal frame of size $n$ containing it is an order-preserving map such that $r(\sigma)\geq\sigma$ for all $\sigma \in \F(X,\sqcup)$ and it is the identity on $\widehat{\F}(X,\sqcup)$.
Hence $\F(X,\sqcup) \simeq \widehat{\F}(X,\sqcup)$.

The function $\widehat{\F}(X,\sqcup)^{\op} \hookrightarrow \redm{\D(X,\sqcup)}$ defined in item 3 is an inclusion of posets by \ref{coro:uniquenessEPandCP}.
\end{proof}

\begin{corollary}
\label{coro:eulerCharOPDUniqueCompl}
Assume the following hold for $\Sub(X,\sqcup)$:
\begin{enumerate}
    \item $\Sub(X,\sqcup)$ is finite and uniquely downward $\sch$-complemented.
    \item For $y\in \Sub(X,\sqcup)$, $\Sub(X,\sqcup)_{\leq y} = \Sub(Y,\sqcup)$ for a representative $y = [(Y,i)]$.
    \item Property (EX) holds.
    \item The lower intervals in $\D(X,\sqcup)$ and $\PD(X,\sqcup)$ decompose as in \ref{lm:intervalsOrdered}(2).
\end{enumerate}
Then $\tilde{\chi}(\red{\Ord\PD(X,\sqcup)}) = -1$.
\end{corollary}

\begin{proof}
We compute $\tilde{\chi}(\red{\Ord\PD(X,\sqcup)})$ by using the wedge decomposition given in \ref{thm:orderedVersions}(2) and the following formula for the Euler characteristic of a finite poset:
\begin{equation}
\label{eq:posetFormulaEulerD}
\tilde{\chi}\big({\redm{\D(X,\sqcup)}}\big) = -1 - \sum_{\sigma\in \redm{\D(X,\sqcup)}} \tilde{\chi}(\redm{\D(X,\sqcup)}_{<\sigma}).    
\end{equation}

First, note that $\Sub(Y,\sqcup)$ is also uniquely downward $\sch$-complemented since it coincides with $\Sub(X,\sqcup)_{\leq y}$ for $y = [(Y,i)]$.
From this, it is also not hard to see that property (EX) holds for $\Sub(Y,\sqcup)$.
Thus, by \ref{prop:uniqueDownwardhComplementationAndEP}, $\red{\PD(Y,\sqcup)} \simeq \redm{\D(Y,\sqcup)}$ for any $y = [(Y,i)]\in \Sub(X,\sqcup)$.

Now let $\sigma = \{y_1,\ldots,y_r\}\in \redm{\D(X,\sqcup)}$, with representatives $y_i = [(Y_i,j_i)]$.
We have
\begin{align*}
\tilde{\chi}( (\red{\PD(X,\sqcup)})_{<\sigma} ) & = \tilde{\chi}\bigg( \red{ \big(\prod_i \PD(Y_i,\sqcup) \big) } \bigg) = \prod_i \tilde{\chi}\big( \red{\PD(Y_i,\sqcup)} \big) = \prod_i \tilde{\chi}\big( \redm{\D(Y_i,\sqcup)} \big),    
\end{align*}
and
\begin{align*}
    \tilde{\chi}( \redm{\D(X,\sqcup)}_{<\sigma} ) & = \tilde{\chi}\bigg( \redm{ \big(\prod_i \D(Y_i,\sqcup) \big) } \bigg) =  \tilde{\chi}\bigg( \redm{\D(Y_1,\sqcup)} * \cdots * \redm{\D(Y_r,\sqcup)}  \bigg)\\
    & = (-1)^{|\sigma|-1} \prod_i \tilde{\chi}( \redm{\D(Y_i,\sqcup)} ) \\
    & = (-1)^{|\sigma|-1} \tilde{\chi}( \red{\PD(X,\sqcup)}_{<\sigma} ).
\end{align*}
By item 2 of \ref{thm:orderedVersions} and the above equalities, we get
\begin{align*}
\tilde{\chi}(\red{\Ord\PD(X,\sqcup)}) & = \tilde{\chi}(\redm{\D(X,\sqcup)}) - \sum_{\sigma\in \redm{\D(X,\sqcup)}} (-1)^{|\sigma|-2} \tilde{\chi}( \red{\PD(X,\sqcup)}_{<\sigma} )\\
& = \tilde{\chi}(\redm{\D(X,\sqcup)}) + \sum_{\sigma\in \redm{\D(X,\sqcup)}} \tilde{\chi}( \redm{\D(X,\sqcup)}_{<\sigma} )\\
& = \tilde{\chi}(\redm{\D(X,\sqcup)}) + (-\tilde{\chi}(\redm{\D(X,\sqcup)}) - 1)\\
& = -1.
\end{align*}
\end{proof}

The next example shows that $\red{\Ord\PD(X,\sqcup)}$ 
has the homotopy type of a sphere of the maximal possible dimension.

\begin{example}
\label{ex:OPDBooleanLattice}
Let $L = \Sub(X,\sqcup)$ be the Boolean lattice on $n$ elements, hence of height $n$.
Since this is a uniquely downward $\sch$-complemented geometric lattice (see \ref{sub:matroids}), it satisfies hypotheses of \ref{coro:eulerCharOPDUniqueCompl}.
Moreover, $\PD(L)$ is the partial partition lattice, and its proper part retracts onto $\redm{\D(L)}$ by \ref{prop:uniqueDownwardhComplementationAndEP}, which is contractible (recall it has a unique minimal element $\sigma_0$ given by the partition into one-blocks).

By \ref{thm:orderedVersions}, we see that
\[ \red{\Ord\PD(L)}\simeq \bigvee_{\sigma \in \redm{\D(K)}} S^{|\sigma|-2} * ( \red{\PD(L)})_{<\sigma}.\]
Now, by property (LI), for $\sigma \in \redm{\D(K)}$ it is not hard to see that $\PD(L)_{\leq \sigma}$ is the product of the partial partition lattices on the subsets of $\sigma$.
Hence, after removing the top and minimal element of this product, we see that this is contractible if $\sigma$ contains a block of size at least two.
Hence
\[ \red{\Ord\PD(L)} \simeq S^{|\sigma_0|-2} * ( \red{\PD(L)})_{<\sigma_0}. \]
Now, $(\red{\PD(L)})_{<\sigma_0}$ is the poset of proper faces of the simplex of size $|\sigma_0|$.
Therefore, $(\red{\PD(L)})_{<\sigma_0}$ is a sphere of dimension $|\sigma_0|-2=n-2$. From the join decomposition,
we see that $\red{\Ord\PD(L)}$ has the homotopy type of a sphere of dimension $2n-3$.

For instance, for $n = 2$, $|\red{\Ord\PD(L)}|$ is the triangulation of the $1$-dimensional sphere with exactly four vertices and four edges, as depicted in \ref{fig:OPDSphereBoolean2}.

\begin{figure}
    \centering
    \begin{tikzpicture}
	\draw[draw=black, fill=black, thin, solid] (-2.00,-1.00) circle (0.1);
	\draw[draw=black, fill=black, thin, solid] (2.00,-1.00) circle (0.1);
	\draw[draw=black, fill=black, thin, solid] (0.00,1.00) circle (0.1);
	\draw[draw=black, fill=black, thin, solid] (0.00,-3.00) circle (0.1);
	\draw[draw=black, thin, solid] (-2.00,-1.00) -- (0.00,1.00);
	\draw[draw=black, thin, solid] (0.00,1.00) -- (2.00,-1.00);
	\draw[draw=black, thin, solid] (2.00,-1.00) -- (0.00,-3.00);
	\draw[draw=black, thin, solid] (0.00,-3.00) -- (-2.00,-1.00);
	\node[black, anchor=south west] at (-2.96,-1.24) {$(1)$};
	\node[black, anchor=south west] at (2.07,-1.21) {$(2)$};
	\node[black, anchor=south west] at (-0.58,1.17) {$(1,2)$};
	\node[black, anchor=south west] at (-0.56,-3.75) {$(2,1)$};
\end{tikzpicture}
    \caption{Geometric realization of $\red{\Ord\PD}$ for the Boolean lattice on $\{1,2\}$}
    \label{fig:OPDSphereBoolean2}
\end{figure}
\end{example}

The following example shows that, in the conditions of \ref{coro:eulerCharOPDUniqueCompl}, $\red{\Ord\PD(X,\sqcup)}$ might be a wedge of several spheres of different dimensions.
See \ref{sub:formedSpaces} for more details.

\begin{example}
Let $n=2$ and take a field $\KK$ of characteristic $\neq 2$.
Consider $V= \KK^2$ with bilinear form given by the identity matrix on the canonical basis.
Here we consider the category of finite-dimensional $\KK$-vector spaces with a non-degenerate bilinear form, with monoidal product $\sqcup$ given by the orthogonal product $\oPerpSymbol$.
Taking as our object $X=V$, we see that
$\Sub(X,\sqcup)$ is the poset of non-degenerate subspaces of $V$.
Thus, $\Sub(X,\sqcup)$ is uniquely downward $\sch$-complemented, where the complement is given by the orthogonal complement.
Indeed, it fulfills the hypotheses of \ref{coro:eulerCharOPDUniqueCompl} except maybe for the finiteness condition in item 1 (which is equivalent to the finiteness of $\KK$).

Then, $\red{\Ord\PD(X,\sqcup)}$ is the disjoint union of $1$-spheres indexed by the full frames of $V$.
That is, the disjoint union of the posets given in \ref{fig:OPDOrthogonal}, one for each set $\{S,S^\perp\}$, where $S$ is a non-degenerate $1$-dimensional subspace of $V$.
In particular, if $\KK$ is finite, the reduced Euler characteristic of the space is $-1$, but spheres of different dimensions arise.
If $\KK$ is infinite, then an infinite number of $1$-spheres and $0$-spheres appear.

\begin{figure}
    \centering
    \begin{tikzpicture}
	\draw[draw=black, fill=black, thin, solid] (-2.00,2.00) circle (0.1);
	\draw[draw=black, fill=black, thin, solid] (1.00,2.00) circle (0.1);
	\draw[draw=black, fill=black, thin, solid] (-2.00,-1.00) circle (0.1);
	\draw[draw=black, fill=black, thin, solid] (1.00,-1.00) circle (0.1);
	\draw[draw=black, thin, solid] (-2.00,-1.00) -- (-2.00,2.00);
	\draw[draw=black, thin, solid] (-2.00,2.00) -- (1.00,-1.00);
	\draw[draw=black, thin, solid] (1.00,-1.00) -- (1.00,2.00);
	\draw[draw=black, thin, solid] (1.00,2.00) -- (-2.00,-1.00);
	\node[black, anchor=south west] at (-2.56,-1.75) {$(S)$};
	\node[black, anchor=south west] at (0.44,-1.75) {$(S^\perp)$};
	\node[black, anchor=south west] at (-3.06,2.25) {$(S,S^\perp)$};
	\node[black, anchor=south west] at (0.44,2.25) {$(S^\perp, S)$};
\end{tikzpicture}
    \caption{Connected components of the poset $\red{\Ord\PD}$ of a vector space of dimension two equipped with a non-degenerate bilinear form}
    \label{fig:OPDOrthogonal}
\end{figure}
\end{example}

\section{Application to classical examples}
\label{sec:examples}

This section shows how our constructions give rise to well-studied objects.
In particular, for each particular case we will check properties (LI), (EX), and (CM).
We will see that usually (LI) and (EX) hold, but (CM) might fail.

\subsection{Finitely generated free groups}

Let $\C$ be the category of groups.
We take $\sqcup$ to be the coproduct of this category, which is the free product of groups, denoted by $*$.
Then, for a group $X$, we have that $\Sub(X)$ is the lattice of subgroups.
If $X$ is a free group of finite rank $n$, $\Sub(X,*) = \Freefactors(X) = \Freefactors_n$ is the poset of free factors of $X$.
Hatcher and Vogtmann studied this poset in several articles (see for example \cite{HV1, HV3, HV2}), and they proved it is Cohen-Macaulay of dimension $n$.
Note that this poset is ranked, and the poset-rank of an element is exactly its rank as a free group.
Moreover, $\Freefactors(X)$ is indeed a lattice by the Kurosh subgroup theorem, and for $H\in \Freefactors(X)$ we have $\Freefactors(X)_{\leq H} = \Freefactors(H)$.

On the other hand, if $H_1 * H_2 * \cdots * H_r\in \Freefactors(X)$, then $\{H_1,\ldots,H_r\}$ spans a Boolean lattice such that the map $I\subseteq\{1,\ldots,r\} \mapsto \langle H_i \tq i\in I\rangle \in \Freefactors(X)$ is a lattice embedding.
Therefore, $\D(X,*)$ is the poset of decompositions into free factors, and we may just denote it by $\D(X)$.
That is, its elements are sets of non-trivial free factors $\{H_1,\ldots,H_r\}$ such that $H_1 * \cdots * H_r = X$.
This coincides with the poset defined by Hatcher and Vogtmann in \cite{HV2}.
It is proved in \cite{HV2} that $\redm{\D(X)}$ is spherical of dimension $n-2$.

The poset $\Freefactors(X)$ is (non-uniquely) downward $\sch$-complemented, and it satisfies property (EX). However, property (CM) fails in this case, as the following example shows.

\begin{example}
\label{exa:freeGroupsFailCP}
Let $X$ be the free group of rank $3$ with basis $\{ a,b,c \}$.
Consider the free decompositions $X = \gen{a} * \gen{ca,cb} = \gen{a,b} * \gen{ca}$.
Then we have $\gen{a}\leq \gen{a,b}$ and $\gen{ca,cb}\geq \gen{ca}$.
However,
\[ \{ \gen{a,c} \wedge \gen{ca,cb}, \gen{a}, \gen{ca}\} \setminus \{ 1 \} = \{ \gen{a},\gen{ca} \}\notin \D(X,\sqcup).\]
Thus $X$ fails property (CM).

The same example can be generalized to show that property (CM) fails on any free group of finite rank $\geq 3$.
\end{example}

We naturally have $\PD(X)_{\leq \sigma} = \prod_{H\in \sigma} \PD(H)$ for any $\sigma \in \PD(X)$, so property (LI) also holds.
Moreover, upper intervals in $\D(X)$ are partition lattices by \ref{lm:upperIntervalsDecomp}, and lower intervals are products of smaller decomposition posets.
Thus we conclude that $\D(X)$ is Cohen-Macaulay.

\begin{theorem}
[{Hatcher-Vogtmann}]
\label{thm:HV}
Let $X$ be a free group of finite rank $n$.
Then the posets $\Freefactors(X)$ and $\D(X)$ are Cohen-Macaulay of dimension $n$ and $n-1$ respectively.
\end{theorem}

By \ref{thm:orderedVersions} we obtain the following corollary.

\begin{corollary}
Let $X$ be a free group of finite rank $n$.
Then the poset $\Ord\D(X)$ is Cohen-Macaulay of dimension $n-1$.
\end{corollary}

Next, note that every partial frame $\{H_1,\ldots,H_r\}$ of $X$ extends to a full frame.
Hence $\PF(X,*) = \F(X,*)$ and the set $\A$ of vertices of $\F(X,*)$ is exactly the set of free factors of $X$ of rank $1$.
We denote this complex simply by $\F(X)$.

For the partial basis complex, we choose the family $P = (P_H)_{H\in \A}$, where $P_H = \{ x\in H \tq \gen{x} = H\} \groupiso \{ \pm 1\}$.
Then $\Vog(X,*)$ is exactly the complex of partial bases of a free group, that is, the maximal simplices are bases of $X$.
This complex is also denoted by $\PB(X) = \PB_n$, and it was shown to be Cohen-Macaulay of dimension $n-1$ by Sadofschi Costa \cite{Ivan}.
In particular, by \ref{prop:inflationDecompositionCM}, we conclude that $\F(X)$ is Cohen-Macaulay.

\begin{theorem}
[{Sadofschi Costa}]
\label{thm:sadofschi}
Let $X$ be a free group of finite rank $n$.
Then the posets $\F(X)$ and $\PB(X)$ are Cohen-Macaulay of dimension $n-1$.
\end{theorem}

Again \ref{thm:orderedVersions} yields the following corollary.

\begin{corollary}
Let $X$ be a free group of finite rank $n$.
Then the posets $\Ord\F(X)$ and $\Ord\PB(X)$ are Cohen-Macaulay of dimension $n-1$.
\end{corollary}

However, little is known about the homotopy type of the poset of partial decompositions $\PD(X) := \PD(X,*)$.
Thus we propose the following question.

\begin{question}
  Let $X$ be a free group of rank $n$.
  Are the proper parts of $\PD(X)$ and $\Ord\PD(X)$ spherical (of dimension $2n-3$)? If so, are they Cohen-Macaulay?
\end{question}

By \ref{thm:highConnectivity}, \ref{thm:HV} and property (LI), we see that $\red{\PD(X)}$ is at least $(n-3)$-connected.

If $n = 2$ then $\red{\PD(X)} = \F(X)$, which is connected by \ref{thm:sadofschi}.
Thus $\red{\PD(X)}$ is connected for all $n\geq 2$.
Moreover, $\red{\PD(X)}$ is simply connected for $n\geq 3$:

\begin{proposition}
If $n\geq 3$ then $\red{\PD(X)}$ is simply connected.
\end{proposition}

\begin{proof}
We only need to analyze the case $n = 3$.
By \ref{coro:highConnectivity}, it is enough to show that homotopy classes of the standard apartments generate the fundamental group of $\red{\Freefactors(X)}$.
But this follows from Proposition 5.8 of \cite{Ivan} applied to the empty partial basis there.
\end{proof}

\bigskip

\subsection{Vector spaces}
\label{sub:vectorspaces}

Let $\C$ be the category of vector spaces over a field $\KK$, and let $X$ be a $\KK$-vector space of finite dimension $n$.
Here we take $\sqcup$ to be the coproduct of the category, that is, the direct sum $\oplus$ of vector spaces.
Then $\Sub(X) = \Sub(X,\oplus)$ is the poset of subspaces of $X$ since every subspace has a direct sum complement.
The order complex of $\red{\Sub(X)}$ is the Tits building of $\SL(X)$, which is spherical by the famous Solomon-Tits theorem \cite{Sol}.
Intervals in $\Sub(X)$ are subspace posets again, so $\Sub(X)$ is indeed Cohen-Macaulay of dimension $n$.
This poset is also ranked, and the rank of an element coincides with its dimension as a $\KK$-vector space.
Note that $\{V_1,\ldots,V_r\}\subseteq \Sub(X)$ are subspaces in internal direct sum if and only if $\sum_i \dim(V_i) = \dim(V_1+\cdots+V_r)$.
Hence $\D(X,\oplus)$ is the poset of sets of non-zero subspaces which are in internal direct sum and span $X$.
We denote it simply by $\D(X)$.
This poset was shown to be shellable for finite fields in \cite{Welker}, and in particular Cohen-Macaulay.
In the case $X = \GF{q}^n$, the reduced Euler characteristic of $\D(X)$ is

\begin{equation}
\label{eq:eulerCharDecompVS}
 \frac{(-1)^n}{n} \cdot \prod_{i=1}^{n-1} (q^i-1) \cdot f_n(q),
\end{equation}
\noindent
where $f_n(q)$ is a monic polynomial of degree $\binom{n}{2}$ and positive integer coefficients related to a $q$-analog of Catalan numbers.
The coefficient sequences of these polynomials seem to be always unimodal (see \cite[pp.241-242]{Welker}).

On the other hand, for arbitrary fields, we show later that $\red{\D(X)}$ is spherical by invoking the results from \cite{Charney} on the Charney poset $\Charney(X,\oplus)$ (see  \ref{def:charney}).

Similar to the case of the free factor poset, $\F(X) := \F(X,\oplus) = \PF(X,\oplus)$ and we take $\PB(X) := \Vog(X,\oplus)$ to be the complex of partial bases of $X$ (i.e., linearly independent sets of vectors).
For finite vector spaces $X$, the complexes $\F(X)$ and $\PB(X)$ are matroid complexes, which are shellable and in particular Cohen-Macaulay of dimension $n-1$ (see \cite{Bjo92} and \ref{prop:inflationDecompositionCM}).
For infinite fields, this follows from Theorem 2.6 of \cite{vdK} since fields satisfy the Bass' stable range condition $\mathrm{SR}_2$.
By item (i) of this theorem, we see that $\Ord \PB(X)$ is spherical of dimension $n-1$, and by (ii) there also the intervals $\Ord \PB(X)_{>z}$, for $z\in \Ord \PB(X)$, are spherical.
This implies that $\Ord \PB(X)$ is Cohen-Macaulay of dimension $n-1$.
Then by \ref{thm:orderedVersions}, $\PB(X)$ is Cohen-Macaulay of dimension $n-1$.
By \ref{prop:inflationDecompositionCM}, $\F(X)$ is Cohen-Macaulay of dimension $n-1$.
Again, by \ref{thm:orderedVersions}, we see that $\Ord \F(X)$ is Cohen-Macaulay of dimension $n-1$.

Finally, $\PD(X) := \PD(X,\oplus)$ is the poset of partial direct sum decompositions of $X$, that is, its elements are sets of non-zero subspaces of $X$ which are in internal direct sum.
An unpublished work by Hanlon, Hersh and Shareshian \cite{HHS} proves that $\red{\PD(X)}$ is Cohen-Macaulay of dimension $2n-3$, and in particular spherical, if $X = \GF{q}^ n$.
Moreover, they show that the reduced Euler characteristic of this poset is

\begin{equation}
\label{eq:eulerCharPDVS}
- \frac{1}{n} \cdot q^{ \binom{n}{2} } \cdot \prod_{i=1}^{n-1} (q^i-1).
\end{equation}

Note that this formula is closely related to the Euler characteristic of the join of the poset of proper non-zero subspaces and $\D(X)$. The former has Euler characteristic $(-1)^n q^ {\binom{n}{2}}$, and the Euler characteristic of the latter is given in \ref{eq:eulerCharDecompVS}. 
The join of two posets has as its Euler characteristic the negative of the
product of the two Euler characteristics, so in \ref{eq:eulerCharPDVS} the polynomial $f_n(q)$ is missing.

For finite-dimensional vector spaces $X$ over arbitrary fields, the poset $\PD(X)$ is homotopy equivalent to $\CB(X)$ (see \cite{BPW}), which was proved to be $(2n-4)$-connected in \cite{MPW}.
Hence, $\PD(X)$ is spherical of dimension $2n-3$.
By \ref{thm:vectorspaces} and \ref{thm:orderedVersions}, we can also conclude that $\red{\Ord\PD(X)}$ is spherical.

To conclude Cohen-Macaulayness for $\PD(X)$, we must show that intervals are spherical.
Note that property (LI) holds, so if $\sigma\in \PD(X)$, then $\PD(X)_{\leq \sigma} \cong \prod_{S\in \sigma} \PD(S)$ is spherical by induction.
However, upper intervals do not have a clear description, unless we have full decomposition.
That is, if $\sigma\in \D(X)$ then $\PD(X)_{\geq \sigma} = \D(X)_{\geq \sigma}$ is the partition lattice on $\sigma$, whose proper part is spherical.

Finally, note $\Sub(X)$ is downward $\sch$-complemented for $\sqcup = \oplus$, and it satisfies properties (EX) and (CM).
Hence by \ref{prop:charney}, the Charney complex $\Delta \Charney(X,\oplus)$ is isomorphic to the poset $( \redm{\Ord \D(X,\oplus)} )^{\op}$.
Since the $\Charney(X,\oplus)$ is spherical by \cite{Charney}, we conclude that $\redm{\Ord \D(X)}=\redm{\Ord \D(X,\oplus)}$ is spherical of dimension $n-2$.
Thus $\redm{\D(X)}$ is spherical of dimension $n-2$ by \ref{thm:orderedVersions}.

We conclude the following theorem, whose pieces follow from the references and arguments in the preceding paragraphs. 

\begin{theorem} \label{thm:vectorspaces}
Let $X$ be a vector space of finite dimension $n$ over a field $\KK$.
Then the posets $\PB(X)$, $\Ord\PB(X)$, $\F(X)$, $\Ord\F(X)$, $\Sub(X)$, $\D(X)$ and $\Ord\D(X)$ are Cohen-Macaulay.

The posets $\red{\PD(X)}$ and $\red{\Ord\PD(X)}$ are spherical of dimension $2n-3$.
If $\KK$ is a finite field then $\PD(X)$ and $\Ord\PD(X)$ are Cohen-Macaulay.
\end{theorem}

As we mentioned above, the upper intervals in the partial decomposition posets are harder to compute.
We leave open then the following question.

\begin{question}
  Are the posets $\PD(X)$ and $\Ord\PD(X)$ Cohen-Macaulay, for any finite-dimensional vector space $X$?
\end{question}

We close this section with a brief discussion on the Euler characteristic of frames, partial decompositions and some of the ordered versions.

For $X=\GF{q}^n$, it is easily checked that the Euler characteristic of $\Ord\F(X)$ is a $q$-analog of the derangement numbers. \ref{tab:qder} provides a list for small $n$, and \ref{tab:eulerFrames} contains the values for the unordered version.
These polynomials do not seem to be well-behaved and we do not know of any other
proposed $q$-analog of derangement numbers coinciding with them. Similarly, the 
reduced Euler characteristic of $\Ord\D(X)$ does not suggest any nice pattern.  

\begin{figure}[ht]
\centering
\begin{tabular}{c|c}
        $n$ & $ (-1)^{n-1} \cdot \tilde{\chi}\left(\Ord\F(\GF{q}^n)\right)$ \\
        \hline
        $2$ & $q^2$\\
        $3$ & ${q}^{2} \left( q+1 \right) \left( {q}^{3}+{q}^{2}-1 \right)$\\
        $4$ & ${q}^{2} \left( {q}^{6}+{q}^{5}-{q}^{2}-q+1 \right) \left( {q}^{2}+q+1 \right) ^{2}$\\
        $5$ & $
        {q}^{2} \left( q+1 \right) \left( {q}^{2}+1 \right) \left( {q}^{15}+3\,{q}^{14}+5\,{q}^{13}+6\,{q}^{12}+5\,{q}^{11}\right.$\\
        & $\left.+2\,{q}^{10}-2\,{q}^{9}-5\,{q}^{8}-5\,{q}^{7}-3\,{q}^{6}+2\,{q}^{4}+2\,{q}^{3}+{q}^{2}-1\right)$
    \end{tabular}
\caption{Euler characteristic of $\Ord\F(\GF{q}^n)$}
\label{tab:qder}
\end{figure}

\begin{figure}[ht]
    \centering
    \begin{tabular}{c|c}
        $n$ & $ (-1)^{n-1}n! \cdot \tilde{\chi}\left(\F(\GF{q}^n)\right)$ \\
        \hline
        $2$ & $q(q-1)$\\
        $3$ & $q(q-1)(q^2-1)(q^2+3q+3)$\\
        $4$ & $q(q-1)(q^3-1)( q^7+4q^6+9q^5+12q^4+8q^3-4q^2-12q-12 )$\\
        $5$ & $q(q-1)^2(q^4-1)(q^{13}+6q^{12}+20q^{11}+49q^{10}+94q^9+145q^8$\\
        & $+180q^7+170q^6+105q^5-100q^3-140q^2-120q-60)$
    \end{tabular}
    \caption{Euler characteristic of $\F(\GF{q}^n)$}
    \label{tab:eulerFrames}
\end{figure}

In the following theorem, we compute the reduced Euler characteristic of  $\Ord\PD(X)$ for an $n$-dimensional vector space $X$ over the finite field $\GF{q}$.

\begin{theorem}
\label{thm:eulerCharOPDVectorSpaces}
For $n\geq 1$ and $q$ a prime power, if $X = \GF{q}^n$ then $\tilde{\chi}(\red{\Ord\PD(X)}) = -q^{n(n-1)}$.
\end{theorem}

\begin{proof}
We argue by induction on $n$.
The cases $n = 1$ and $n = 2$ are straightforward to verify. Hence we assume $n\geq 3$.
Let $\U = \red{\Ord\PD(X)}\setminus \Ord\D(X)$.
Then, by item 2 of Theorem \ref{thm:orderedVersions} we have
\begin{align*}
\tilde{\chi}(\red{\Ord\PD(X)}) & = -1 - \sum_{\sigma\in \red{\Ord\PD(X)}\setminus \Ord\D(X)} \tilde{\chi}(\red{\Ord\PD(X)}_{<\sigma}) - \sum_{\sigma\in \redm{\Ord\D(X)}} \tilde{\chi}(\red{\Ord\PD(X)}_{<\sigma})\\
& = \tilde{\chi}(\U) - \sum_{\sigma\in \redm{\Ord\D(X)}} \tilde{\chi}(\red{\Ord\PD(X)}_{<\sigma})\\
& = \tilde{\chi}(\red{\S(X)}) - \sum_{\sigma\in \redm{\Ord\D(X)}} \tilde{\chi}(\red{\Ord\PD(X)}_{<\sigma})\\
& = (-1)^nq^{\binom{n}{2}} - \sum_{\sigma\in \redm{\Ord\D(X)}} \tilde{\chi}(\red{\Ord\PD(X)}_{<\sigma}).
\end{align*}
Let
\[ \alpha := \sum_{\sigma\in \redm{\Ord\D(X)}} \tilde{\chi}(\red{\Ord\PD(X)}_{<\sigma}).\]
Thus $\tilde{\chi}(\red{\Ord\PD(X)}) = -q^{n(n-1)}$ if and only if
\[ \alpha - q^{n(n-1)} = (-1)^nq^{\binom{n}{2}}.\]
We write $S\in \sigma$ if $S$ lies in the underlying set of $\sigma \in \Ord\D(X)$.
Now we expand $\alpha$ by taking into account the inductive hypothesis:
\begin{align*}
\alpha & = \sum_{\sigma\in \redm{\Ord\D(X)}} \tilde{\chi}(\red{\Ord\PD(X)}_{<\sigma})\\
& = \sum_{\sigma\in \redm{\Ord\D(X)}} \tilde{\chi}\bigg( \red{ \big( \prod_{S\in \sigma} \Ord\PD(S) \big)} \bigg)\\
& = \sum_{\sigma\in \redm{\Ord\D(X)}} \prod_{S\in \sigma}\tilde{\chi}\big( \red{\Ord\PD(S)} \big)\\
& = \sum_{\sigma\in \redm{\Ord\D(X)}} \prod_{S\in \sigma}(-q^{\dim S(\dim S -1)}).
\end{align*}
Fixing $2\leq m\leq n$ and positive integers $k_1,\ldots,k_m$ such that $k_1+\cdots+k_m=n$, the number of ordered decompositions $\sigma=(S_1,\ldots,S_m)\in \redm{\Ord\D(X)}$ such that $\dim S_i = k_i$ equals 
\[\frac{|\GL_n(q)|}{\displaystyle{\prod_{i=1}^m |\GL_{k_i}(q)|}}.\]
Denote by $\Ord_{n,m}$ the set of sequences $\mathbf{k} = (k_1,\ldots,k_m)$ of positive integers with $\displaystyle{\sum_{i=1}^ m} k_i = n$.
Then
\begin{equation}
    \label{eq:alpha1}
    \begin{split}
    \alpha - q^{n(n-1)} & = \sum_{m=1}^n \,\sum_{\mathbf{k} \in \Ord_{n,m}} \frac{|\GL_n(q)|}{\displaystyle{\prod_{i=1}^ m |\GL_{k_i}(q)|}} \,\prod_{i=1}^m(-q^{k_i(k_i-1)})\\
   & = \sum_{m=1}^n (-1)^m \sum_{\mathbf{k} \in \Ord_{n,m}} \frac{|\GL_n(q)|}{\displaystyle{\prod_{i=1}^ m |\GL_{k_i}(q)|}} \,\,q^{-n + {\sum_{i=1}^ m} k_i^2}.
    \end{split}
\end{equation}
On the other hand, the Euler characteristic of the poset of subspaces is given by the following formula
\begin{equation}
    \label{eq:solomonIdentity}
    (-1)^n q^{\binom{n}{2}} = \sum_{m=1}^n (-1)^m \sum_{\mathbf{k}\in \Ord_{n,m}} \frac{|\GL_n(q)|}{\displaystyle{\prod_{i=1}^ m |\GL_{k_i}(q)|}} \, q^{ -{\sum_{i<j}k_ik_j}},
\end{equation}
which arises by counting flags of subspaces $0<V_1<\cdots<V_{m-1}<X$.
This is also known as the Solomon identity \cite[Corollary 1.1]{Solomon}.
Hence we must show that \ref{eq:alpha1} and \ref{eq:solomonIdentity} coincide.

Note that $-n+\displaystyle{\sum_{i=1}^m k_i^2} = n^2-n - 2 \displaystyle{\sum_{1 \leq i<j \leq m}} k_ik_j$.
Thus in \ref{eq:alpha1} we get
\begin{equation}
    \label{eq:alpha2}
    \alpha-q^{n(n-1)} = \sum_{m=1}^n (-1)^m \sum_{\mathbf{k} \in \Ord_{n,m}} \frac{|\GL_n(q)|}{\displaystyle{\prod_{i=1}^m |\GL_{k_i}(q)|}} \,q^{n(n-1) - 2 \sum_{i<j} k_ik_j}
\end{equation}
If we let $\beta := q^{-n(n-1)}(\alpha-q^{n(n-1)})$, we must show that $\beta = (-1)^n q^{-\binom{n}{2}}$.
This identity now follows by replacing $q$ by $q^{-1}$ in \ref{eq:solomonIdentity} and observing that
\[ ``|\GL_n(q^{-1})|" = q^{-\binom{n}{2}} \prod_{i=1}^n (q^{-i}-1) = (-1)^n q^{-\frac{n(3n-1)}{2}}\, |\GL_n(q)|.\]
\end{proof}

\subsection{Vector spaces with forms}
\label{sub:formedSpaces}

In this section we consider vector spaces equipped with forms. 
This case will require some
preparation before we can turn our attention to the concrete posets and complexes.

Recall that for a vector space $V$ over a field $\KK$, a form is just a map $\Psi:V\times V\to \KK$.
In our setting, we require $\Psi$ to be bi-additive and such that $\Psi(-,v)$ is $\KK$-linear for all $v\in V$.
We call this property \textit{almost $\KK$-bilinear}.
The form $\Psi$ is reflexive if for all $v,w\in V$, we have that $\Psi(v,w) = 0$ if and only if $\Psi(w,v) = 0$.
The radical $\Rad(V) = \{v\in V\tq v\perp w $ for all $ w\in V\}$ and the orthogonal complement $S^\perp = \big\{\,v\in V\tq v\perp w $ for all $ w\in S\,\big\}$  of a subset $S\subseteq V$ are subspaces of $V$ by $\KK$-linearity in the first variable.
Here $v\perp w$ if $\Psi(v,w) = 0$.
A subspace $S\leq V$ is non-degenerate if its radical $\Rad(S) = S\cap S^\perp$ is zero.
In particular, $V$ is non-degenerate if $\Rad(V) = 0$. In this case, we also say $\Psi$ is
non-degenerate.
If $S\leq V$, we regard $S$ as a vector space with a form by restricting $\Psi$ to $S\times S$.
We usually denote this restriction by $\Psi_S$, or simply $\Psi$ when it is clear from the context.
Note that if $\Psi$ is reflexive, then so is $\Psi_S$ for any $S\leq V$.


Consider the category $\Ivec_\KK$ whose objects are pairs $(V,\Psi_V)$, where $V$ is a vector space over $\KK$ and $\Psi_V:V\times V\to \KK$ is an almost $\KK$-bilinear form, and maps are isometries.
That is, a map $f:(V,\Psi_V)\to (W,\Psi_W)$ is a linear map $f:V\to W$ such that $\Psi_W(f(v),f(w)) = \Psi_V(v,w)$ for all $v,w\in V$.
We also consider the full subcategory $\RIvec_\KK$ of reflexive forms.

Since we want to work with posets of subobjects, we characterize in the following proposition the monomorphisms of these categories.

\begin{lemma}
\label{lm:monoIVec}
  In the categories $\Ivec_\KK$ and $\RIvec_\KK$, monomorphisms are exactly the injective isometries (that is, monomorphisms when regarded as linear maps).

  In particular, if $\C = \Ivec_{\KK}$ or $\C = \RIvec_{\KK}$ and
  $(V,\Psi)\in \C$ then
  $\Sub_{\C}(V,\Psi)$ is the poset of subspaces of $V$.
\end{lemma}

\begin{proof}
  The proof is the same for both categories.
  It is clear that if $f:(V,\Psi_V)\to (W,\Psi_W)$ is an injective isometry, then $f$ is a monomorphism of the category.
  Conversely, assume that $f$ is a monomorphism in the category, and suppose that $f(v) = 0$.
  We prove that $v=0$.
  Since $f$ is an isometry, $\Psi_V(\lambda v, \lambda' v) = 0$ for all $\lambda,\lambda'\in\KK$.
  Let $S = \gen{v}$ with $\Psi_S = \Psi_V|_S$, which is the null form (and in particular it is reflexive).
  Consider the maps $g_1,g_2:S\to V$ given by $g_1(v) = v$ and $g_2(v) = 0$.
  Then $g_1$ and $g_2$ are isometries and $fg_1 = 0 = fg_2$.
  Thus $g_1,g_2:(S,\Psi_S)\to (V,\Psi_V)$ are maps in $\Ivec_\KK$ (and also in $\RIvec_{\KK}$ if $\Psi_V$ is reflexive).
  Since $fg_1=fg_2$, we conclude that $g_1=g_2$, that is, $v=0$.
\end{proof}

We will also need the following dimension formula for orthogonal complements of subspaces in a vector space with a reflexive almost $\KK$-bilinear form.
The proof follows the same ideas as the original proof for $\sigma$-sesquilinear forms, and we leave the details to the reader.

\begin{proposition}
[{Dimension-formula}]
\label{prop:dimFormulaReflexiveVS}
  Let $V$ be a finite-dimensional vector space and let $\Psi$ be a reflexive almost $\KK$-bilinear form on $V$.
  Then, for all $S\leq V$ we have
  \begin{equation}
      \label{eq:dimFormula}
      \dim(S) + \dim(S^\perp) = \dim(V) + \dim(S\cap \Rad(V)). 
  \end{equation}
In particular, $(S^\perp)^\perp = S + \Rad(V)$ and $\Rad(S^\perp) = \Rad(S) + \Rad(V)$.
\end{proposition}





Having reviewed basic properties of the relevant categories, we now define a suitable monoidal product.
For any two objects $(V,\Psi_V),(W,\Psi_W)$ of $\Ivec_{\KK}$ there is a natural orthogonal sum $(V\oplus W, \Psi_V\operp \Psi_W)$ where 
\[\big(\Psi_V\operp\Psi_W\big)(v\oplus w, v'\oplus w') = \Psi_V(v,v') + \Psi_W(w,w'),\quad v,v'\in V \, w,w'\in W.\]
We denote this construction by $(V,\Psi_V) \operp (W,\Psi_W)$.
The following lemma is a simple consequence of the definitions.

\begin{lemma}
  For any two objects $(V,\Psi_V),(W,\Psi_W)$ of $\Ivec_{\KK}$, $(V\oplus W, \Psi_V\operp \Psi_W)$ is almost $\KK$-bilinear.
  Moreover, if $i_V: V\to V\oplus W$, $v\mapsto v\oplus 0$ and $i_W: W\to V\oplus W$, $w\mapsto 0\oplus w$ are the canonical embeddings, then $i_V, i_W$ are injective isometries and $V\leq W^\perp$ and $W\leq V^\perp$ in $(V\oplus W, \Psi_V\operp \Psi_W)$.
\end{lemma}

Note that $(V\oplus W, \Psi_V\operp \Psi_W)$ might not be reflexive even if $\Psi_V$ and $\Psi_W$ are.
Nevertheless, this is the case for many important examples such as $\sigma$-sesquilinear forms.

\begin{example}
\label{ex:nonReflexiveProduct}
Let $\KK$ be a field of characteristic different from $2$.
Let $(\KK,\Psi_1)$ be the $1$-dimensional vector space with $\Psi_1$ the product of $\KK$, that is, $\Psi_1(x,y) = xy$.
Consider also $(\KK^2,\Psi_2)$, the symplectic space of dimension $2$.
Here $\Psi_2( (x_1,x_2), (y_1,y_2)) = x_1y_2-x_2y_1$.
Then $(\KK,\Psi_1)\operp (\KK^2,\Psi_2) = (\KK^3,\Psi_3)$ where
\[ \Psi_3( (x_1,x_2,x_3), (y_1,y_2,y_3) ) = x_1y_1 + x_2y_3 - x_3y_2.\]
Therefore, both $\Psi_1,\Psi_2$ are $\KK$-bilinear and reflexive, and $\Psi_3$ is $\KK$-bilinear but not reflexive.
For instance, take $x,y\in\KK^3$ with $x_1=y_1=1$, $x_3=y_2=1$ and $x_2=y_3=0$.
Then
\[ \Psi(x,y) = 1 + 0 - 1 = 0,\]
\[ \Psi(y,x) = 1 + 1 - 0 = 2\neq 0.\]
\end{example}

It is also important to observe that $(V,\Psi_V)\operp (W,\Psi_W)$ is not a coproduct in the category $\Ivec_\KK$.
For example, it fails the universal property for the orthogonal sum of two non-orthogonal lines embedded in the real plane.
In fact, the coproduct of ``most'' pairs of objects of this category does not exist.
Nevertheless, it is easily seen that the operation $\sqcup := \operp$ turns $\Ivec_\KK$ into an ISM-category.
However, by \ref{ex:nonReflexiveProduct}, the product $\operp$ does not restrict to a monoidal product in $\RIvec_\KK$.

The following proposition describes the $\operp$-complemented subspaces of $V$, where $(V,\Psi)\in \Ivec_{\KK}$.
We identify subobjects with subspaces by restricting the form $\Psi$.

\begin{proposition} \label{prop:formsposet}
  Let $(V,\Psi)\in \Ivec_{\KK}$ and $S,W\leq V$ such that $S\cap W = 0$.
  The following are equivalent:
  \begin{enumerate}
    \item $S,W\in \Sub(V,\Psi)$ are $\operp$-complements of each other in $\Sub(S\oplus W,\Psi)$,
    \item $W\leq S^\perp$, $S\leq W^\perp$.
  \end{enumerate}

  Moreover, if $\Psi$ is reflexive, then $S\leq V$ is $\operp$-complemented if and only if $\Rad(S)\leq \Rad(V)$.
  In particular, non-degenerate subspaces are $\operp$-complemented, and if $\Psi$ is also non-degenerate then $\Sub( (V,\Psi) ,\operp)$ is the poset of non-degenerate subspaces.
\end{proposition}

\begin{proof}
  The equivalences in the first part follow easily from the definitions.

  Suppose that $\Psi$ is reflexive and let $S \leq V$.
  First we assume that $\Rad(S) \leq \Rad(V)$ and show that $S$ is $\operp$-complemented.
  Let $W\leq S$ be a subspace complement for the radical, 
  that is, $S = \Rad(S) \oplus W$. We claim that $\Rad(W) = 0$.
  By construction, for a given $s\in S$ there are $r \in \Rad(S)$ and $w \in W$ such that $s = r+w$.
  Also for any $v \in \Rad(W)$ we have $\Psi(v,w) = 0$ and by reflexivity  
  $\Psi(w,v) = 0$. It follows that
  \[ \Psi(s,v) = \Psi(r,v) + \Psi(w,v) = \Psi(w,v) = 0. \]
  This shows that $\Rad(W)\leq \Rad(S)\cap W = 0$ and hence $W\cap W^\perp = \Rad(W) = 0$.

  Decompose $W^ \perp = \Rad(V) \oplus W'$ and $\Rad(V) = \Rad(S) \oplus R$.
  We claim that $V = S\oplus (R\oplus W')$.
  By \ref{prop:dimFormulaReflexiveVS},
  \[ \dim(W) + \dim(W^\perp) = \dim(V) + \dim(W\cap \Rad(V)).\]
  Since $W\cap \Rad(V) \leq \Rad(W) = 0$, we see that 
  \[ \dim(W) + \dim(W^\perp) = \dim(V).\]
  Hence we obtain
  \[ V = W\oplus W^\perp = W \oplus \Rad(S) \oplus R \oplus W' = S \oplus (R\oplus W'). \]
  This shows that $R\oplus W'$ is a complement for $S$.
  Finally, note that $R\oplus W'\perp S$ and $(R\oplus W')\cap S = 0$, and by the first part of the proposition we conclude that $R\oplus W'$ is a $\operp$-complement for $S$.

  Now we show that if $S$ is $\operp$-complemented then $\Rad(S) \leq \Rad(V)$. 
  Let $W$ be a $\operp$-complement of $S$ in $\Sub(V,\Psi)$.
  By definition we have $S\oplus W = V$ and the first part of the proposition implies that
  $S\leq W^\perp$ and $W\leq S^\perp$.
  Since $S\leq \Rad(S)^\perp $ and $W\leq \Rad(S)^\perp $, we conclude by reflexivity 
  that $V = S\oplus W \leq \Rad(S)^\perp$ and hence $\Rad(S) \leq \Rad(V)$.
\end{proof}

\begin{remark}
Note that if $(V,\Psi)$ is reflexive, $S\in \Sub((V,\Psi),\operp)$ and $T\in\Sub((S,\Psi),\operp)$, then $\Rad(T)\leq \Rad(S)\leq \Rad(V)$ and so $T\in \Sub((V,\Psi),\operp)$.
This means that $\Sub((V,\Psi),\operp)_{\leq S} = \Sub((S,\Psi),\operp)$.
\end{remark}

Next, we identify the elements of $\D((V,\Psi),\operp)$ under a suitable
hypothesis. For that we call a collection $\sigma = \{V_1,\ldots,V_r\}$ of subspaces of
$V$ an orthogonal decomposition of $V$ if $V$ is the (internal) direct sum of the $V_1,\ldots, V_r$
and $V_i \perp V_j$ for all $i\neq j$. Any subset of an orthogonal decomposition is called
a partial orthogonal decomposition.
In $\D(V,\Psi)$ we collect all orthogonal decompositions of $V$ and in $\PD(V,\Psi)$
all partial orthogonal decompositions of $V$. Both posets are ordered by refinement and hence are subposets of $\PD(V)$ and $\D(V)$.
The following is now an immediate consequence of \ref{prop:formsposet}.

\begin{corollary} \label{cor:lowerinvervalpd}
  Let $(V,\Psi) \in \Ivec_\KK$ with $V$ finite-dimensional and $\Psi$ reflexive and non-degenerate. 
  Then we have:
  \begin{enumerate}
    \item The poset $\S((V,\Psi),\operp)$ satisfies properties (EX) and (CM), and it is uniquely downward $\sch$-complemented for $\sqcup = \operp$.
    \item $\PD((V,\Psi),\operp) = \PD(V,\Psi)$ and $\D((V,\Psi),\operp) = \D(V,\Psi)$.
    \item Property (LI) holds, i.e., $\PD((V,\Psi),\operp)_{\leq \sigma} = \prod_{S\in\sigma} \PD( (S,\Psi_S),\operp)$.
    \item For $S\in \S((V,\Psi),\operp)$, we have $\S((V,\Psi),\operp)_{\leq S} = \S((S,\Psi|_S),\operp)$.
  \end{enumerate}
\end{corollary}

Using \ref{prop:formsposet} and \ref{prop:uniqueDownwardhComplementationAndEP}, 
as well as by \ref{lm:GMapOrdDtoS}, \ref{lm:isoGMap} and \ref{lm:GMapSurjective}
we conclude:

\begin{corollary}
    \label{cor:uniquecomplementform}
    Let $(V,\Psi) \in \Ivec_\KK$ with $V$ finite-dimensional and $\Psi$ reflexive and non-degenerate.
    Then we have:
    \begin{enumerate}
        \item $\red{\PD(V,\Psi)} \simeq \redm{\D(V,\Psi)}$. 
        \item $\Ord\redm{\D(V,\Psi)}\cong \big(\Delta \red{\Sub((V,\Psi),\operp)}\big)^{\op}$.
        \item $\tilde{\chi}( \red{\Ord\PD(V,\Psi)} ) = -1$.
    \end{enumerate}
\end{corollary}

\begin{proof}
Item 1 follows by \ref{prop:uniqueDownwardhComplementationAndEP}, item 2 by \ref{lm:GMapOrdDtoS}, \ref{lm:isoGMap} and \ref{lm:GMapSurjective}, and item 3 by \ref{coro:eulerCharOPDUniqueCompl}.
In all cases, we invoke the properties described in \ref{cor:lowerinvervalpd}.
\end{proof}

\begin{remark}
By \ref{cor:uniquecomplementform}, the posets $\red{\Ord\PD(V,\Psi)}$ and $\red{\PD(V,\Psi)}$ cannot be spherical since the latter is homotopy equivalent to $\redm{\D(V,\Psi)}$ which has dimension $n-2$ and is not contractible in general, and $\red{\Ord\PD(V,\Psi)}$ is a wedge of $\red{\PD(V,\Psi)}$ and suspension of intervals.
In particular, $\Ord\PD(V,\Psi)$ and $\PD(V,\Psi)$ are not Cohen-Macaulay.
\end{remark}

For general $\Psi$, we do not know of results about the homotopy type or Cohen-Macaulay\-ness of the posets
discussed in this paper. In the following theorem we collect a few known results for $\sigma$-Hermitian and symplectic forms.
Recall that for a field involution $\sigma$ of $\KK$ and $\epsilon\in \{1,-1\}$, an almost $\KK$-bilinear form $\Psi$ is $(\sigma,\epsilon)$-Hermitian if $\Psi(v,w) = \epsilon \sigma( \,\Psi(w,v)\,)$. 
Note that $(\sigma,\epsilon)$-Hermitian forms are always reflexive.
If $\epsilon = 1$ we just say that $\Psi$ is $\sigma$-Hermitian.
When $\epsilon = 1$ and $\sigma$ has order $2$, we get a Hermitian form, and if in addition $\Psi$ is non-degenerate we say that $\Psi$ is a unitary form and $V$ a unitary space.
If $\epsilon=-1$, $\sigma$ is the identity and $\Psi(v,v) = 0$ for all $v\in V$, we get a symplectic form.
For $\epsilon = 1$ and $\sigma$ the identity map, we get an orthogonal form.

\begin{theorem} \label{thm:formposets} ~

\begin{enumerate}
\item 
Let $\KK$ be a field and $\sigma$ an automorphism of $\KK$ of order one or two.
If $\sigma$ is the identity then we assume that $\KK$ has odd characteristic.
Let $V$ be an $(n+1)$-dimensional $\KK$-vector space and $\Psi$ a non-degenerate $\sigma$-Hermitian form. 
If $\KK$ is a finite field $\GF{q}$ then assume $2^{n}<q$ in the case $\sigma$ is the identity
and $2^{n-1}\,(\sqrt{q}+1)<q$ in case $\sigma$ has order $2$. 
Then the following posets are Cohen-Macaulay:
\begin{enumerate}
    \item \cite{DGM} $\Sub((V,\Psi),\operp)$.
    \item $\Ord\D(V,\Psi)$.
    \item $\D(V,\Psi)$.
    \end{enumerate}
\item Let $\KK$ be a finite field of size $>2$ and $V$ a $\KK$-vector space of dimension $2n$ equipped with
a non-degenerate symplectic form $\Psi$. 
Then the following posets are Cohen-Macaulay: 
\begin{enumerate}
    \item \cite{Das} $\Sub((V,\Psi),\operp)$.
    \item $\Ord\D(V,\Psi)$.
    \item $\D(V,\Psi)$.
\end{enumerate}
\end{enumerate}
 In both cases $\red{\Ord\PD(V,\Psi)}$ is homotopy equivalent to a wedge of spheres.  
\end{theorem}
\begin{proof}
    The conclusions on $\Sub((V,\Psi),\operp)$ are the main results of \cite{DGM} and \cite{Das}. 
    By item 2 of \ref{cor:uniquecomplementform} and \ref{lm:intervalsOrdered}, we have that (a) implies (b) both in items 1 and 2.
    By induction on the dimensions of $V$, \ref{thm:orderedVersions} implies that 
    $\redm{\D(V,\Psi)}$ is spherical. By \ref{cor:lowerinvervalpd}
    and \ref{lm:upperIntervalsDecomp} all proper intervals are spherical. Thus $\D(V,\Psi)$ is Cohen-Macaulay in both cases. 
    Again \ref{thm:orderedVersions} shows that $\red{\Ord\PD(V,\Psi)}$ is homotopy equivalent to a wedge of spheres. 
\end{proof}

The restriction to characteristic $\neq 2$ in item 1 of \ref{thm:formposets} is not present in 
\cite[Main Theorem]{DGM}. Nevertheless, in \cite{DGM} the Main Theorem is derived from Theorem 4.4 there, which applies to Phan geometries, and it requires the existence of non-isotropic vectors (see \cite[Definition 3.2]{DGM}).
The latter is not guaranteed in characteristic $2$ and for $\sigma$ the identity. We consider the validity of item 1 of \ref{thm:formposets} for characteristic $2$ and $\sigma$ the identity open.
Even if the answer is affirmative, the dimension of the spheres provided in 
\cite[Main Theorem]{DGM} needs some clarification in that case.
Note that if $\KK$ is a field of characteristic $2$ and $\sigma$ the identity, 
then $\Psi(v,w) = \sum_{i=1}^n v_iw_{n-i+1}$ gives a non-degenerate $\sigma$-Hermitian form (in the sense of \cite{DGM}) over a vector space of even dimension $n$.
However, $(V,\Psi)$ is also a symplectic space, so the dimension of the order complex of 
$\S(V,\Psi)$ is not $n-2$, but $\frac{n}{2}-2$.

On the other hand,
Das \cite{Das} works with non-degenerate symplectic forms over finite fields of size $> 2$, so this is where our restriction on the size of the field comes from.
Indeed, we expect these results to be true also for infinite fields.
Furthermore, by the proof of item 2 of \ref{thm:formposets}, Cohen-Macaulayness will follow for the decomposition posets once
the following conjecture is resolved.

\begin{conjecture}
    Let $\KK$ be any field and $V$ a finite-dimensional $\KK$-vector space with 
    a non-degenerate symplectic form $\Psi$. Then 
    $\Sub((V,\Psi),\operp)$ is Cohen-Macaulay.
\end{conjecture}

We also mention that some vanishing results for the rational homology of the frame complex of unitary and symplectic spaces were obtained in \cite{P23,PW22}, by using Garland's method.
Note that, by \ref{cor:lowerinvervalpd} and \ref{coro:collapseFrames}, $\F(V,\Psi):=\F(V,\Psi,\operp)$ is homotopy equivalent to the poset $\hat{\F}(V,\Psi)$ of non-empty frames of size $\neq n-1$.

\begin{theorem}
Let $V$ be a finite-dimensional vector space over $\KK$.

\begin{enumerate}
\item If $(V,\Psi)$ is a unitary space of dimension $n$ and $\KK = \GF{q^2}$, then the following hold:
\begin{enumerate}
    \item If $n < q+1$ then $\hat{\F}(V,\Psi)$, $\red{\Sub((V,\Psi),\operp)}$ and $\redm{\D(V,\Psi)}$ are Cohen-Macaulay of dimension $n-2$ over $\QQ$.
    \item For $n \geq 11$, the posets $\hat{\F}(V,\Psi)$, $\red{\Sub(V,\Psi)}$ and $\redm{\D(V,\Psi)}$ are $([2n/3]-1)$-connected over $\QQ$.
    \item If $n \geq q(q-1)+1$ then $\hat{\F}(V,\Psi)$ is not Cohen-Macaulay over $\QQ$.
\end{enumerate}

\item If $(V,\Psi)$ is a non-degenerate symplectic space of dimension $2n$ and $\KK = \GF{q}$, then the following hold:
\begin{enumerate}
    \item If $n<q+3$ then $\hat{\F}(V,\Psi)$ is Cohen-Macaulay of dimension $n-2$ over $\QQ$.
    \item If $n\geq 7$ then $\hat{\F}(V,\Psi)$ is $[n/2]$-connected over $\QQ$.
    \item If $n >q^2(q^2+1) + n(n-2)q^{-4}(q^4+q^2+1)^{-1}$ then $\hat{\F}(V,\Psi)$ is not Cohen-Macaulay over $\QQ$.
\end{enumerate}
\end{enumerate}
\end{theorem}

The reader can verify that these results on Cohen-Macaulayness over $\QQ$ also hold for the corresponding ordered versions (under the given relations between the dimension and the size of the field).

\medskip

On the other hand, results by Vogtmann \cite{Vog} show that the complex of orthonormal bases of an orthogonal space (under suitable conditions on the ground field) is at least ``$n/3$"-connected.
In fact, Vogtmann's proof can be extended to our context by using our definition of the partial basis complex.
Here, for an atom $S$ of $\Sub(V,\Psi)$ we take as $P_S$ in the definition of $\Vog((V,\Psi),\operp,P)$ a suitable set of bases of $S$.
For example, if $\Psi$ is unitary then $S$ has dimension $1$, and we can take the set of those vectors $v\in S$ such that $\Psi(v,v) = 1$.
Similar considerations can be made for the orthogonal case, but note that there may not exist vectors of norm $1$.
In such cases we can fix a non-square scalar $d$ of the ground field and take those vectors $v\in S$ with $\Psi(v,v) = d$.
In the symplectic case, we can consider for example symplectic bases, that is, pairs $(v,w)$ such that $\Psi(v,w) = 1$.

\medskip

We close this subsection with some comments on the Euler characteristics of the different decomposition posets for unitary and symplectic spaces over finite fields.

If $\Psi$ is unitary or symplectic then $\D(V,\Psi)$ is an exponential structure
in the sense of Stanley \cite{Sta1}. Hence, generating function techniques,
analogous to those applied in \cite{Welker} which yield  \ref{eq:eulerCharDecompVS},
can also be
used to study the reduced Euler characteristic of $\redm{\D(V,\Psi)}$. Therefore one can derive the following results.
For $\KK = \FF_{q^2}$ and $\Psi$ unitary on $V = \FF_{q^ 2}^ n$, the reduced Euler characteristic of $\redm{\D(V,\Psi)}$ is
\begin{equation}
\label{eq:unitaryeulerCharDecompVS}
 \frac{(-1)^{\frac{n(n+1)}{2}}}{n} \cdot \prod_{i=1}^{n-1} ((-q)^i-1) \cdot f_n(-q),
\end{equation}
\noindent
where $f_n(q)$ is the same polynomial as in \ref{eq:eulerCharDecompVS}. Indeed 
\ref{eq:unitaryeulerCharDecompVS} is the evaluation of \ref{eq:eulerCharDecompVS} in $-q$,
which makes it another manifestation of Ennola duality. 

For $\KK = \FF_{q}$ and $\Psi$ symplectic on $V = \FF_{q}^{2n}$, the reduced Euler characteristic of $\redm{\D(V,\Psi)}$ is
\begin{equation}
\label{eq:symplecticeulerCharDecompVS}
 \frac{(-1)^n}{n} \cdot \prod_{i=1}^{n-1} (q^{2i}-1) \cdot f_n(q^2).
\end{equation}
\noindent
Again, $f_n(q)$ is the polynomial from \ref{eq:eulerCharDecompVS}.

We expect similar behavior in case $\Psi$ is
a non-degenerate bilinear form, or even a quadratic form. Since in this case there is no unique form up to isometry,
$\redm{\D(V,\Psi)}$ does not have to be an exponential structure and the calculations above need to be amended. 

\begin{question}
    Let $V$ be a finite-dimensional vector space over
    a finite field and $\Psi$ a non-degenerate
    quadratic form on $V$. What is the reduced Euler characteristic of $\redm{\D(V,\Psi)}$?
\end{question}

\subsection{Matroids}
\label{sub:matroids}

Matroids are a natural generalization of vector spaces. Nevertheless, as we will see, due to a lack of
a categorical setting suiting our approach, we have to proceed differently.

Recall (see e.g., \cite{Bjo92}), that a matroid is a pair $(M,I(M))$, where $I(M)$ is a pure finite-dimensional 
simplicial complex on the set $M$ such that if $\sigma,\tau \in I(M)$ and $|\tau|<|\sigma|$ then there exists $y\in \sigma$ such that $\tau\cup\{y\}\in I(M)$ (this is usually referred as the \textit{exchange property}). The simplices of $I(M)$ are called independent sets, and the rank of $M$ is the size dimension of $I(M)$, that is, the size of a maximal independent set.
We will later recognize $I(M)$ as the
complex of partial bases of the matroid.
Note that in our definition of a matroid, we allow infinite sets $M$. Even though this
definition for infinite $M$ may not be a satisfactory concept from the point of view of matroid
theory, it still allows the construction of interesting posets and many of our results hold in this
generality. 

A subset $\sigma\subseteq M$ is called dependent if $\sigma\notin I(M)$.
Two elements $x,y\in M$ are parallel if $\{x,y\}$ is a dependent set.
A loop is an element $x\in M$ such that $\{x\}$ is not an independent set.

The rank $\rk(F)$ of a subset $F \subseteq M$ is the maximum size of an independent set contained in $F$.
A flat of $M$ is a subset $F\subseteq M$ such that for all $x\in M\setminus F$ 
we have $\rk(F \cup \{x\}) > \rk(F)$. 
In particular, the set of loops, which we denote by $0$, is a flat, and also $M$ is a flat.
The poset of flats, denoted by $\L(M)$, is the set of all flats of $M$ ordered by inclusion.
This is a geometric and hence graded lattice with minimal element $0$, maximal element $M$ and rank function $\rk$. 
Moreover, for flats $F$ and $G$ the rank function satisfies
\[ \rk(F\cap G) + \rk(F\vee G) \leq \rk(F) + \rk(G).\]
This rank inequality is known as the upper semimodular condition (see \cite{Sta2}). Note that the rank of $M$ coincides with the value $\rk(M)$ of the rank function, which is the height of $\L(M)$.

We are interested in analyzing the posets and simplicial complexes
arising from $\L(M)$ and its decomposition poset $\D(\L(M))$.
For the sake of brevity, we will write $\D(M)$ for $\D(\L(M))$ and
$\PD(M)$ for $\PD(\L(M))$.
If $M$ is a finite matroid, then $\L(M)$ is well-known to be shellable
and hence Cohen-Macaulay (see \cite{Bjo92}).
We are not aware of a proof of a similar result in the case $M$ is an infinite matroid of finite rank.
The proof for infinite matroids will be provided in an upcoming work \cite{PW25}.

There is some literature on the category of matroids (see for example \cite{HP}). 
Unfortunately, these results do not allow us to consider matroids as an 
ISM-category. Indeed, there is no natural monoidal product of matroids which, when restricted to disjoint flats of a fixed matroid, yields the 
matroid of the flat which is the join. For example, rank one flats from 
different projective planes are indistinguishable, but the join of two flats of rank one already distinguishes the planes.
Consequently, for matroids, we have to resort to the combinatorial construction of $\D(M)$. 
Nevertheless, since $\L(M)$ is a complemented lattice of finite height, by \ref{ex:latticeAsCategoryApproach}, if we regard $\L(M)$ as a poset category with $\sqcup = \vee$, then $\L(M) = \Sub(M,\vee)$.
In particular, we can apply all our results.

In \cite{Welker}, the lattice $\D(M)$ and related structures associated to rank selections 
of $\L(M)$ are studied. 
Recall first that if $M$ is a matroid of rank $n$ with rank function $\rk$, then a collection of flats $F_1,\ldots, F_k$ is a partial direct sum decomposition if and only if the sum of their ranks equals the rank of their join $F_1\vee\cdots \vee F_j$.
It is not hard to see that partial direct sum decompositions in matroids always extend to direct sum decompositions, and the latter are exactly the decomposition as in \ref{def:decompositionsAndPD}.
Thus the poset of (partial) direct sum decompositions of a matroid, ordered by refinement, is exactly the poset of (partial) decompositions of the lattice of flats.
Therefore, $\D(M) = \D(M,\vee)$ and $\PD(M) = \PD(M,\vee)$.
Moreover, if $M$ has rank $n$, by \ref{prop:heightDandPD} we see that $\D(M)$ and $\PD(M)$ have height $n-1$ and $2n-1$ respectively.
It is shown in \cite{Welker} that, for modularly complemented (finite) matroids, 
$\D(M)$ is shellable and hence homotopically Cohen-Macaulay.

For matroids in which all flats of the same rank define isomorphic matroids, $\D(M)$ is an exponential structure in the sense of Stanley \cite{Sta1}.
This occurs, for example, when $M$ is the matroid defined by the projective geometry derived from a finite vector space of dimension $n$ over a field with $q$ elements.
Note that $M$ is also modularly complemented, so $\D(M)$ is shellable.
In fact, the Euler characteristic of $\D(M)$ is computed in \cite{Welker} by using the exponential structure machinery from \cite{Sta1} (see Eq. \ref{eq:eulerCharDecompVS}).
Although this case was already treated from a different perspective in \ref{sub:vectorspaces}, we see that the shellability of $\D(M)$ is obtained from purely theoretical matroid properties.

Another case for which $\D(M)$ is an exponential structure and $M$ is modularly complemented is when $\L(M)$ is the Boolean lattice of subsets of an $n$-element set.
Here, $\D(M)$ is the partition lattice $\Pi_n$, and, in particular, it has a unique minimal element.
As a consequence of \ref{coro:minElementD}, we then get the following well-known fact.

\begin{corollary}
    The proper part of the poset of partial partitions of $[n]$ is contractible and Cohen-Macaulay.
\end{corollary}

\begin{remark}
\label{rk:partitionsByInclusion}
As mentioned in \ref{rk:severalOrders}, there is an alternative ordering in $\PD(M)$, namely, the contention ordering, which makes $\PD(M)$ a simplicial complex.
A recent work \cite{Gottstein} shows that the simplicial complex of partial partitions of $[n]$, denoted by $D_n$, is non-pure shellable.
Indeed, $D_n$ has non-zero homology in different degrees.
See Theorem 4.4 of \cite{Gottstein}.
\end{remark}

Recall that if there are no parallel elements in the matroid $M$ then $\L(M)$ is the Boolean lattice on
$n$ elements if and only if $M$ is the uniform matroid $U_{n,n}$. In general, for 
$1 \leq k \leq n$, 
the uniform matroid $U_{n,k}$ is the matroid on the set $[n]$ whose independent sets are the subsets of size $\leq k$.
Then $U_{n,k}$ has rank $k$ and $\L(U_{n,k})$ is the lattice of all subsets of $[n]$ of size $n$ or $< k$, from which it follows that $\D(U_{n,k})$ 
is an exponential structure.
Assume $k \leq n-1$. In general, $\PD(U_{n,k})$ consists of the trivial partition $[n]$ and 
partial partitions of $[n]$ with block sizes
$< k$ and sum of block sizes $\leq k$. We can map such a partial partition to the 
partition with only singleton blocks and one singleton block for each element appearing
in the partition. This map is a downward closure operator on $\red{\PD(U_{n,k})}$ whose image is isomorphic to the poset of non-empty subsets of $[n]$ of size $\leq k$.
This is exactly the proper part of the lattice of flats of $U_{n,k+1}$, which is indeed the $(k-1)$-skeleton of the $n$-simplex.
Thus, $\red{\PD(U_{n,k})}$ has homology concentrated in dimension $k-1$ and its $(k-1)$th homology group is free of rank $\binom{n-1}{k}$. This already shows that for $2 < k \leq n-1$ the poset $\red{\PD(U_{n,k})}$ is not spherical.
Similarly, the downward closure operator restricts to $\redm{\D(U_{n,k})}$ with image an antichain of $\binom{n}{k}$ elements.
This shows that again for $2 < k \leq n-1$ the poset $\D(U_{n,k})$ is not Cohen-Macaulay and $\redm{\D(U_{n,k})}$ is not spherical.

If $\D(M)$ is not an exponential structure, the computation of the 
Euler characteristic becomes even more challenging. 
This is the case for the partition lattice.
Since the lattice of partitions $\Pi_n$ is itself a geometric lattice, there is a
matroid $M$ with $\L(M) = \Pi_n$. 
Note that this matroid has rank $n-1$ and the rank of a flat (=partition) $\pi$ is $n - |\pi|$.
A decomposition of this matroid is, therefore, a set $\{ \pi_1,\ldots,\pi_k\}$ of partitions $\pi_j$ such that $$\displaystyle{\sum_{j=1}^k} \big(\,n - |\pi_j|\,\big) = n-1$$ and $\pi_1\vee\cdots\vee \pi_k$ is the coarsest partition.
Since $\Pi_n$ is modularly complemented, by \cite{Welker} we see that $\D(\Pi_n)$ is Cohen-Macaulay and hence $\redm{\D(\Pi_n)}$ spherical of dimension $n-3$.

In the following result, we relate partial decompositions of $\Pi_n$ to hyperforests, and compute the Euler characteristic of $\redm{\D(\Pi_n)}$ and $\red{\PD(\Pi_n)}$.
Denote by $\HF_n$ the poset
of hyperforests on an $n$-element set ordered by refinement (see \cite{Bacher}), and by 
$\HT_n$ the subposet of $\HF_n$ consisting of hypertrees (see \cite{McCM}). 

\begin{proposition}
\label{prop:partitioncase}
Let $M$ be a matroid such that $\L(M) = \Pi_n$. Then the following holds:
\begin{enumerate}
    \item $\redm{\D(M)}$ is homotopy equivalent to the proper part of $\HT_n$.
    It follows that $\D(M)$ is Cohen-Macaulay and its Möbius number is 
    $\redm{\D(M)}$ is $(-1)^{n-1} (n-1)^{n-2}$.
    \item 
    $\red{\PD(M)}$ is homotopy equivalent to the proper part of $\HF_n$.
    In particular, the reduced Euler characteristic of $\red{\PD}_n$ is $-(n-2)!$.
\end{enumerate}
\end{proposition}

\begin{proof}
We prove first the homotopy equivalence in each case.
We show that if a partial decomposition $\sigma$ contains a partition with at least two blocks of size bigger than one, then $\red{\PD(M)}_{< \sigma}$ and $\redm{\D(M)}_{< \sigma}$ (if $\sigma$ is also a full decomposition) are contractible.
This will yield a homotopy equivalence between $\redm{\D(M)}$ (resp. $\red{\PD(M)}$) and the subposet $X_n$ (resp. $Y_n$) of (resp. partial) direct sum decompositions whose partitions have precisely one block of size bigger than one.
But these posets $X_n$ and $Y_n$ are exactly the proper parts of $\HT_n$ and $\HF_n$ respectively.

Let $\pi$ be a partition of $[n]$.
Consider the following partial decomposition:
\[\sigma_\pi = \Big\{\, \big\{\,B\,\big\} \cup \big\{\, \{x\} \tq x\in [n]\setminus B \big\} \tq B\in \pi \text{ and } |B|\geq 2\,\Big\}.\]
If $\pi$ contains just one block of size $\geq 2$, then $\sigma_\pi = \{\,\pi\,\}$.
It is also not difficult to see that the sum of the ranks of the partitions belonging to $\sigma_\pi$ equals the rank of $\pi$.
Also, we have $\pi = \bigvee_{\rho \in \sigma_\pi} \rho$, since the latter is the smallest partition that contains all the blocks $B\in \pi$ of size $\geq 2$.
Therefore, $\sigma_\pi$ is a direct sum decomposition of $\pi$.

Now, let $\tau = \{\pi_1,\ldots,\pi_r\} \in \red{\PD(M)}$ be a partial direct sum decomposition.
Define the map $f(\tau) = \bigcup_{i=1}^r \sigma_{\pi_i}$.
We claim that $f(\tau)$ is also a partial direct sum decomposition:
\[\bigvee_{\rho\in f(\tau)} \rho =  \bigvee_{i=1}^r \bigvee_{\rho\in \sigma_{\pi_i}}\rho = \bigvee_{i=1}^r \pi_i,\]
and this implies that the sum of the ranks must be correct:
\[\rk\left( \,\bigvee_{i=1}^r \pi_i\, \right) = \sum_{i=1}^r \rk(\pi_i) = \sum_{i=1}^r 
\sum_{\rho\in \sigma_{\pi_i}} \rk(\rho) \geq \sum_{\rho\in f(\tau)} \rk(\rho) \geq \rk\left(\bigvee_{\rho\in f(\tau)} \rho\right) = \rk\left(\,\bigvee_{i=1}^r \pi_i\,\right).\]
Thus $f:\red{\PD(M)} \to Y_n$ is a well-defined and surjective function that restricts to a surjective map $\redm{\D(M)} \to X_n$.
Finally, it is straightforward to verify that if $\rho$ is finer than $\pi$ then $\sigma_\rho\leq \sigma_\pi$, and so $f$ is an order-preserving map.
Since we have $f(\tau) \leq \tau$ and $f$ is the identity on $Y_n$, we conclude that $f$ is a homotopy equivalence.
Moreover, this also shows that the lower intervals ${\red{\PD(M)}}_{< \tau}$ and ${\redm{\D(M)}}_{< \tau}$ (if $\tau$ is a decomposition) are contractible for $\tau\in \red{\PD(M)} \setminus Y_n$ since for $\tau' < \tau$ we have $\tau' \geq f(\tau') \leq f(\tau) < \tau $.

We have already seen that $\D(M)$ is Cohen-Macaulay by \cite{Welker}.
The value of the Euler characteristic of $\redm{\D(M)}$ follows from the homotopy equivalence with the proper part of the hypertree lattice $\HT_n$ and Theorem 1.1 of \cite{McCM}.
For hyperforests, this follows from Proposition 19 of \cite{Bacher}.
\end{proof}

Concerning the ordered versions, we conclude from item 1 in \ref{thm:orderedVersions} that $\redm{\Ord\D(M)}$ is Cohen-Macaulay of dimension $n-3$ if $\L(M) = \Pi_n$.
In \cite{PSW} we show that the M\"obius number in this case is
$(-1)^{n-1} \, (2n-1)^{n-2}$. In the same work we also
show that the M\"obius number of $\red{\Ord\PD(M)}$ is $-(n-1)! \cdot C_{n-1}$, where $C_{n-1}$ is
the $(n-1)$st Catalan number.

On the other hand, we do not have information on the sphericity of $\red{\PD(M)}$ and $\red{\Ord\PD(M)}$.
Preliminary computations show that these posets should be Cohen-Macaulay (and even shellable).
So we propose the following question:

\begin{question}
\label{q:PDOPDPin}
Let $M$ be a matroid such that $\L(M)$ is isomorphic to $\Pi_n$, $n \geq 2$.
Are $\PD(M)$ and $\Ord\PD(M)$ Cohen-Macaulay?
\end{question}

When $M$ is the matroid such that $\L(M)$ is the Boolean lattice, then we have already 
used the well-known fact that $\redm{\Ord\D(M)}$ is the face poset of the permutohedron 
and hence its order complex is homeomorphic to a sphere. The latter is also a consequence of
\ref{lm:isoGMap}.
More interesting is $\red{\Ord\PD(M)}$.
By \ref{coro:eulerCharOPDUniqueCompl},
its Euler characteristic is $-1$. Using item 2 of \ref{thm:orderedVersions} and the fact that
the interval $\red{\PD(M)}_{< \sigma}$ below a full decomposition $\sigma \in \red{\D(M)}$ is contractible, it even 
follows that $\red{\Ord\PD(M)}$ is homotopy equivalent to a sphere arising from the wedge-term in item 2 of \ref{thm:orderedVersions} corresponding to the unique minimal element of $\D(M)$.
Summarizing these results, we obtain:

\begin{proposition}
\label{prop:ODandOPDBooleanLattice}
  Let $M$ be a matroid such that $\L(M)$ is a Boolean lattice on $n$ elements.
  Then:
  \begin{enumerate}
      \item $\redm{\Ord\D(M)}$ is homeomorphic to a polytopal sphere of dimension $n-2$.
      \item $\red{\Ord\PD(M)}$ is homotopy equivalent to a sphere of dimension $2n-3$. 
  \end{enumerate}
\end{proposition}

It remains open if, in this case,  $\red{\Ord\PD(M)}$ is actually homeomorphic to a sphere. Note that there is a construction 
of a polytopal sphere based on ordered partial partitions -- the elements of $\Ord\PD(M)$ --
in \cite{HGJ}. This construction seems to be unrelated to $\Ord\PD(M)$ and its
order complex, though. It is easily seen that in general the poset $\Ord\PD(M)$ is not a lattice and hence not a face lattice of a polytope. 
Nevertheless, $\red{\Ord\PD(M)}$ is the face poset of a regular CW-complex (see for instance Proposition 3.1 in \cite{BjornerCWPoset}).

\begin{question}
  Let $M$ be a matroid such that $\L(M)$ is the Boolean lattice on $n$ elements.
  Is the geometric realization of $\red{\Ord\PD(M)}$ homeomorphic to the sphere of dimension $2n-3$?
\end{question}

Finally, we discuss the frame complex and the complex of partial bases of a matroid $M$.
Indeed, in our definition, we can regard $\Vog(M,\vee,P)$ as the initial independence complex of $M$ by letting $P_F$ be the set of parallel elements in a flat $F\in \L(M)$ of rank one.
Hence $\Vog(M,\vee,P) = I(M)$.
For a finite matroid, this is a well-studied and important object, which, as we mentioned at the beginning of this section, is well-known to be Cohen-Macaulay.
We deduce then that, for a finite matroid $M$, the frame complex $\F(M)$ is also Cohen-Macaulay.
In fact, $\F(M)$ is also the independence complex of the matroid $M'$ whose ground set is the vertex set of $\F(M)$ (and so $\L(M)$ is isomorphic to $\L(M')$).
Thus, in both cases, the ordered versions are
Cohen-Macaulay by \ref{thm:orderedVersions}.

As for $\L(M)$ in the situation when the matroid is infinite, we do not know of a suitable reference proving Cohen-Macaulayness.
In \cite{PW25}, we also show that $I(M)$ is Cohen-Macaulay of dimension $n-1$ if $M$ is an infinite matroid of finite rank $n$. In particular, we see that the frame complex $\F(M)$ is Cohen-Macaulay of dimension $n-1$ by \ref{prop:inflationDecompositionCM}.

As mentioned above, the results from \ref{prop:bergmantop} on the topology 
of the augmented Bergman complex in the (finite) matroid case are well-known
(see \cite{BKR}). 
    
\subsection{Modules over Dedekind domains}

Let $M$ be a left module over a ring $R$.
In the category of (left) $R$-modules with monoidal product equal to the coproduct, i.e., the direct sum, $\Sub(M)$ is the poset of submodules and $\Sub(M,\oplus)$ is the subposet of direct summands of $M$.
Without further requirements, $\Sub(M,\oplus)$ may have infinite height.
Hence we assume that $M$ is a finitely generated module over a Dedekind domain $R$.
Then we have a structure theorem: $M$ decomposes as the direct sum $T_M\oplus P_M$, where $T_M$ is the torsion submodule and $P_M$ is a projective module.
Moreover, $T_M$ decomposes as a finite direct sum of cyclic torsion modules.
However, $P_M$ may not be a free module, but it decomposes as the direct sum of a free module $R^{n-1}$ and a rank-$1$ projective module $I$.
Thus $M$ is a Noetherian module, and it is projective if and only if it is torsion-free.

The following lemma describes the poset $\Sub(M,\oplus)$ in terms of the torsion submodule and the projective part.

\begin{lemma}
\label{lm:directProductDedekind}
Let $R$ be a Dedekind domain and $M$ a finitely generated $R$-module.
Denote by $T_M$ the torsion submodule of $M$ and $P = M/T_M$.
If $q:M\to M/T_M$ is the quotient map, then
\[ \rho :\left\{ \begin{array}{ccc} \Sub(M,\oplus) & \to & \Sub(T_M,\oplus) \times \Sub(P,\oplus)\\
  N & \mapsto & (N\cap T_M, q(N)) \end{array} \right.\]
is an isomorphism with inverse $(T,R)\mapsto i(T)+s(R)$, where $s:P\to M$ is a section of $q$ and $i:T_M\to M$ is the inclusion map.
\end{lemma}

\begin{proof}
Let $N\in \Sub(M,\oplus)$.
Then $N$ is also finitely generated and the torsion submodule of $N$ is $T_N = N\cap T_M$.
Indeed we also have $N = T_N\oplus s(q(N))$ since the sequence $0\to T_N\to N \to q(N) \to 0$ splits with section $s|_{q(N)}$.
It is not hard to show that if $N'$ is a complement of $N$ in $M$, then $N'\cap T_M$ is a complement of $N\cap T_M$ in $T_M$ and $q(N')$ is a complement of $q(N)$ in $P$.
Hence $\rho$ is a well-defined order-preserving map, and it is an isomorphism with inverse given by $(T,R)\mapsto i(T) + s(R)$.
\end{proof}

The structure of the poset of direct summands $\Sub(T,\oplus)$ of a torsion module $T$ over a Dedekind domain is closely related to the poset of direct summands in an Abelian group.
In this article, we will not study the torsion part and thus only focus on the poset of direct summands of a projective module of finite rank.

\begin{lemma}
\label{lm:dedekindProperties}
Let $M$ be a finitely generated projective module of rank $n$ over a Dedekind domain $R$ with field of fractions $\KK$.
The following hold:
\begin{enumerate}
    \item Let $V = M\otimes_R \KK \groupiso \KK^n$. Then the map $N\mapsto N\otimes_R \KK$ gives an isomorphism $\Sub(M,\oplus) \to \Sub(V)$, where the latter is the poset of subspaces of $V$.
    \item If $N,N'\in \Sub(M,\oplus)$ then $N\cap N'\in \Sub(M,\oplus)$.
    Thus $N\wedge N' = N\cap N'$.
    \item For $P\in M$ it holds $\Sub(M,\oplus)_{\leq P} = \Sub(P,\oplus)$.
    In particular, $\Sub(M,\oplus)$ is downward $\sch$-complemented, with $\sqcup = \oplus$.
    \item Property (LI).
    \item Property (EX).
    \item Property (CM).
\end{enumerate}
\end{lemma}

\begin{proof}
Item 1 is discussed in \cite[p. 3]{Charney} and item 2 is \cite[Lemma 1.2]{Charney}.
Item 3 is a consequence of the modular law.

Next, suppose that $\sigma \in \PD(M,\oplus)$ and let $\tau_P\in \D(P,\oplus)$.
Fix $P'\in \Sub(M,\oplus)$ such that $\sigma \cup \{P'\}\in \D(M,\oplus)$. 
Then $\bigcup_{P\in \sigma} \tau_P \cup \{P'\}\in \D(M,\oplus)$.
Reciprocally, by item 3, if $\sigma \in \PD(M,\oplus)$ and $\Phi(\sigma)\leq P$, then $\Phi(\sigma)\in \Sub(P,\oplus)$ and if $\{ \Phi(\sigma), U \}\in \D(P,\oplus)$ and $\{P,P'\}\in \D(M,\oplus)$, then $\sigma \cup \{U,P'\}\in \D(M,\oplus)$.
From this, it is not hard to conclude that item 4 holds.

Property (EX) also holds by similar arguments involving the use of modular law.

Finally, we show that property (CM) holds.
Let $(P,P'), (N,N')$ two direct sum decompositions of $M$ such that $P\leq P'$ and $N\geq N'$.
Then, by item 2, $P'\cap N = P'\wedge N \in \Sub(M,\oplus)$.
Also ,the submodules $\{P, P'\cap N, N'\}$ are clearly in direct sum, and their span is $M$.
Thus $\{P,P'\cap N, N'\}\setminus\{0\}\in \D(M,\oplus)$.
\end{proof}

In \cite{Charney}, Charney proved that if $R$ is a Dedekind domain and $M$ is a projective module of rank $n$, then $\Charney(M,\oplus)$ is spherical of dimension $n-2$.
By \ref{prop:charney} and \ref{lm:dedekindProperties}, we see that $\Delta \Charney(M,\oplus) \cong {\Ord\redm{\D(M,\oplus)}}^{\op}$.
Now, by \ref{lm:intervalsOrdered} and \ref{lm:dedekindProperties}, the intervals $\Ord\redm{\D(M,\oplus)}_{< z}$ and $\Ord\redm{\D(M,\oplus)}_{> z}$ are spherical of the correct dimension for $z\in \Ord\D(M,\oplus)$.
Thus $\Ord\D(M,\oplus)$ is Cohen-Macaulay (so $\Charney(M,\oplus)$ is Cohen-Macaulay as well).
By \ref{thm:orderedVersions}, $\redm{\D(M,\oplus)}$ is spherical.
Since its intervals are intervals in partition lattices or products of smaller decomposition posets, we also conclude that $\D(M,\oplus)$ is Cohen-Macaulay.

Hence, by Charney's result we get:

\begin{theorem}
[Charney]
Let $M$ be a finitely generated projective module of rank $n$ over a Dedekind domain $R$.
Then $\Sub(M,\oplus)$, $\D(M,\oplus)$, $\Ord\D(M,\oplus)$ and $\Charney(M,\oplus)$ are Cohen-Macaulay posets of dimension $n$, $n-1$, $n-1$ and $n-2$ respectively.
In particular, this holds when $R$ is a PID.
\end{theorem}

On the other hand, the complex of partial bases of a projective module of finite rank over a Dedekind domain is widely studied.
But first, we need to make a few remarks on its definition.
According to the work by Van der Kallen \cite{vdK}, the poset $\U(M)$ of unimodular sequences of an $R$-module $M$ is the set of ordered sequences of distinct elements of $M$ which form a basis of a free summand of $M$, ordered in the usual way.
That is, if $z,w\in \U(M)$ then $z\leq w$ if and only if $z$ is obtained by deleting elements of the tuple $w$.

However, the definition of $\U(M)$ does not always coincide with our definition of (ordered) frames and partial bases for $M$.
To be more precise, let $\C$ be the category of $R$-modules, where $R$ is a Dedekind domain, $\sqcup= \oplus$, and let $M=R^n$ be the free $R$-module of rank $n\geq 2$.
Then there are frames $\{N_1,\ldots,N_n\}$ of $M$ such that not all the $N_i$ are free, but just projective of rank one.
Indeed, this is the case if $R$ is not a PID.
This leads to the question of which definition we should also employ for the complex of partial bases.
For example, a natural choice for the partial basis complex is to take, for each rank-$1$ summand $N$ of $M$, the set of generating sets of $N$ of minimal possible size.
In the case $N$ is free, this is exactly the set of bases of $N$.
But if $N$ is non-free, these generating sets have size $2$.
We denote our ``natural" choice for the partial basis complex by $\Vog(M,\oplus)$ (note that $\PF(M,\oplus) = \F(M,\oplus)$).
When $R$ is a PID, $\Ord\Vog(M,\oplus) = \U(M)$.

If we work instead in the category of free $R$-modules with $\sqcup = \oplus$, then a rank-$1$ summand of $M = R^n$ is always free and therefore $\U(M)$ coincides with the ordered version of the complex of partial bases as we defined.
To avoid confusion on these different approaches and definitions, we write $\Vog(M)$, $\F(M)$, etc., for the versions of these posets and complexes that only involve decompositions into free summands (so working in the category of free $R$-modules).
We reserve notation with $\oplus$, namely $\Vog(M,\oplus)$, $\F(M,\oplus)$, etc., for the corresponding versions in the category of $R$-modules (so decompositions might involve non-free projective summands).
With this notation 
\[ \Ord\Vog(M) = \U(M) \subsetneq \Ord\Vog(M,\oplus). \]

We recall now some results on these posets and complexes.
By Theorem 2.6 of \cite{vdK}, $\U(R^n)$ has dimension $n-1$ and it is $(n-3)$-connected.
This result settled a conjecture raised by Quillen on higher connectivity of the posets of unimodular sequences on finite-dimensional Noetherian rings.
Moreover, by \ref{thm:orderedVersions}, the unordered version of this poset, i.e., the complex $\B(R^n)$, is $(n-3)$-connected.
In particular, $\F(R^n)$ is $(n-3)$-connected by \ref{prop:inflationDecompositionCM}.

However, it was more recently shown \cite[Theorem 2.1]{CFP} that $\U(R^n)$ might not be spherical, and hence the family of complexes $\U(R^n)$ are not Cohen-Macaulay in general (and thus neither are $\F(R^n)$ by \ref{prop:inflationDecompositionCM}).
For example, for $R = \ZZ[\sqrt{-5}]$ and $n=2$, the complex $\U(R^2)$ is disconnected.
Nevertheless, Theorem E of \cite{CFP} establishes the Cohen-Macaulay property for the complex $\U(R_S^n)$ if $R_S$ is a Dedekind domain of arithmetic type constructed from a global field or number field $K$, and $S$ is a set of places of $K$ of size at least two (and infinite if $K$ is a number field), which also contains a non-complex place.
In particular, its unordered version is Cohen-Macaulay by \ref{thm:orderedVersions}.
Also by \ref{prop:inflationDecompositionCM}, $\F(R_S^n)$ is Cohen-Macaulay.

We propose the study of the partial decompositions in this context:

\begin{question}
For which Dedekind domains $R$ are $\red{\PD(R^n)}$ and $\red{\PD(R^n,\oplus)}$ spherical?
\end{question}

Note that in the example above for $R = \ZZ[\sqrt{-5}]$, $\red{\PD(R^2)} = \F(R^2)$ is disconnected since $\Vog(R^2) = \U(R^2)$ is.
But also $\red{\PD(R^2,\oplus)} = \F(R^2,\oplus)$ is disconnected, as we show below.

\begin{lemma}
Let $R$ be a Dedekind domain.
If $\F(R^2)$ is disconnected, then so is $\F(R^2,\oplus)$.
\end{lemma}

\begin{proof}
We prove first that for a frame $\{N_1,N_2\}\in \F(R^2,\oplus)$, $N_1$ is free if and only if $N_2$ is.
Indeed, the $N_i$ are rank-$1$ projective $R$-modules, and hence can be regarded as ideals of $R$.
To be more precise, if $R^2 = N_1\oplus N_2$ then, regarding $N_1,N_2$ as ideals, in the class group we have that $ 1 = [R]^2 = [R^2] = [N_1][N_2]$, which means that $N_1$ is the inverse of $N_2$.
Now note that $N_1$ is free if and only if it is a principal ideal, that is, $[N_1] =1$.
Thus $N_1$ is free if and only if $N_2$ is.

On the other hand, a frame $\{N_1,N_2\}\in \F(R^2,\oplus)$ contains either two rank-$1$ free modules or two non-free rank-$1$ projective modules.
Thus, two frames of free rank-$1$ submodules that are in different connected components in $\F(R^2)$ cannot be connected by using non-free rank-$1$ projective summands in $\F(R^2,\oplus)$.
This shows that $\F(R^2,\oplus)$ is disconnected if $\F(R^2)$ is.
\end{proof}

On the other hand, by Charney's results and \ref{thm:highConnectivity}, we can conclude that $\red{\PD(M,\oplus)}$ is at least $(n-3)$-connected for a rank $n$ projective module $M$.


\end{document}